\documentclass{article}
\usepackage[utf8]{inputenc}
\usepackage{graphicx} % Required for inserting images
\usepackage{amsmath,amsfonts,amssymb}
\usepackage{url}
\usepackage{amsthm}
\usepackage{mathrsfs}\usepackage{bbm}
\usepackage{dsfont}
\usepackage{color}
\usepackage{comment}
\usepackage[top=1in,bottom=1in,left=1in,right=1in]{geometry}
\usepackage{authblk}

\newcommand{\pa}[1]{\left(#1\right)}
\newcommand{\ac}[1]{\left\{#1\right\}}

\usepackage[backend=biber,doi=false,isbn=false,maxbibnames=99,url=false,style=alphabetic]{biblatex}
\usepackage{hyperref}
\usepackage{caption}

\newtheorem{theorem}{Theorem}
\newtheorem{conjecture}{Conjecture}
\newtheorem{corollary}{Corollary}
\newtheorem{proposition}{Proposition}
\newtheorem{lemma}{Lemma}
\newtheorem{remark}{Remark}

\title{Low-degree Lower bounds for clustering in moderate dimension}
\author[1]{Alexandra Carpentier}
\author[2]{Nicolas Verzelen}

\affil[1]{\textit{\scriptsize{Institut für Mathematik, Universität Potsdam, Germany.}}}
\affil[2]{\textit{\scriptsize{INRAE, MISTEA, Univ. Montpellier, Montpellier, France}} }

\begin{document}
\maketitle

\begin{abstract}
We study the fundamental problem of clustering $n$ points into $K$ groups drawn from a mixture of isotropic Gaussians in $\mathbb{R}^d$. Specifically, we investigate the requisite minimal distance $\Delta$ between mean vectors to partially recover the underlying partition. While the minimax-optimal threshold for $\Delta$ is well-established, a significant gap exists between this information-theoretic limit and the performance of known polynomial-time procedures. Although this gap was recently characterized in the high-dimensional regime ($n \leq dK$), it remains largely unexplored in the moderate-dimensional regime ($n \geq dK$). In this manuscript, we address this regime by establishing a new low-degree polynomial lower bound for the moderate-dimensional case when $d \geq K$. We show that while the difficulty of clustering for $n \leq dK$ is primarily driven by dimension reduction and spectral methods, the moderate-dimensional regime involves more delicate phenomena leading to a "non-parametric rate". We provide a novel non-spectral algorithm matching this rate, shedding new light on the computational limits of the clustering problem in moderate dimension.
%We consider the basic problem of clustering $n$ points into $K$ groups when those are sampled from a mixture of isotropic Gaussian in dimension $d$. Defining $\Delta$  as the minimal distance between the means, we investigate condition on $\Delta$ to recover the underlying partition. While the minimax-optimal value for $\Delta$ is well known,  there is a large gap between this threshold and the best known performances of polynomial-time procedures. Whereas this gap has been recently characterized in the high-dimensional regime ($n\leq dK$), it remains largely unknown  in the moderate-dimensional regime $(n\geq dK)$. In this manuscript, we establish a new low-degree polynomial lower bound to handle the latter in the undercomplete case $d\geq K$. While the difficulty of clustering for $n\leq dK$ is mostly driven by dimension reduction steps, the moderate-dimensional regime is much more subtle and we unveil completely different phenomenons. We also provide a matching low-degree polynomial upper bound that is tailored to the prior distribution considered in the lower bound. 
\end{abstract}

%\alex{
%TODO:
%\begin{itemize}
%    \item Faire preuve borne sup $d\geq K$?
%    \item Tableau?
%    \item Refaire une passe unification du vocabulaire?
%\end{itemize}
%}

\section{Introduction}

Gaussian mixture models are arguably the most iconic distribution model for clustering purpose. This attracted a large attention both in statistics and in machine learning
\cite{
Dasgupta99,
VEMPALA2004,
LuZhou2016,
diakonikolas2018list,
Regev2017,giraud2019partial,
fei2018hidden,chen2021hanson,Kwon20,SegolNadler2021,romanov2022,LiuLi2022,diakonikolasCOLT23b,Even24}.
%Despite this, the best possible performances of computationally efficient clustering procedures are still only partly understood. 

\paragraph{Set-up.}
For some unknown vectors $\mu_1,\ldots,\mu_K\in\mathbb{R}^d$, some  $\sigma>0$, we observe   $Y_{i}\in \mathbb{R}^d$ with $i=1,\ldots, n$ that are sampled independently with distribution
\begin{equation}\label{eq:GMM-intro}
Y_{i} \sim \mathcal{N}(\mu_{\underline{k}^*(i)},\sigma^2 I_d)\enspace ,
\end{equation}
where the function $\underline{k}^*:[n]\mapsto [K]$ encodes the unknown partition of the observation into $K$ groups. To simplify the discussion and by homogeneity, we assume throughout this paper that $\sigma=1$. We write $Y\in \mathbb{R}^{n\times d}$ for the full observation matrix. The model~\eqref{eq:GMM-intro} is an instance of the isotropic Gaussian mixture model that is conditioned to the latent assignments $\underline{k}^*$. 

For $l=1,\ldots, K$, denote $S^*_l= (\underline{k}^{*})^{-1}(\{l\})$ for the $l$-th group. The main statistical objective for such Gaussian mixture models~\eqref{eq:GMM-intro} is to cluster the data matrix $Y$, that is to recover the unknown partition 
%In the , there are two main objectives: estimating the mean parameters $\mu_l$'s  --see e.g.~\cite{diakonikolas2018list,diakonikolasCOLT23b,LiuLi2022,Kwon20} or clustering, that is reconstructing the partition 
$S^*_1,\ldots, S^*_K$ of $[n]$%--see e.g~\cite{VEMPALA2004,LuZhou2016,ndaoud2022sharp}
. Given an estimated partition $\hat{S}_1,\ldots, \hat{S}_K$ of $[n]$, consider the alignment  loss $\mathrm{err}(\hat{S}, S^*)$ defined by 
 \begin{align}\label{eq:definition_loss}
 \mathrm{err}(\hat{S}, S^*):= \min_{s \text{ permutation of [K]} }\quad  \frac{1}{2n}\sum_{k=1}^K |S_{s(k)}\triangle S^*_k|\ , 
 \end{align}
where  $\triangle$ stands here for the symmetric difference. When $\mathrm{err}(\hat{S},S^*)=0$, the partition is perfectly reconstructed. We say $\hat{S}$ achieves partial reconstruction when the error $\mathrm{err}(\hat{S}, S^*)$ is smaller than a random guess (independent of $Y$). For simplicity, we focus here on the case where the partition is balanced, i.e. when all groups  $S^*_k$'s have similar cardinality of the order of $n/K$.

The key quantity characterizing the difficulty of the clustering task is the minimal inter-group separation $\Delta_* = \inf_{k\neq k'}\|\mu_k - \mu_{k'}\|_2$, which plays the role of signal strength. Here $\|.\|_2$ stands for the $l_2$ norm in $\mathbb{R}^d$. As clustering  is combinatorial by nature, computational difficulties are central to this problem. Hence, the overarching question in the literature is to characterize the minimum separation $\Delta_*$ that allows to reconstruct the unknown partition with a polynomial-time procedure.

\subsection{State-of-the art}

\paragraph{Minimax conditions.}
Define the quantity  $\Delta_{\mathrm{stat}}$ the statistical threshold
%the square separation $\Delta_*^2$ is of larger order than the the square statistical minimax distance %larger than  $\Delta_*\gtrsim \Delta_{\mathrm{stat}}$ where
\begin{align}\label{eq:separation:minimax}
\Delta_{\mathrm{stat}}^2 := \sqrt{\frac{dK\log(K)}{n}} \lor  \log(K)\enspace .
\end{align}
If we set aside computational aspects, it has been established~\cite{Regev2017,Kwon20,Even24} that partial reconstruction of the partition is information-theoretically possible as soon as $\Delta_*\gtrsim \Delta_{\mathrm{stat}}$. Here, $u\gtrsim v$ means that $u\leq c v$ for some positive numerical constant $c$. Minimax bounds  for the low-dimensional regime ($n\geq dK$) are provided in~\cite{Regev2017}, whereas the general case is dealt with in~\cite{Even24}. The exact $k$-means algorithm, which minimizes a $l_2$-type criterion over all possible partitions into $K$ groups, is shown in~\cite{Even24} to partially recover the underlying partition with high probability as soon $\Delta_*^2$ is large compared to $\Delta_{\mathrm{stat}}^2$. However, there exists no known polynomial-time algorithm for computing this $k$-means estimator. Furthermore, in worst-case instances, this estimator is NP-hard to compute and to approximate~\cite{awasthi2015hardness}. This raises the important question whether it is possible to partially recover the underlying partition 
at the square statistical minimax distance $\Delta_{\mathrm{stat}}^2$ or whether a statistical-computational gap arises for this problem. 

%An efficient implementation of exact k-means that is polynomial-time on all datasets - namely that manages to minimise the k-means criterion in polynomial time - is however not known, and is conjectured not to exist \alex{ref?}.

\paragraph{High-dimensional regime ($n\leq dK$).}  
In an asymptotic regime where the number $K$ of groups is fixed, $n,p\to\infty$ with $p/n\to \alpha\geq \frac{1}{K^2}$, Lesieur et al.~\cite{lesieur2016phase} conjectured, using statistical physics arguments, that the problem is indeed hard below the BBP threshold~\cite{BBP05} $\Delta_*^2\asymp \sqrt{dK^2/n}$. Recently, \cite{Even24,Even25} established a low-degree polynomial lower bound stating that clustering better than random guess with $\log(n)$-degree polynomials can be impossible when $\Delta_* \lesssim_{\log} \Delta_{\mathrm{comp,HD}}$ where $\Delta_{\mathrm{comp,HD}}$ is defined through
\begin{align}\label{eq:LD_HD}
 \Delta^2_{\mathrm{comp,HD}}:= 1\bigvee \left(\sqrt{d}\wedge \sqrt{\frac{K^2d}{n}}\right)\enspace .
\end{align}
Here, $u \lesssim_{\log} v$ means that there exists two positive numerical constants $c$ and $c'$ such that $u\leq c \log^{-c'}(n) v$.
Such a lower bound provides strong evidence~\cite{SurveyWein2025} of hardness of clustering when $\Delta_*$ is of smaller order (up to logarithmic terms) than  $ \Delta_{\mathrm{comp,HD}}$. 
Note that, the lower bound~\eqref{eq:LD_HD} from~\cite{Even25} is valid both in the  high-dimensional regime ($n\leq dK$) and moderate-dimensional regime $(n\geq dK$).
On the one hand, this confirmed the conjecture of Lesieur et al.~\cite{lesieur2016phase} when the number $K$ of groups is not too large $(K\leq \sqrt{n})$. On the other hand, this unveiled another regime in the large group regime $K\geq \sqrt{n}$. Conversely, the rate $ \Delta^2_{\mathrm{comp,HD}}$ is, up to logarithmic factors, matched by a combination of hierarchical clustering (HC) for $K\geq \sqrt{n}$ and of a spectral projection step together with a HC method for $K\geq \sqrt{n}$~\cite{Even25}. The intuition is the following: If $\Delta_*^2 \gtrsim_{\log} \sqrt{d}$, then observations $Y_i$'s in the same group are closer than observations in distinct groups so that HC recovers the true partition. If $\Delta_*^2\gtrsim \sqrt{\frac{K^2d}{n}}$, then it is  possible to project $Y$ onto a $(K\wedge d)$-dimensional subspace of $\mathbb{R}^d$ corresponding to largest eigenvector. Next,  HC applied to the projected data recovers the true partition as long as $\Delta_*^2\gtrsim_{\log} \sqrt{K}$. Overall, this combination of procedures recovers the true partition as long as $\Delta_*^2\gtrsim_{\log} \Delta^2_{sp}$, with 
\begin{align}\label{eq:UB_Sp}
\Delta^2_{sp}:= \left(\left(\sqrt{\frac{dK^2}{n}}\right) \lor  \sqrt{K}\right) \land \sqrt{d}\enspace , 
\end{align}
which, in  the high-dimensional regime ($n\leq dK$), matches~\eqref{eq:LD_HD}.

\paragraph*{Moderate-dimensional regime $(n \geq dK$).} 
Despite a long series of work (e.g.~\cite{loffler2021optimality, fei2018hidden, LiuLi2022,diakonikolas2018list,Even25}), this regime remains to be understood.  On the computational lower bound side, the low-degree bound~\eqref{eq:LD_HD} still holds in this regime~\cite{Even25} provided that $d\geq \log^{5}(n)$. However, none of the procedures based on  spectral procedures as well as Lloyd's algorithm~\cite{ndaoud2022sharp,LuZhou2016}, SDP~\cite{fei2018hidden,giraud2019partial}, or HC are able to match those bounds. Indeed, among these distance-based methods, the best available guarantees are given by the threshold $\Delta^2_{sp}$
which, in the moderate-dimensional regime $n\geq dK$, simplifies as $\Delta^2_{sp}= \sqrt{K}\land \sqrt{d}$.
%One may interpret~\eqref{eq:UB_Sp} as follows: (i) when $\Delta_*^2$ is large compared to $\sqrt{d\log(n)}$, points in the same cluster are closer than points in distinct clusters. (ii) when $d\geq K$, and  $\Delta_*^2$ is large compared to $\sqrt{dK^2/n}$, then it is possibly to estimate the space spanned by the $\mu_k$ using the largest right singular values of $Y$. Then, projecting the data into this $K$-dimensional subspace, it is possible to recover the partition when $\Delta_*^2$ is large compared to $\sqrt{K\log(n)}$. 
A somewhat parallel stream of literature~\cite{hsu2013learning,diakonikolas2018list,hsu2013learning,kothari2018robust} builds upon high-order tensors built from the matrix $Y$. 
When $n\geq (dK)^c$ for some large non-explicit $c$,~\cite{LiuLi2022} proves that a polynomial-time estimator based on iterative tensor projection partially recovers the partition with high probability when $\Delta_*^2 \gtrsim \log(K)^{1+\epsilon}$, almost matching the minimax threshold~\eqref{eq:separation:minimax}.

\subsection{Open problems and our contribution}

\paragraph{Open Problems.} In the high-dimensional regime, the difficulty of clustering is mostly driven by the dimension reduction problem, at least when $K\geq \sqrt{n}$~\cite{Even25}. 
However, in the moderate-dimensional regime $(n\geq dK)$, the problem is much more subtle: even when $d\geq K$, the separation condition $\Delta_*^2\gtrsim \sqrt{dK^2/n}$ required to reduce the dimension from $d$ to $K$ is mild compared to the conditions $\Delta_*^2\gtrsim \sqrt{(d\wedge K)\log(n)}$ required by  HC algorithm~\cite{Even25,VEMPALA2004}. 
%is mild and the real challenge is to cluster points on a  $(d\wedge K)$-dimensional subspace.
%either stems from a simple distance-based methods for $(K\geq \sqrt{n})$ or from the spectral projection step for $K\leq \sqrt{n}$, as summarized in the low-degree lower bound~\eqref{eq:LD_HD} and as conjectured by Lesieur et al.~\cite{lesieur2016phase} in the fixed $K$ regime.
In particular, in our moderately large sample size $(n\geq dK)$, can we expect to recover the partition for $\Delta_*^2$ below this condition?  To illustrate this point, let us consider a specific yet emblematic example with $2K$ groups with $d=K$, and where, for $k=1,\ldots, K$, the $\mu_k$'s are orthogonal, and $\mu_{K+k}= - \mu_k$. Provided that the unknown partition is exactly balanced, the gram matrix $Y^{T}Y$ is, in expectation, proportional to the identity matrix, and any spectral procedure (or SDP) based on this Gram matrix does not provide any sensible information on the unknown means. This phenomenon is akin to other problems, as  e.g.~planted vector in a subspace~\cite{mao2025optimal} but is yet different since there there are $K$ planted vectors in a $K$-dimensional space. In light of the vacuity of spectral methods, it is tempting to rely on high-order moment of the distributions though the lens of tensor or tensor projection as done in~\cite{LiuLi2022}. Nevertheless, as the sampling size $n$ requirement in in~\cite{LiuLi2022} is implicit and is quite large, this raises the following questions:
\begin{enumerate}
    \item What is the minimum separation $\Delta_*^2$ for efficient clustering method to recover the partition in the moderate-dimensional regime?  Is it related to the spectral threshold 
    $\Delta_*^2 \gtrsim \sqrt{\frac{dK^2}{n}}$? 
    \item What are the optimal algorithms in this regime? 
\end{enumerate}

\paragraph{Main Contribution.} In this paper, we make significant advances to these two questions. Namely, our first main result is a \emph{low-degree polynomial lower bound} providing  strong evidence of the hardness of clustering for $d\geq K$ and when $\Delta_*\lesssim_{\log} \Delta_{\mathrm{comp, MD}}$ with
\begin{equation}\label{eq:comp;MD}
\Delta_{\mathrm{comp, MD}}^2 := \left(\frac{K}{n^{1/4}} \land \sqrt{K}\right) \enspace .
\end{equation}
Although this result is tailored to handle the moderate-dimensional regime, our result is valid for any sample size $n$ as long as $d\geq K$. In particular, the non-parametric rate $\Delta_{\mathrm{comp, MD}}^2$ is significantly larger than the rate $\Delta_{\mathrm{comp, HD}}^2$~\eqref{eq:LD_HD} as long as $n\geq (d\vee K)^2$. In particular, this entails that it is impossible to recover the unknown partition at the spectral threshold $\sqrt{\frac{dK^2}{n}}$ in that regime and that more delicate phenomena arise in that regime. It sheds light on the fact that the hindering to clustering is not directly related to spectral properties of $YY^{T}$ but on the behavior of higher-order quantities.

\medskip 

Our proof of the low-degree lower-bound is of constructive nature. In particular, it provides us strong insights to introduce a new clustering procedure that highly differs from spectral methods. It is somewhat related to order-4 tensor methods, although it is also of different nature. As a second main result, we establish that our new procedure is able to recover the cluster as long as $\Delta_*$ is large compared to $(1\vee \Delta_{\mathrm{comp, MD}})$ in the specific case where $d=K$ and where the parameters of the mixture follow the same prior distribution as in our low-degree lower bound --see Section~\ref{sec:main_results} for a precise definition. Overall, with these two results, we obtain that the threshold $\Delta_{\mathrm{comp, MD}}^2$ is intrinsic for computationally-efficient clustering.

\medskip

Together with low-degree polynomials lower bounds of~\cite{Even25}, our results imply a strong of evidence that, for any dimension $d\geq K$, clustering in polynomial time is not feasible as long as 
\begin{equation}\label{eq:LD_general:lower_bound}
\Delta_*^2 \lesssim_{\log}  \left[1 + \left( \sqrt{\frac{dK^2}{n}}+  \frac{K}{n^{1/4}} \right)  \wedge \sqrt{d} \right]\ .
\end{equation}
In this paper, we only provide a matching upper bound for specific instances of the model~\eqref{eq:GMM-intro}. Nevertheless, as those instances are exactly the ones used for proving the low-degree lower bound, this leads us to the following conjecture.

\begin{conjecture}\label{conject:optimal_clustering_condition}
For any $n$, $d$, $K$, such that $d\geq K$, it is possible to reconstruct the partition in polynomial time as long as
\begin{equation}\label{eq:LD_general:upper_bound}
\Delta_*^2 \gtrsim_{\log}  \left[1 + \left( \sqrt{\frac{dK^2}{n}}+  \frac{K}{n^{1/4}} \right)  \wedge \sqrt{d} \right]\ .
\end{equation}
\end{conjecture}

\subsection{Low-degree polynomial framework}\label{sec:LD}

 In the low-degree polynomial framework, we only consider estimators, or test statistics, within the class of multivariate polynomials of degree at most $D$ of the data. Its premise is that, for a large class of problems, the polynomials of degree $D=O(\log n)$ are as powerful as the best polynomial-time algorithms. Hence, proving the failure of degree $O(\log n)$ polynomials is a strong indication~\cite{KuniskyWeinBandeira,SurveyWein2025} that no polynomial-time algorithm can solve this task. This framework is closely related to other approaches  including statistical queries~\cite{brennan2020statistical}, free-energy landscapes from statistical physics~\cite{bandeira2022franz} or approximate message passing~\cite{montanari2024equivalence}.  Low-degree polynomials have been fruitfully applied to a large range of detection problems, including  community detection~\cite{Hopkins17}, spiked tensor models~\cite{Hopkins17,KuniskyWeinBandeira} among many others. Although it has originally been proposed for detection problems, it has been subsequently extended to estimation problems~\cite{SchrammWein22,SohnWein25,CGGV25}.

\medskip 

In order to fit the clustering problem within the low-degree estimation framework, we need to introduce a prior distribution on the parameters $\mu_k$'s in and $\underline{k}^*$'s of the model~\eqref{eq:GMM-intro}. We write $\mathbb{P}$ (resp. $\mathbb{E}$) for the specific marginal probability (resp. expectation) distribution of $Y$. Also, we need to reduce the clustering problem to the problem of estimating a functional $x\in \mathbb{R}$. 
This will be made explicit in the beginning of Section~\ref{sec:main_results}, but one may think of $x$ as the indicator function that the first and second rows belong to the same group, so that the problem of recovering $x$ is a sub-problem of our global clustering problem from Equation~\ref{eq:definition_loss}. 
For any multiset $S$ of $[n] \times [d]$, we write $Y^S = \prod_{(i,j) \in S} Y_{i,j}$. Given a positive integer  $D$, we write $\mathcal S_{\leq D}$ for the collection of all multisets that satisfy $|S| \leq D$. 
Following~\cite{SchrammWein22}, we define the \emph{minimum low-degree risk} as
$$\mathrm{MMSE}_{\leq D} = \inf_{f: \mathrm{deg}(f)\leq D}\mathbb{E}[(f(Y)-x)^2] = \sup_{(\alpha_S)_{S \in \mathcal S_{\leq D}}} \mathbb E\left[ \left(\sum_{S \in \mathcal S_{\leq D}} \alpha_S Y^S - x\right)^2\right]\enspace ,$$
which corresponds to the minimum square risk achievable by polynomials of degree at most $D$ with respect to $Y$ to estimate the functional $x$.  Our main objective is therefore to show that $\mathrm{MMSE}_{\leq D}\leq \mathrm{MMSE}_{\leq 0} (1+o(1))$ for degrees $D$ of the order of $\log(n)$, which entails that low-degree polynomials of degree, up to $\log(n)$, do not have a smaller error than low-degree polynomials. 
Define the \emph{low-degree polynomial correlation criterion} $\mathrm{Corr}_{\leq D}$ by
\begin{equation}\label{eq:ld_correlation}
\mathrm{Corr}_{\leq D} := \sup_{(\alpha_S)_{S \in \mathcal S_{\leq D}}} \frac{\mathbb E\left[ x\sum_{S \in \mathcal S_{\leq D}} \alpha_S Y^S\right]}{\sqrt{\mathbb E\left[\left[\sum_{S \in \mathcal S_{\leq D}} \alpha_S Y^S\right]^2\right]}}\ , 
\end{equation}
which satisfies $\mathrm{MMSE}_{\leq D}=\mathbb{E}[x^2]- \mathrm{Corr}^2_{\leq D}$ by~\cite{SchrammWein22}. Hence, the key goal is to control the low-degree correlation.

For simple detection problems of a distribution $\mathbb{P}$ against $\mathbb{Q}$ considered e.g. in~\cite{HopkinsFOCS17,KuniskyWeinBandeira}, one can bound the counterpart of the low-degree correlation~\eqref{eq:ld_correlation} by considering an orthonormal basis of the space of polynomials with respect to $\mathbb{P}$, this allows to directly express the supremum in~\eqref{eq:ld_correlation} as a sum of projections.  When $\mathbb{P}$ corresponds to the distribution of independent Bernoulli random variables or independent standard normal distributions, the canonical basis or the multivariate Hermite basis are easily shown to be orthogonal. However, for generation estimation problem as considered here, the distribution $\mathbb{P}$ is not a product distribution and no explicit orthonormal basis is known. This technical hurdle prevented the development of low-degree lower bound for estimation problems. Schramm and Wein~\cite{SchrammWein22} proposed to lower bound the denominator of~\eqref{eq:ld_correlation} using Jensen's trick, which then allows to control $\mathrm{Corr}_{\leq D}$ as a sum of multivariate cumulants.
This provided a versatile tool which has been applied among others to submatrix estimation~\cite{SchrammWein22}, stochastic block models and graphons~\cite{luo2023computational}, and Gaussian mixture models~\cite{Even25}. However, for Gaussian mixture models, bounding cumulants does not allow to improve over the spectral-type lower bound~\eqref{eq:LD_HD} from~\cite{Even25}. More recently, Sohn and Wein~\cite{SohnWein25} developed  sharper theory, but solving the corresponding overcomplete linear systems is extremely involved and we did not manage to apply this approach for our problem. 

As an alternative, Carpentier et al.~\cite{CGGV25} have advocated for a more direct approach to lower bound the risk of low-degree polynomials by constructing a family of polynomials that is almost orthogonal under $\mathbb{P}$. This has been applied to several planted problems in random graphs~\cite{CGGV25} and in particular to stochastic block models with many groups~\cite{carpentier2025phase}. One advantage is that such a basis  provides strong  intuition on the form of optimal polynomials that match the lower bound~\eqref{eq:comp;MD}. In this work, we extend and sharpen this approach  from random graph to Gaussian mixture models. In turn, this provides us insights to introduce our new clustering procedures.

\subsection{Further related literature}

\paragraph{Estimation of the means.} In this manuscript, we focus on the problem of estimating the partition $(S_1,\ldots, S_K)$ of $[n]$. Some part of the literature --see e.g.~\cite{Regev2017,diakonikolas2018list,LiuLi2022,Kwon20} rather study the problem of estimating the means $\mu_k$'s. Nevertheless, there is a close connection between between mean estimation and clustering. Indeed, in a moderate-dimensional regime $(n\geq dK)$ a consistent estimation of the $\mu_k$'s is possible when it is possible to partially reconstruct the partition~\cite{LiuLi2022}. Conversely, if the means $\mu_k$'s are well estimated, it is possible to recover the underlying partition when $\Delta_*$ is large compared to $\sqrt{\log(K)}$, the latter condition being necessary for clustering from a information-theoretical perspective~\cite{Even24}. Hence, in our regime of interest, both mean estimation and clustering turn out to be equivalent. Nevertheless, we point out that mean estimation is still possible even in low-separation regime ($\Delta_*\leq \sqrt{\log(K)}$) where clustering in not achievable. However, the estimation rates and the statistical phenomenons at stake are quite different~\cite{Regev2017,doss2020optimal}.

\paragraph{Anisotropic Gaussian mixture models.} We underline that we focus in this work on the isotropic case where the covariance matrices of the mixture is proportional to the identity matrix. In the non-isotropic case, there is an additional statistical-computational gap that does not only arise when the number of $K$ of groups is large but also because of the unknown covariance structure. In particular,  \cite{SQclustering} and \cite{pmlr-v195-diakonikolas23b}  establish such gaps through the lens of the Statistical Query lower bounds

\subsection{Organization of the manuscript}

In Section~\ref{sec:main_results}, we state our main results. 
 Section~\ref{sec:family:polynomial} is devoted to the construction of an family of polynomials indexed by multigraph and we establish its almost orthonormality. Then, we rely on that family to establish our low-degree lower bound. The main arguments for our upper bound are given in Section~\ref{sec:upper}. Section~\ref{sec:discussion} further discusses remaining results. All the proofs are postponed to the end of the manuscript and to the appendix.

\section{Model and main results}\label{sec:main_results}

\paragraph{Gaussian mixture model with orthogonal means.} For the purpose of establishing our low-degree lower bounds, we introduce a specific prior distribution on the means of the Gaussian mixture. Fix any $\Delta>0$, and any positive integer $n$, $d$, and $K$, with $d\geq K$.  
First, sample $\mu_1,\ldots, \mu_K$ independently, such that for $k=1,\ldots, K$, $\|\mu_k\|_2= \Delta$, the $\mu_k$'s are orthogonal almost surely, and their joint distribution is invariant by rotation on $\mathbb{R}^d$. Then, sample $k^*(1),\ldots, k^*(n)$ independently and uniformly at random on $[K]$. For $i=1,\ldots, n$, we also independently sample independent Rademacher variables $b_i$. Finally, given i.i.d. standard normal random variables $Z_i$, 
we observe
\begin{equation}\label{eq:model:GGM}
Y_i =  b_i \mu_{k^*(i)} + Z_i \ , \quad i=1,\ldots, n\ .
\end{equation}
In the sequel, we write $X_i= b_i\mu_{k^*(i)}$ for the signal. Observe that the matrix $Y$ obtained from~\eqref{eq:model:GGM} is, conditionally to the $\mu_k$'s, $b_i$'s, and $k^*(i)$, a Gaussian mixture model with $2K$ groups\footnote{This slightly departs from the formalism in~\eqref{eq:GMM-intro} as the number of groups is now equal to $2K$, but we keep with the abuse of notation as $K$ rather than $2K$ is the quantity of interest here.}, separation $\Delta_* = \sqrt{2}\Delta$, and such that the corresponding means are either orthogonal or opposite. The partition of $[n]$ into $2K$ groups is given by $S^*_{k,b}:= \{i\in [n]: k^*(i)=k\text{ and }b_i= b\}$ with $k\in [K]$ and $b\in \{-1,1\}$. As long as $n\geq K\log(K)$, the partition is balanced with high probability on $k^*$ and the $b_i$'s.  Henceforth, we write $\mathbb{P}$ for the marginal distribution of the matrix $Y$ according to this model. 

\medskip

Here, we choose this prior distribution on the means in such a way that, for $d=K$,  $\mathbb{E}[Y^TY|(\mu_k)]= (1+ \frac{\Delta^2}{K})I_K$ and that, at least in expectation, the spectral properties of $Y^TY$ are irrelevant to recover the partition.

%Let $\Delta >0, d,n,K \in \mathbb N$. Consider the following clustering model. We focus on the case where $K=d$, although we still use 
%$K$ and $d$ to emphasize the difference between dimension and number of groups. 
%\begin{itemize}
%\item The $\mu_1, \ldots, \mu_K$ are orthogonal almost-surely, satisfy $\|\mu_k\|_2 =\Delta^2$, and . 
%\item the $k^*(i)$ for $i \in [n]$ are taken 
%\item the $b_i$ for $i \in [n]$ are taken independently and as Rademacher random variables.
%\item for any $i$, $Y_i = X_i +Z_i$ where the $Z_{i}$ are i.i.d and follow standard normal distributions. 
%\item $X_i = b_i\mu_{k^*(i)}$. 
%\end{itemize}

\paragraph{Low-degree lower bound.}
As usual in the literature on clustering-type problems, let it be stochastic block models~\cite{SohnWein25} or Gaussian mixtures models~\cite{Even24}, we reduce clustering to the problem of estimating the \emph{functional} $x= \mathbf{1}_{k^*(1)=k^*(2)}$. It is clear that being able to recover the partition for the model~\eqref{eq:model:GGM} gives a strong indication on value of $x$. More precisely, observe that $\mathbb{P}[x=1]=1/K$ and, the risk $\mathbb{E}[(\hat{x}-x)^2]$ of an estimator $\hat{x}$ independent of $Y$ is at least $1/K(1-1/K)$. In Appendix A.2 of~\cite{Even25}, it is shown that the hardness of reconstructing $x$ within a square error smaller than $K^{-1}(1+o(1))$ implies that all polynomial-time balanced estimators $\hat{S}$ of the partition $(S_{k,b}^*)$ achieve $\mathrm{err}(\hat{S},S^*)= 1+o(1)$, i.e. that partial reconstruction is hard.  Thus, we focus in this manuscript on the problem of estimating the functional $x$.  Observe $\mathbb{E}[x^2]=1/K$ and $\mathrm{MMSE}_{\leq 0}=1/K(1-1/K)$. Recalling that $\mathrm{MMSE}_{\leq D}= \mathbb{E}[x^2]- \mathrm{Corr}^2_{\leq D}$ --see Section~\ref{sec:LD}--, we see that only have to prove $\mathrm{Corr}^2_{\leq D}$ is small compared to $1/K$ to establish that $\mathrm{MMSE}_{\leq D}\leq \mathrm{MMSE}_{\leq 0}(1+o(1))$ and thereby that any low-degree polynomial performs better than random guess.

        \begin{theorem}[Low-degree lower bound]\label{thm:LB_polynomial}
         There exist positive numerical constant $c_0>0$ and $c'_0>0$ such that the following holds for any $d\geq K$. Assume that $D\geq c'_0$ and that 
        \[
        \left(\frac{n^{1/4}\Delta^2}{K}\right)\vee \left(\frac{\Delta^2}{\sqrt{K}}\right)\vee \left(\frac{1}{K^{1/4}}\right)\leq  \left(\frac{1}{4D}\right)^{c}
        \]
        for some $c\geq c_0$. Then 
            $$\mathrm{Corr}^2_{\leq 2D} \leq \frac{1}{K^{2}} \left(1+24\left(\frac{8}{D}\right)^{c/16}\right).$$
        \end{theorem}

If we take $D= \log^{c'}(n)$, for some $c'>0$, the above theorem entails that  $\mathrm{MMSE}_{\leq D}\leq \mathrm{MMSE}_{\leq 0}(1+o(1))$ as long as $K\gtrsim \log^{c''}(n)$ and
\[
\Delta^2 \lesssim_{\log}  \frac{K}{n^{1/4}}\wedge \sqrt{K}\ . 
\]
Here, it is unavoidable to assume that $K$ is large. Indeed, there is no significant (i.e. larger than $\log(n)$) statistical-computational gap when $K\leq \log(n)$. For instance, one may observe that condition $\Delta^2 \gtrsim_{\log} \Delta^2_{sp}$ of~\eqref{eq:UB_Sp} for polynomial-time reconstruction is, up to polylog, matching the minimax condition of $\Delta^2\gtrsim \Delta_{\mathrm{stat}}$~\eqref{eq:separation:minimax} when $K\leq \log(n)$.

%Therefore, we only need to focus on invariant monomials associated to a connected 

%that have at least degree $4$, and that are connected components. Another last intuituve observation is that invariant monomials of a given degree - namely associated to templates that have a given number of edges - that are best for solving the clustering probelem are those associated templates that maximize the number of nodes - and that satisfy the constraints above. From there on, it is intuitive that such a template will be a double chain, which is what we use for our upper bound we actually do not exactly use this, but some slight modification, for technical reasons. \alex{point picture}. Proving our lower bound is then akin to formalising these intuitions, and proving that indeed the double chain is the template associated to the best possible monomial.

%\nico{Paragraphe qui suit a bouger}

%These intuitions and constructions have important consequences when we turn in Section~\ref{sec:upper} to building a polynomial that matches~\eqref{eq:comp;MD} under the model~\eqref{eq:model:GGM}. In light of the above, we only have to focus on polynomials associated to a multigraph $G$ that is connected, such that the degree of all nodes is even and at least equal to $4$. Then, technical computations of the variance also imply that, given a set of nodes, the number of edges of the multigraph has to be minimized which, together with other constraints, leads to the construction given in Figure~\ref{fig:UB}.

\paragraph{Matching upper bound.} The proof of Theorem~\ref{thm:LB_polynomial} described in Section~\ref{sec:family:polynomial} provides strong insights for constructing an efficient estimator of the functional $x$. In Section~\ref{sec:upper}, we construct an estimator $\hat{x}$  (see~\eqref{eq:definition:hat_x}) which is a transformation  of a low-degree polynomial of $Y$. This estimator is somewhat related to an order-4 tensor method, while being also different. In the following theorem $M$ and $L$ are two tuning parameters so that the degree of the corresponding polynomial is of order of $ML$.

\begin{theorem}\label{thm:final_upper_bound}
    There exists two numerical constants $c$ and $c'$ such that the following holds. Consider the estimator $\hat{x}$ defined in~\eqref{eq:definition:hat_x} where 
    we fix $M$ as smallest odd integer such that $M\geq \log(K) \vee 24$ and we fix $L= \lfloor \log(K)\rfloor$.  If we assume that $n \geq c \log^{5}(n) K$ and 
\begin{align}\label{eq:condition:final_upper_bound}
    \Delta^2 \geq c' \left[\log^{3}(n)\frac{K}{n^{1/4}} + \log^{13} (n) \right] \ , 
\end{align}
then, $\hat{x}$ defined in~\eqref{eq:definition:hat_x} satisfies $\mathbb{P}\pa{\hat x=x}\geq 1-n^{-3}$. 
\end{theorem}

As a consequence, the MMSE of $\hat{x}$ is smaller than $1/n^3$. Besides, if we apply this strategy to estimate all $n(n-1)/2$ functional $\mathbf{1}\{k^*(i)= k^*(j)\}$, we can recover the full partition of $[n]$ according to $k^*(.)$ with probability at least $1-1/n$. Given the knowledge of $k^*$, it is straightforward to recover the original partition given by $\overline{k}^*(.)$ as this simply amounts to considering $K$ problems of clustering with $2$ groups. Recovering these groups in polynomial-time is feasible as soon as  $\Delta^2\geq c\log(n)$ for a numerical constant $c>0$--see e.g.~\cite{ndaoud2022sharp}. In conclusion, under Condition~\ref{thm:final_upper_bound}, we can perfectly recover the partition $\overline{k}^*(.)$ with high probability.

\medskip 

In Theorem~\ref{thm:final_upper_bound}, the degree of the polynomial is of the order $\log^{2}(K)$. Fix any $\epsilon < 1/24$. It is in fact possible to reduce the degree of the polynomial to the order of $1/\epsilon^2$ by choosing $M$ and $L$ of the order of $1/\epsilon$. In this case, Condition~\eqref{eq:condition:final_upper_bound} is replaced by 
\[
   \Delta^2 \geq c_1 K^{c_2 \epsilon} \epsilon^{c_3}\left[\log^{5/4}(n)  \frac{K}{n^{1/4}} + \log(n)\epsilon \right] \ .
\]
This matches the low-degree lower bound up to an arbitrarily small power of $K$.

\section{Proof overview for the lower bound}\label{sec:family:polynomial}

In this section, we describe important steps for the proof of Theorem~\ref{thm:LB_polynomial}. First, we introduce an almost-orthogonal basis under $\mathbb{P}$ which is central in our arguments. After reducing the polynomial space by invariance properties, we introduce a multigraph formalism that will be used to define the polynomials. Then, we introduce a class of invariant polynomials that will be used to construct the estimator and to prove the lower bound. We prove that it is sufficient to restrict to the vector sub-space generated by these polynomials. Then, we correct this polynomial basis to establish its near orthonormality under $\mathbb{P}$.

\subsection{Invariance properties}

The following lemmas entail that it is sufficient to consider a subspace of polynomials that are rotation invariant in the $d$-dimensional feature space  and permutation-invariant on  the individuals $\{3,\ldots, n\}$.
\begin{lemma}\label{lem:reduction:rotation:permutation}
Fix any any degree $D>0$. Then, the minimum low-degree risk  $\min_{f: \mathrm{deg}(f)\leq D}\mathbb{E}[(f(Y)-x)^2]$ is achieved by a function $f$ such that 
\begin{itemize}
    \item[(i)] there exists  a polynomial $g$ of $\mathbb{R}^{n\times n}$ of degree at most $\lfloor D/2\rfloor$ such that $f(Y)= g(YY^T)$.
    \item[(ii)] $g(U)$ is invariant by permutation of the variable except one and two. In other words, for any bijection $\sigma: [n]\mapsto [n]$ such that $\sigma(1)=1$ and $\sigma(2)=2$, upon writing $U_{\sigma}= (U_{\sigma(i),\sigma_{j}})$,   we have $g(U)= g(U_{\sigma})$. 
\end{itemize}
\end{lemma}
The first property follows from the fact that the distribution of our matrix $X$ is invariant by rotation in the space of dimension $d$. The second property follows from the fact that the distribution of $X$ is invariant by permutation of the $n$ individuals. Let us denote $\mathcal{P}^{\mathrm{inv}}_{\leq D}$ the subspace of polynomials $f$ of all polynomials in $\mathbb{R}^{nd}$  with  degree at most $D$  that satisfies the properties (i) and (ii) of Lemma~\ref{lem:reduction:rotation:permutation}.

\subsection{Multigraph formalism}

In the sequel, we consider multi-graphs $G=(V,E)$ where $V=\{v_1,\ldots, v_r\}$ stands for it nodes and  $E$ stands for its multi-set of edges. Importantly, the multi-graphs $G$ are allowed to have self-loops as well as multiple edges. We write henceforth $|V|$ for its number of nodes, and $|E|$ the number of edges. In this paper, we shall only consider multigraphs and, for the sake of conciseness, we shall sometimes write graphs for multigraphs. 
The degree $\mathrm{deg}_G(i)$ of a node $i$ is the number of edges that are incident to $i$, self-edges counting twice. Two multi-graphs $G^{(1)}=(V^{(1)},E^{(1)})$ and $G^{(2)}=(V^{(2)},E^{(2)})$ are said to be equivalent if there exists a bijection between $V^{(1)}$ and $V^{(2)}$ that maps $v_1^{(1)}$ to $v_1^{(2)}$,  $v_2^{(1)}$ to $v_2^{(2)}$, and  that preserves the edges and their multiplicity. In Section~\ref{sec:multigraph}, we shall introduce this notion of equivalence through the lens of the half-edges. This will be more convenient in the proofs to work with half-edges, but we can skip this for now. 
 
\medskip

In the sequel, we define $\mathcal{G}_{\leq D}$ as a maximum collection of non-equivalent multigraphs $G=(V,E)$ with at least two nodes, such that, all the nodes to the exception of $v_1$ and $v_2$ are non-isolated, and with at most $D$ edges. Such multigraphs $G\in \mathcal{G}_{\leq D}$ are henceforth referred to as templates.

\bigskip 

Consider a multigraph $G= (V,E)$ where $V=\{v_1,v_2,\ldots, v_r\}$ with possible self edges but without isolated nodes (except possibly $1$ and $2$). Consider $\Pi_V$ the set of injective maps from  $V$ to $[n]$ such that $\pi(v_1)=1$ and $\pi(v_2)=2$. For $\pi\in \Pi_V$, we define the polynomials
\begin{equation}\label{eq:definition:P_G}
P_{G,\pi}(Y)= \prod_{(i,j)\in E} \langle Y_i,Y_{j}\rangle\ ; \quad \quad P_G = \sum_{\pi\in \Pi_V} P_{G,\pi}
\end{equation}
For short, we sometimes write $P_G$ for $P_G(Y)$ when there is no ambiguity. These polynomials $P_G$ are what we called in the introduction the {\bf invariant monomials}. 
%- although they are not monomials, but sum of monomials indexed by multigraphs that are indexed by the same multigraph $G$ - which we called in the introduction the corresponding {\bf template}.

%Finally, we remind that $\mathcal{G}_{\leq D}$ is such that any maximum collection of such multi-graphs with total degree less than $D$ and such that any two different multigraphs $G$ and $G'$ in this collection are not isomorphic. 

As a consequence of Lemma~\ref{lem:reduction:rotation:permutation},the collection $(P_{G})$ with $G\in \mathcal{G}_{\leq D}$ spans the space of invariant polynomials and that we can therefore restrict ourselves to this span for the analysis of the correlation criterion.
\begin{lemma}\label{lem:invariant:graph}
 For any $f$ in $\mathcal{P}^{\mathrm{inv}}_{\leq 2D}$, there exist numerical values $(\alpha_{G})_{G \in \mathcal G_{\leq D}}$ such that $f(Y)=\sum_{G} \alpha_G P_G(Y)$.
\end{lemma}

In the sequel, we shall reparametrize  the collection $(P_G)$ in a collection $(\widetilde{\Psi}_G)_G$ so that the $\widetilde{\Psi}_G$ are almost orthogonal.

\subsection{Towards almost orthogonal polynomials}

In the following, we will modify our basis of so-called invariant monomials into a  basis that is almost orthonormal under the distribution $\mathbb{P}$. While some of our results will remain valid for any value of $d,K$, the only case where our basis is almost orthonormal is in the case where $d=K$. In the discussion section, we will however explain how this construction remains useful for the case where $d\geq K$, through a projection into the $k$-dimensional space that contains the signal.

\paragraph{Hermite polynomials and specific correction for degree 2 nodes.} We start by recalling the standard properties of non-standardized Hermite polynomials -- see e.g. page 254 in~\cite{magnus2013formulas}.
\begin{lemma}\label{lem:hermite}
    For any integer $k>0$, we define $\psi_k$ as the Hermite polynomial of order $k$. 
    For $z\sim \mathcal{N}(0,1)$, we have $\mathbb{E}[\psi_k(z+\mu)]= \mu^k$, $\mathbb{E}[\psi_k(z)\psi_l(z)]= \mathbf{1}_{k=l}\sqrt{k!l!}$ and 
\begin{equation}\label{eq:recursion_hermite}
    \psi_k(x+y)= \sum_{l=0}^{k} \psi_l(x) \binom{k}{l}y^{k-l}
\end{equation}
\end{lemma}

Consider a multigraph $G=(V,E)$. For any edge $e\in E$, we write $(l(e),r(e))\in V^{2}$ for the respective left and right node of this edge. Given $v\in V$, $j\in [d]$, we define the modified hermite polynomial $\overline{\psi}_{{\beta_{v,j},G}}(x)$ by
\begin{equation}\label{eq:psi_bar}
\overline{\psi}_{{\beta_{v,j},G}}(x):= \left\{ \begin{array}{cc} 
    x^2 - (1+\Delta^2/K)  & \text{ if } \beta_{v,j}=2, \text{ }v\notin\{v_1,v_2\},\text{  and }\mathrm{deg}_G(v)=2\\
      \psi_{\beta_{v,j}}(x) & \text{ else }\enspace . 
\end{array}\right.
\end{equation}
Here, $\overline{\psi}_{{\beta_{v,j},G}}$ corresponds the Hermite polynomial of degree $\beta_{v,j}$ except if $\beta_{v,j}=2$ and $v$ is a degree $2$-node distinct from $v_1$ and $v_2$.

Then, given a labeling $\pi$, we define the counterpart $\overline{\Psi}_{G,\pi}$ of $P_{\pi,G}$ by introducing modified Hermite polynomials
\begin{equation}\label{eq:Hermite_graph}
 \overline{\Psi}_{G,\pi} := \sum_{(j_e, e\in E) =1}^d \overline{\psi}\left[\prod_{e\in E} Y_{\pi(l(e)),j_e}Y_{\pi(r(e)),j_e}  \right]\ , 
\end{equation}
and where $\overline{\psi}\left[\prod_{e\in E}  Y_{\pi(l(e)),j_e}Y_{\pi(r(e)),j_e}  \right]:= \prod_{i=1}^n\prod_{j=1}^d \overline{\psi}_{\beta_{\pi^{-1}(i),j},G}(Y_{i,j})$ and $\beta_{\pi^{-1}(i)j}$ is the number of times $Y_{ij}$ arises in the previous product. 

Let us comment on this definition by considering two specific cases:
\begin{itemize}
    \item 
\underline{The template $G$ does not have any  degree $2$ node.} If the multigraph $G$ does not contain any degree $2$ node aside from $v_1$ and $v_2$, then $\overline{\Psi}_{G,\pi}$ corresponds to the sum of multivariate Hermite polynomials associated to the development of $P_{G,\pi}$ in a sum of monomials in $Y$. In particular, it follows from Lemma~\ref{lem:hermite} that $\mathbb{E}[\overline{\Psi}_{G,\pi}|X]= P_{G,\pi}(X)$. In this way, $\overline{\Psi}_{G,\pi}$ is an unbiased estimator $P_{G,\pi}(X)$.

%we correct the invariant polynomial to remove the bias coming from the noise, as is classically done with the help of Hermite polynomials.

%Here note that in addition to the classical Hermite correction, we also have a special correction for nodes of degree $2$. We will discuss this later, after discussing what happens when no such node is present.

\item \underline{The template $G$ contains a degree 2 node.} What we will now discuss is only valid in the case where $d=K$. If the multigraph $G$ contains a degree $2$ node - write $v$ for such a node - then the additional correction applied to this node has a strong effect, namely that 
$$\mathbb{E}[\overline{\Psi}_{G,\pi}|\mu, k^*(\pi(v'))_{v'\in V\setminus v}] = \mathbb{E}[\overline{\Psi}_{G,\pi}|X_{[n]\setminus\{\pi(v)\}}] = 0,$$
which means that the conditional expectation of the corrected polynomial conditional to everything but the knowledge of the group of $\pi(\{v\})$ is $0$. See the proof of Lemma~\ref{lem:first_moment_degree2} for details.
This strong property is a \textit{key ingredient} in the proof of the lower bound. This hints toward the fact that degree $2$ nodes do not bring relevant information for the clustering problem. This intuition behind this property is that, the $d\times d$ matrix $\mathbb{E}[Y_{i}Y_{i}^T|\mu, k^(j)_{j\notin n}]= [1+ \frac{\Delta^2}{K}] I_d$ because the means $\mu_1,\ldots, \mu_K$ form an orthogonal basis in $\mathbb{R}^d$. This is related to the fact that the expectation of the matrix $Y^T Y$ is proportional to $I_d$ and does not carry any information on the means. Our specific construction~\eqref{eq:psi_bar} for degree $2$-nodes allows, later in the proof, to discard all templates $G$ with degree $2$ nodes. Alternatively, this correction is crucial for bypassing the spectral threshold $\sqrt{k^2d/n}$ as discussed below.

%Indeed, the largest eigenvalue of $Y^TY$, is well approximated by the power methods, by entries of $(YY^T)^D$
%In, fact and a hint toward this is that the prevous property implies that the correlation between $x$ and any polynomial $\overline{\Psi}_{G,\pi}$ with $G$ containing a node of degree $2$ is $0$ - namely $\mathbb E[x\widetilde{\Psi}_{G}]=0$ in this case.\\
%The intuition behind this is that since we have a very structured problem where we have several groups with orthogonal means, this means that the information carried by a low-degree node of degree $2$ is useless. This is related to the fact that spectral methods - which are, as discussed in the introduction, well-approximated by a power method $(YY^T)^D$ - fail. Indeed, a spectral method can be represented in our framework as the combination of several invariant polynomials $P_G$ where the most important multigraph $G$ that is used in this construction is a chain of length $D$ between $v_1$ and $v_2$. Such a multigraph $G$ is composed almost exclusively - except at its extremities - of nodes of degree $2$. What we remarked above highlights that information carried by such a multigraph is not relevant for estimating $x$ in our model.
\end{itemize}

If we had $\Delta = 0$ so that $X=0$, then $(\overline{\Psi}_{G,\pi})_{G: |E| \leq D,\pi \in \Pi_V}$ would correspond to a family of multivariate Hermite polynomials on $\mathbb{R}^{dn}$. By standard properties (see Lemma~\ref{lem:hermite}), multivariate Hermite polynomials are orthonormal under the standard normal distribution, as it is the case for $\Delta=0$. However, for $\Delta>0$, it is not the case anymore and we have to adapt the construction in a similar fashion to what was done in~\cite{CGGV25}.

%{\color{red} Ajouter ici pas mal d'explication intuitive sur ce qui se passe: ie $X=0$ la base serait orthogonale,... }

\paragraph{Correction of connected components and transformation to invariant polynomials.} Finally, we correct the polynomials in a similar fashion to~\cite{CGGV25} to better handle the correlation with polynomials associated to disconnected multigraphs. Consider a template $G\in \mathcal{G}_{\leq D}$ with $c$ connected components $(G_1,G_2,\ldots, G_c)$ that contain at least one edge, i.e. we leave aside the possible isolated nodes from these connected component. 
Given $G_l$, we define $G^*_l$ as the multigraph where we possibly add the isolated nodes $v_1$ and $v_2$ if $v_1$ or $v_2$ are not already in $G_l$. 
Then, we define
\begin{equation}\label{eq:definition_psi_tilde} 
\widetilde{\Psi}_{G} := \sum_{\pi\in \Pi_V}\widetilde{\Psi}_{G,\pi} \ ; \quad  \widetilde{\Psi}_{G,\pi}:=\prod_{l=1}^c \left[\overline{\Psi}_{G^*_l,\pi} - \mathbb{E}[\overline{\Psi}_{G^*_l,\pi} ]\right] \ . 
\end{equation}
 Note that $\mathbb{E}[\overline{\Psi}_{G^*_l,\pi} ]$ does not depend on the choice of $\pi$. Here, we use the convention of $\widetilde{\Psi}_G=1$ when $c=0$, that is when the multigraph $G\in \mathcal{G}_{\leq D}$ only contains two node and no edge.  An important property of the polynomials $\widetilde{\Psi}_{G}$ after this correction is that, when $G$ has more than one connected component, $x$ is not correlated to $\widetilde{\Psi}_{G}$, namely $\mathbb E[x\widetilde{\Psi}_{G}]=0$. 
 
 %Intuitively, the idea is that the correction that we enforced "centers" each connected component individually and that adding more connected components is therefore not bringing anything in terms of signal. This goes toward the remark that we made at the beginning, namely that it is possible to write an optimal polynomial using only those indexed by multigraphs indexed by connected components.

\paragraph{Restriction to templates where all nodes have an even degree.}

Define $\mathcal{G}_{\leq D}^{even}\subset \mathcal{G}_{\leq D}$ the subcollection of templates $G$ such that the degree of each  node is even. It turns out that the polynomials $\widetilde{\Psi}_{G}$ such that $G$ contains at least one node with odd degree do not bring any suitable for estimating $x$ in our model as stated by the following lemma. The idea behind this result is that because of the symmetrization of the distribution through the Rademacher random variables $b_i$, the polynomials $\widetilde{\Psi}_{G}$ such that $G$ contains at least one node with odd degree are not correlated with $x$ and are not correlated with any $\widetilde{\Psi}_{G'}$ where $G'$ only contains  nodes of even degree. Hence, we can restrict ourselves to the span of $\widetilde{\Psi}_{G}$ such that $G$ contains only nodes with even degree.

 Given a multi-graph $G$, we write $\mathrm{Aut}(G)$ for its group of automorphisms --see Section~\ref{sec:multigraph} for a formal definition. 
 For any template $G\in \mathcal{G}_{\leq D}$, introduce
\begin{equation}\label{eq:definition:V(G)}
\mathbb{V}(G)= |\mathrm{Aut}(G)| d^{|E|}\frac{(n-2)!}{(n-|V|)!}\enspace , 
\end{equation}
which corresponds to a variance proxy for $\widetilde{\Psi}_{G}$.

\begin{lemma}\label{lem:reduction:degree_pair}
We have 
\[
\mathrm{MMSE}_{\leq 2D}= \inf_{f(Y)\in \mathrm{Vect}(\widetilde{\Psi}_{G}, G \in \mathcal{G}_{\leq D}^{even})}\mathbb{E}[(f(Y)-x)^2]\enspace ,
\]
which is equivalent to
\begin{align*}
    \mathrm{Corr}_{\leq 2D} &= \sup_{(\alpha_G)_{G\in \mathcal G^{\mathrm{even}}_{\leq D}}} \frac{\mathbb E\left[x \cdot \sum_{G \in \mathcal G^{\mathrm{even}}_{\leq D}} \alpha_G \frac{1}{\sqrt{\mathbb V(G)}}\widetilde{\Psi}_{G}\right]}{\sqrt{\mathbb E\left[\left[\sum_{G\in \mathcal G^{\mathrm{even}}_{\leq D}} \alpha_G \frac{1}{\sqrt{\mathbb V(G)}}\widetilde{\Psi}_{G}\right]^2\right]}}\enspace .
\end{align*}
\end{lemma}

The following theorem states that, when $d=K$, the family of polynomials $\widetilde{\Psi}_{G}$ with $G\in \widetilde{\Psi}^{\mathrm{even}}_{G}$ is  almost orthogonal.
    \begin{theorem}\label{thm:Riez_constant}
        Assume that $d=K$. There exist positive numerical constant $c_0>0$ and $c'_0>0$ such that the following holds. 
        Assume that $D\geq c'_0$ and that 
        \begin{equation}    \label{eq:condition:orthogonal}
        \left(\frac{n^{1/4}\Delta^2}{K}\right)\vee \left(\frac{\Delta^2}{\sqrt{K}}\right)\vee \left(\frac{1}{K^{1/4}}\right)\leq  \left(\frac{1}{4D}\right)^{c}\ , 
        \end{equation}
        
        for some $c\geq c_0$. Define $u =  8 \left(\frac{8}{D}\right)^{c/16}$. Then, for any vector $\alpha=(\alpha_G)$, with $G\in \mathcal{G}_{\leq D}^{\mathrm{even}}$  we have 
        \begin{align}\label{eq:norm}
            (1-u) \|\alpha\|_2^2 \leq \mathbb{E}\left[\left(\sum_{G\in \mathcal{G}_{\leq D}^{\mathrm{even}}} \alpha_G \frac{\widetilde{\Psi}_G}{\sqrt{\mathbb V(G) }}\right)^2\right] \leq (1+u) \|\alpha\|_2^2\enspace .
        \end{align} 
       % where $\mathbb{E}\left[\big(\sum_{G\in \mathcal{G}_{\leq D}^{\mathrm{even}}} \alpha_G \frac{\widetilde{\Psi}_G}{\sqrt{\mathbb V(G) }}\big)_2^2\right] = \mathbb E \left[\left(\sum_{G\in \mathcal{G}_{\leq D}^{\mathrm{even}}} \alpha_G \frac{\widetilde{\Psi}_G}{\sqrt{\mathbb V(G) }} \right)^2\right]$. \alex{virer la notation norme?}\nico{oui pour cela }
        \end{theorem}

For $D\asymp \log(n)$, Condition~\eqref{eq:condition:orthogonal} for almost-orthogonality is equivalent to  $\Delta^2 \lesssim_{\log} \frac{K}{n^{1/4}}\wedge  \sqrt{K}$, which precisely corresponds to our regime of interest for the low-degree lower bound in the next theorem. 

        \medskip 

\paragraph{Proof idea.} To establish the almost orthonormality of the basis $\widetilde{\Psi}_G/\sqrt{\mathbb V(G) }$, we compute the Gram matrix associated to this basis and we show, as in~\cite{CGGV25}, that the row-sums of the absolute values of this matrix are negligible in comparison to the diagonal terms. Hence, we need to get a tight bound of the covariance terms $\mathbb{E}[\tilde{\Psi}_{G^{(1)}}\tilde{\Psi}_{G^{(2)}}]$. By developing these terms over the labeling $\pi^{(1)}\in \Pi^{V^{(1)}}$ and $\pi^{(2)}\in \Pi^{V^{(2)}}$, this mainly amounts to control quantities of the form $\mathbb{E}[\overline{\Psi}_{G^{(1)},\pi^{(1)}}\overline{\Psi}_{G^{(2)},\pi^{(2)}}]$ as a function of the two labelled multigraphs $(G^{(1)},\pi^{(1)})$ and  $(G^{(1)},\pi^{(2)})$. Because the noise structure of the matrix $YY^{T}\in \mathbb{R}^{n\times n}$ is quite different from the noise structure for a matrix $Y\in \mathbb{R}^{n\times n}$ sampled from a stochastic Block model and because we have to care about the specific corrections for degree 2 nodes, the proof arguments are quite distinct from that in~\cite{CGGV25} although the general organization is similar. The formalism for handling these moments is postponed to Section~\ref{sec:multigraph} and the proof of this theorem is given in Section~\ref{sec:proof:Riez}.

        \subsection{Low-degree lower bound}\label{sec:lower_bound}

Now, we are in position to establish Theorem~\ref{thm:LB_polynomial}.

\paragraph{Proof idea for Theorem~\ref{thm:LB_polynomial}.} By combining Lemma~\ref{lem:reduction:degree_pair} and Theorem~\ref{thm:Riez_constant}, we readily get 
 \begin{align*}
    \mathrm{Corr}^2_{\leq 2D} &= \sup_{(\alpha_G)_{G\in \mathcal G^{\mathrm{even}}_{\leq D}}} \frac{1}{1-u}\frac{\left[\mathbb  E\left(x \cdot \sum_{G \in \mathcal G^{\mathrm{even}}_{\leq D}} \alpha_G \frac{1}{\sqrt{\mathbb V(G)}}\widetilde{\Psi}_{G}\right) \right]^2}{\sum_{G \in \mathcal G^{\mathrm{even}}_{\leq D}} \alpha^2_G }\\
    &\leq \frac{1}{1-u}\sum_{G \in \mathcal G^{\mathrm{even}}_{\leq D}}\left[\mathbb  E\left( \frac{x}{\sqrt{\mathbb V(G)}}\widetilde{\Psi}_{G}\right) \right]^2\ , 
\end{align*}
by Cauchy-Schwarz inequality. As a consequence, we need to compute the moments  $\mathbb E[x \widetilde{\Psi}_{G}]$ for all $G\in \mathcal  G^{\mathrm{even}}_{\leq D}$. It follows from   our definitions that  (i) $\mathbb E[x \widetilde{\Psi}_{G}] = 0$ if some node (aside from $v_1$ and $v_2$) in $G$ has degree $2$ and (ii) $\mathbb E[x \widetilde{\Psi}_{G}] = 0$ if $G$ has more than one connected component. Thus,  we only need to focus on connected multigraphs $G$ such that the degree of all nodes except $v_1$ and $v_2$ is larger or equal to $4$ --see the lemmas in Section~\ref{sec:mom}.          The rest of the proof relies on direct bounds on these moments. See Section~\ref{sec:proof_LB}.  

    \begin{remark}[Spectral Methods]
    This lower bound supports our claim that it is not possible to recover the group  at the spectral threshold  $\Delta^2 \asymp \sqrt{\frac{K^2 d}{n}}$ in moderate-dimensional regime. Let us informally explain how we manage to bypass this threshold and how spectral procedure are connected to our multigraph formalism. 
    It is known that the largest eigenvalue of the matrix $YY^T$ is captured by power iterations of $YY^T$.
    %Alternatively, as alluded  in the introduction, spectral methods are strongly related to the evaluation of the entries of the power iterated matrix $(YY^T)^D$. 
    In particular, $(YY^T)^D_{1,2}= \sum_{i_1,\ldots, i_{D-1}}\prod_{s=1}^{D}\langle Y_{i_{s-1}},Y_{i_s}\rangle$ with $i_0=1$ and $i_D=2$ is the sum over all paths of length $D$ from $1$ to $2$. This quantity is arguably well approximated by the sum over \emph{non-overlapping paths}  $\sum_{i_1,\ldots, i_{D-1}}\prod_{s=1}^{D}\langle Y_{i_{s-1}},Y_{i_s}\rangle$ where the quantities $(i_0,i_1,\ldots, i_{D-1},i_D)$ are distinct. The latter exactly corresponds to $P_{G}(Y)$ where $G$ is the simple path of length $D$. This graph $G$ is only made of  degree $2$ nodes, except for the two extremities. In fact, the corresponding modified Hermite polynomial $\overline{\Psi}_{G}$ is uncorrelated to $x$ when $d=K$--see Lemma~\ref{lem:first_moment_degreeuneven}-- and therefore does not provide any sensible information on $x$. In summary, our choice for the prior distribution of the means was done to the make the spectrum of $YY^T$ informative about the clustering. As, with our formalism, this spectrum is related to the simple path multigraph, we crafted our family of polynomials in such a way that degree $2$ nodes correspond to polynomials that are uncorrelated to $x$.
    \end{remark}

\section{Proof overview for the upper bound} \label{sec:upper}

In this section, we introduce our estimator $\hat{x}$ and provide some intuition behind its definition and its analysis.

\subsection{Construction of the template}

 We now provide, for completeness,  a matching upper bound in the emblematic case where $d=K$. As discussed in the introduction, hierarchical clustering (HC) algorithms perfectly recovers the underlying partition when $\Delta^2\gtrsim_{\log} \sqrt{K}$. Hence, we only have to craft an estimator $\hat{x}$ that achieves a small error when $\Delta^2 \gtrsim_{\log} \frac{K}{n^{1/4}}+ d$.

 We have proved that the basis $(\overline{\Psi}_{G})$ is almost orthogonal, at least when $\Delta$ is not too large so that it satisfies~\eqref{eq:condition:orthogonal}. In light of this, we want to build a template  $\bar{G}=(\bar{V},\bar{E})$ such that 
 $\overline{\Psi}_{\bar{G}}$ tends to take different value when $x=0$ and when $x=1$. Intuitively, there is no benefit of using disconnected templates $\bar{G}$ and we focus henceforth on connected ones.

 It follows from our choice of the prior distribution and of the construction of the basis that  $\mathbb{E}[\overline{\Psi}_{\bar{G}}|x]=0$ if $\bar{G}$ contains a least a node of odd degree or at least a node of degree $2$ aside from $\bar{v}_1$ and $\bar{v}_2$ --see Lemmas~\ref{lem:first_moment_degreeuneven} and \ref{lem:first_moment_degree2} in Section~\ref{sec:mom}. As a consequence, all the nodes $\bar{v}_i$ with $i\geq 3$ should have a degree at least four, and we restrict our attention to such $\bar{G}$ in the following.

 In Remark~\ref{remark:expection:G}, we establish that $\mathbb{E}[\overline{\Psi}_{\bar{G}}|x=0]$ whereas $\mathbb{E}[ \overline{\Psi}_{\bar{G}}|x=1]= \frac{(n-2)!}{(n-|\bar V|)!}\frac{\Delta^{|\bar{E}|}}{K^{|\bar{V}|-1}}$. In order to show that $\overline{\Psi}_{\bar{G}}$ can reliably estimate $x$, we need that its variance  given $x=0$ and given $x=1$ should be small compared to $\mathbb{E}^2[ \overline{\Psi}_{\bar{G}}|x=1]$. In the proofs, we compute explicitly these conditional variances. If we denote  $\mathrm{Var}_1(\overline{\Psi}_{\bar{G}})$ the variance of conditional distribution of $\overline{\Psi}_{\bar{G}}$ given $x=1$, we show in particular in~\eqref{eq:lower_variance} that 
\[
\frac{\mathrm{Var}_1(\overline{\Psi}_{\bar{G}})}{\mathbb E^2[\overline{\Psi}_{\bar{G}}|\mu,x=1]} \geq 0.5 \frac{K}{\Delta^{2(d_{\bar{G}}(\bar v_1)+d_{\bar{G}}(\bar v_2))}}\ , 
\]
as long  as $|\bar{V}|$ is small compared to $\sqrt{n}$. Our goal is to reliably estimate $x$ as long as $\Delta^2$ is large compared $\sqrt{K}\wedge K/n^{-1/4}$. As a consequence, the right-hand-side of the above equation should be small even for $\Delta^2=K^{o(1)}$.  Hence, we need to choose $\bar{G}$ in such a way that $d_{\bar{G}}(\bar v_1)+d_{\bar{G}}(\bar v_2)$ is large. Furthermore, we also establish in \eqref{eq:lower_variance_2} that 
\begin{align} 
\frac{\mathrm{Var}_1(\overline{\Psi}_{\bar{G}})}{\mathbb E^2[\overline{\Psi}_{\bar{G}}|\mu,x=1]} \geq  \left(\frac{K^4}{n\Delta^8}\left(\frac{K}{\Delta^4}\right)^{\frac{r}{2(|\bar{V}|-2)}}\right)^{|\bar{V}|-2}  \enspace ,  
\end{align}
where $r := d_{\bar{G}}(\bar v_1)+d_{\bar{G}}(\bar v_2)+ \sum_{i=3}^{|\bar{V}|}(d_{\bar{G}}(\bar v_i)-4)$. Since we also want the-right-hand-side of the above equation to be small  
 for $\Delta^2$ small compared to $\sqrt{K}$ and $\Delta^2$ of the order of $K/n^{1/4}$,  the quantity $r$ has to be small compared to $|\bar{V}|-2$. In other words, almost all, if not all, the nodes $\bar{v}_i$ with $i\geq 3$ should be of degree $4$ and the degree $d_{\bar{G}}(\bar v_1)+d_{\bar{G}}(\bar v_2)$ should be small compared to $|\bar{V}|-2$.

 \medskip

These observations together with other constrains lead to introducing the following multigraph $G^*=(V,E)$ which corresponds to a double chain with fastener --see Figure~\ref{fig:UB}. More formally,  fix a positive integer $L$ and an odd positive integer $M$. We fix $V=\{v_1,v_2,\ldots, v_{L M+2}\}$. The multi-set $E$ of edges is defined by
    \begin{itemize}
        \item Both edges $(v_1,v_3)$ and $(v_2,v_{LM+2})$ have multiplicity $2$. 
        \item For any $0 \leq m \leq M-1$ and $1 \leq l \leq L-1$, the edge $(v_{l + mL+2}, v_{l + mL+3})$ has multiplicity $2$.
        \item For any $1\leq m \leq M-1$, the edges $(v_{mL+2},v_{mL+3})$, $(v_1,v_{mL+2})$, $(v_2, v_{mL+3})$ have multiplicity $1$. 
    \end{itemize}
  See Figure~\ref{fig:UB} for an illustration. The multigraph $G^*$ is mostly a double chain with, every $L$ node, a {\bf fastener}, namely two consecutive nodes are only  connected by one edge and that are connected to resp.~$v_1,v_2$. There are $M-1$ such fasteners. The degree of  $v_1$ and $v_2$ is equal to $M+1$, whereas the degrees of all the other nodes is equal to 4. Note that the degree of the polynomial  $\overline{\Psi}_{G^*}$ is equal to $2(LM+1)$.

      \begin{figure}
            \begin{center}
                \includegraphics[scale=1.6]{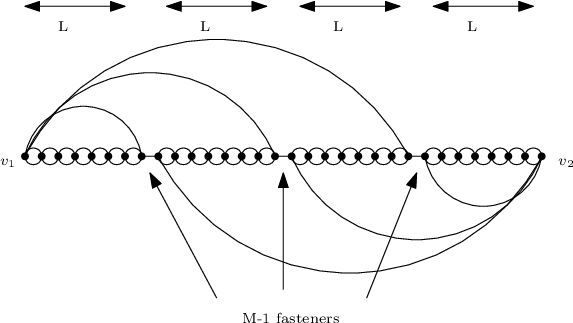}
            \end{center}
            \caption{Template $G^*$ that we use for defining the upper bound. It is mostly a double chain with, every $L$ node, a {\bf fastener}, namely two nodes connected through one edge and that are connected to resp.~$v_1,v_2$ are added. There are $M-1$ such fasteners.}\label{fig:UB}
        \end{figure}

    \medskip 

    In the next subsection, we will show that the statistic $\overline{\Psi}_{G^*}$ tends to be large when $x=1$, whereas $\overline{\Psi}_{G}^*$ is smaller when $x=0$.

    \begin{remark}[Connection with tensors] In Section~\ref{sec:lower_bound}, we have described the connections between the "simple path" graph and power iterations of the matrix $Y^T Y$.
    Similarly, our polynomial $\overline{\Psi}_{G^*}$ is approximated by a combination of power iterations of the matrix $[Y Y^T]^{(*)}$ whose general term   $[Y Y^T]^{(*)}_{ij}$ is defined by $(Y Y^T)_{ij}^2$ and multiplications by  $[Y Y^T]^{(1)}$ and  $[Y Y^T]^{(2)}$ defined by  $[Y Y^T]^{(1)}_{ij}= [Y Y^T]_{ij}[Y Y^T]_{i1}$ and $[Y Y^T]^{(2)}_{ij}= [Y Y^T]_{ij}[Y Y^T]_{2j}$. On that respect, $\overline{\Psi}_{G^*}$ has some connection to an  order-$4$ tensor method on $(Y Y^T)\otimes (YY^T)$, although we did not find a precise counterpart in the literature. 
    \end{remark}

%%More precisely, our estimator will be equal to $1$ if this polynomial is larger enough, and 0 otherwise. Note that this this template does not contain any degree $2$ node and does contain only one connected component, then $\widetilde{\Psi}_G$ is equal to $\overline{\Psi}_G$ up to an additive constant.\\
%As noted earlier, this template is, up to the fasteners, a double chain, which responds to the aim of creating a polynomial where all nodes have degree at least $4$, that posesses only one connected component, and that has as few edges as possible for a given number of nodes - matching the idea behind the "best case polynomial" for the lower bound. The presence of the fasteners is however very important as a simple double chain would not work, the fasteners being necessary to add some additional signal coming from the extremities $v_1$ and $v_2$ that correspond to individuals $1$ and $2$ and are instrumental for computing whether individuals $1$ and $2$ are in the same group.\\

\subsection{Moment bounds and MoM estimator}

For $a=0,1$, we denote
\[
\mathrm{Var}_a(\overline{\Psi}_{G^*}):= \mathbb{E}\left[\left[\overline{\Psi}_{G^*} - \mathbb E[\overline{\Psi}_{G^*}|\mu,x=a]\right]^2 \Big|\mu, k^*(1),k^*(2),b_1,b_2,x=a\right]\ , 
\]
conditional variance of $\overline{\Psi}_{G^*}$ given $\mu$, $k^*(1)$, $k^*(2)$, $b_1$, and $b_2$. It will be established in the proof of the next proposition that that this variance only depends on $a$. 

\begin{theorem}[Moments of $\overline{\Psi}_{G^*}$]\label{thm:moment:rappeur}
    Assume that $M\geq 24$. 
We have 
\[
            \mathbb{E}\left[\overline{\Psi}_{G^*}|\mu, k^*(1),k^*(2), b_1,b_2 \right]= x\frac{(n-2)!}{(n-|V|)!} \cdot \frac{\Delta^{2|E|}}{K^{|V|-2}} \enspace , %\ ; \quad \quad 
           %\mathbb{E}\left[\overline{\Psi}_{G^*}|\mu, x=0 \right]= 0\enspace . 
\] 
almost surely. 
If we further assume that $n\geq 64[4M^2L^2\vee 10^4 (M+2)^4] K$ and 
\begin{eqnarray}\label{eq:condition:signa}
\Delta^2 \geq \left(\frac{40 (M+2) K^{1+\frac{1}{2(M+1)}+\tfrac{1}{L}}}{n^{1/4}}\right)^{(M+1)/M} \bigvee \left(8(M+1)^6 K^{\tfrac{6}{M+1}}\right)\bigvee \left(128(M+1)^{12}\right)\enspace ,
\end{eqnarray}
  we have
\begin{align*}
\frac{\mathrm{Var}_0(\overline{\Psi}_{G^*})\vee \mathrm{Var}_1(\overline{\Psi}_{G^*})}{\mathbb E^2[\overline{\Psi}_{G^*}|\mu,x=1] }&\leq 10^4(M+2)^4   \left[\frac{K^{4+\tfrac{2}{M+1}+\tfrac{4}{L} }}{n\Delta^{8(1-\tfrac{1}{M+1})}}\bigvee \frac{K}{n}\right]+ \frac{K(M+1)^{2(M+1)}}{\Delta^{4(M+1)}} + 2  \frac{(M+1)^{5}}{\Delta^{2}}  +  \frac{4M^2L   ^2K}{n}\\&\leq \frac{1}{16}\enspace , 
\end{align*}
almost surely.
\end{theorem}

Write $t := \frac{(n-2)!}{(n-|V|)!} \cdot \frac{\Delta^{2|E|}}{K^{|V|-2}}$. By Chebychev inequality, is follows that, under the conditions of the above theorem, we have  
\[
\mathbb{P}\left[\overline{\Psi}_{G^*}\leq \frac{t}{2}\Big|\mu, x=1 \right]\leq \frac{1}{4}  \ ; \quad\quad\mathbb{P}\left[\overline{\Psi}_{G^*}\geq \frac{t}{2}\Big|\mu, x=0 \right]\leq \frac{1}{4}  \enspace ,  
\]
almost surely on the values of $\mu$. As a consequence, a test based on $\overline{\Psi}_{G^*}$ is able to to recover $x$ with probability at least $3/4$. However, the trivial test $\hat{x}=0$ is able to recover $x$ with probability at least $1-1/K$ and we therefore need a better guarantee.  Following~\cite{carpentier2025phase}, we add a ``Median-of-Means post-processing step'' to get  concentration bounds good enough for our purpose.

Fix $\Lambda=\lceil 24\log(n)\rceil$. Conditionally to $\mu$ and $k^*(.)$, we have $Y_i\sim \mathcal{N}(b_i\mu_k^*(i), I_d)$. We use the usual trick of  transforming $Y_i$ into a i.i.d. sample $Y_i^{(1)},\ldots, Y_i^{(\Lambda)}$ with distribution  $\mathcal{N}( \Lambda^{-1/2}b_i\mu_k^*(i), I_d)$. One way to do this is to build an orthogonal matrix $O_{\Lambda}$ of size $\Lambda$ whose first column is constant. Then, upon defining $Z'^{(1)}_{i},\ldots Z'^{(\Lambda-1)}_{i}$ as a sample of standard Gaussian random vectors, we take  $(Y_i^{(1)},\ldots, Y_i^{(\Lambda)})^T := O_{\Lambda}(Y_i,Z'^{(1)}_{i},\ldots, Z'^{(\Lambda-1)}_{i} )^T$.  Without loss of generality, we assume that $(n-2)/\Lambda$ is an integer, otherwise we can discard some of the samples.  Then, we partition $\ac{3,\ldots,n}\setminus\ac{i,j}$ into $\Lambda$ disjoint sets $J_1,\ldots J_{\Lambda}$ of cardinality $ (n-2)/\Lambda$.

Then, for $\ell=1,\ldots, \Lambda$, we define $T_{G^*}^{(\ell)}$ as the value of $\overline{\Psi}_{G^*}$ when applied to the data $(Y_1^{(\ell)}, Y_2^{\ell},  (Y_{i}^{\ell})_{i\in J_{\ell}})$
 and define $T_{G^*}$ as the median of the set $\ac{T^{(1)}_{G^*},\ldots,T^{(\Lambda)}_{G^*}}$. Finally, we consider the estimator 
\begin{align}\label{eq:definition:hat_x}
\hat x := \mathbf{1}\ac{T_{G^*}> \frac{1}{2} \frac{[(n-2)/\Lambda]!}{((n-2)/\Lambda-|V|+2)!} \cdot \frac{\Delta^{2|E|}}{\Lambda^{|E|} K^{|V|-2}} \ } \enspace .
\end{align}
In comparison to the threshold defined above, the additional factor $\Lambda^{|E|}$ accounts for the fact that $\Delta$ is replaced by $\Delta/\sqrt{\Lambda}$ in the samples $Y_i^{(\ell)}$ due to the noise inflation, and the $[(n-2)/\Lambda]!$ terms accounts for the fact that the sample size is now $2+ (n-2)/\Lambda$. Importantly, conditionally to $\mu$, $k^*(1)$, $k^*(2)$, $b_1$, and $b_2$, the random variables  $T^{(\ell)}_{G^*}$ are i.i.d.

Finally, Theorem~\ref{thm:final_upper_bound} is a consequence of the moment bounds of Theorem~\ref{thm:moment:rappeur} together with standard arguments for MoM estimators. See the proof in Section~\ref{sec:proof:UB}.

\section{Discussions}\label{sec:discussion}

       \paragraph{Extension to the case where $d\geq K$.} The previous estimator $\overline{\Psi}_{G^*}$ is tailored to the case $d = K$. Nevertheless, for general dimension $d\geq K$, when the means are sampled according to our prior distribution, we conjecture that  we can efficiently estimate $x$ as long as $\Delta^2$ is large compared to 
        \[
        \left[\left(\frac{K\sqrt{d}}{n^{1/2}}+ \left[\frac{K^{1+\epsilon} }{n^{1/4}} \land \sqrt{K}\right] \right)\land \sqrt{d} \right] \lor 1\ ,
        \]
        where $\epsilon>0$ depends on the degree $D_{\epsilon}$ of the polynomial. Note that polynomial-time clustering is known~\cite{Even25} to be possible as long as $\Delta^2$ is large compared to
        $\left[\left(\frac{K\sqrt{d}}{n^{1/2}}+ \sqrt{K} \right)\land \sqrt{d} \right] \lor 1$ and we only have to consider the regime $\Delta^2$ is large compared to $1+ \frac{K\sqrt{d}}{n^{1/2}}+ \frac{K^{1+\epsilon} }{n^{1/4}}$. Intuitively, when $\Delta^2$ is large compared to $\frac{K\sqrt{d}}{n^{1/2}}$, it is possible to project the matrix $Y$ onto the $K$-dimensional space spanned by $\mu_k$'s using a spectral method, then this boils down to studying a $K$-dimensional mixture, the twist being that the projected means of the mixture are not exactly orthogonal, so that we cannot readily use our analysis of $\overline{\Psi}_{G^*}$ in this section. As an alternative, we conjecture that relying on the template $G^{*,sp}$ described in Figure~\ref{UBadv} allows to handle this problem. The difference between $G^*$ and $G^{*,sp}$ is that each edge in $G^{*}$ is replaced by a size $N$-path where $N$ is a large constant. The effect of this size $N$ path is to mimic a spectral projection. We do not provide a formal analysis of this estimator as this is not the main focus of this paper.

%        for some arbitrary small constant $\epsilon>0$ and up to a polynomial factor of the degree $D_\epsilon$ which is depending on $\epsilon$ (increasing when $\epsilon$ decreases). Under this condition, we have that the error of our estimator will be going to $0$. This proves therefore that in the clustering setting that we consider, it is possible to characterise precisely the regimes in which clustering is computationally feasible. The reason why it is possible to extend this result to the case where $d\geq K$ is intuitive: in that case, under the above signal condition, it is possible to project in the lower dimensional space of dimension $K$ and then we just have to perform clustering in dimension $K$. To do that, a simple idea is to apply the template of Figure~\ref{fig:UB} not on $YY^T$ itself, but on a power iteration of $YY^T$ - namely $(YY^T)^N$ where $N$ is large. The template associated to this nmodification is described in : we replace all adges from the template in Figure~\ref{fig:UB} by a chain of $N$ nodes - where $N$ is a large constant. This adds, as desired, a sort of classical spectral step at every step of our procedure - through a classical power iteration method. This ensures that we "almost" project at each step on a space of dimension $K$. 
        %We provide in \alex{ref sur une section} a description of how one should modify our polynomial estimator to use it in the case where $d > K$. 
 %       We however only provide a formal proof in the case where $d=K$ and the formal theorem is therefore only stated in this case.

        \begin{figure}
            \begin{center}
                \includegraphics[scale=0.75]{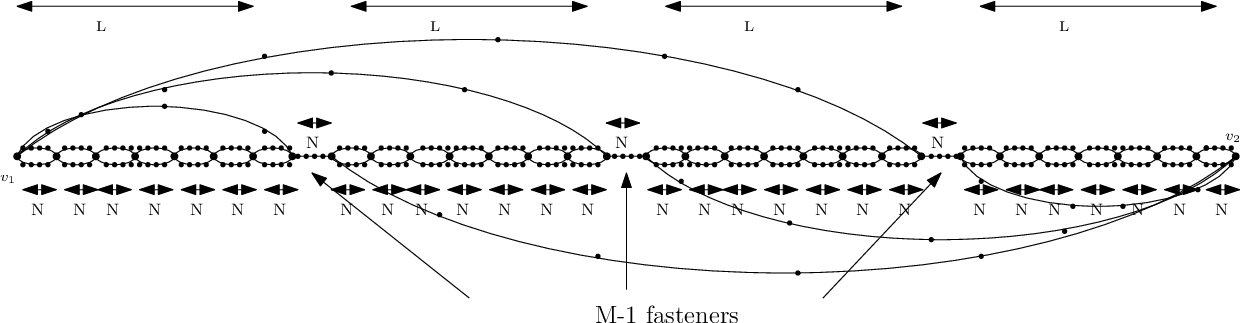}
            \end{center}
            \caption{Modification of the template $G$ in Figure~\ref{fig:UB} to accomodate $d\geq K$. We just replace each edge of $G$ by a simple chain of length $N$. }\label{UBadv}
        \end{figure}

\paragraph{Comparison with tensor Decomposition methods.} Continuing the comparison with tensor methods, we also observe that, using Hermine polynomials with respect to $Y$, we can build an unbiased estimator of the  order $4$-tensor $T$ defined $T=\frac{1}{K}\sum_{k=1}^K \mu_k\otimes \mu_k\otimes \mu_k\otimes \mu_k$.  In our model~\eqref{eq:model:GGM} with orthogonal means, estimating the means $\mu_k$'s then amounts to performing noisy orthogonal tensor decomposition from an empirical version $T_n$ of $T$. Although this approach has been proposed in the seminal work of~\cite{anandkumar2014tensor}, we are not aware any precise sample complexity bounds for our model~\ref{eq:model:GGM} as computationally efficient and robust implementation of tensor decomposition are delicate to analyze. Still, it would be interesting to analyze the Sum-of-Squares (SOS) algorithms of~\cite{ma2016polynomial} or the spectral method of \cite{hopkins2019robust} to check whether the performances matches our conditions $\Delta^2\gtrsim K/n^{1/4}$ coming from the low-degree lower bound.

    %This makes sense intuitively, as squaring the entries captures the first "higher-order" information about the clusters, which is absent in $YY^T$ due to the group orthogonality and multiplicity.
    
    %The classical spectral method based on power iterations of $YY^T$, such as $(YY^T)^D$, can be expressed as a combination of invariant monomials, where the dominant template is a single chain between nodes $v_1$ and $v_2$. However, this method fails in our setting because the single chain primarily involves nodes of degree $2$, which do not provide useful information for clustering.\\Our upper bound, on the other hand, is closely related to a double chain between $v_1$ and $v_2$, where each edge is doubled. This ensures that all nodes, except the extremities $v_1$ and $v_2$, have degree $4$, while maintaining a single connected component. This construction minimizes the number of edges for a given number of nodes under these constraints - that have been proven, in the lower bound, to be relevant. 

\paragraph{Conjecture and open problems.} As explained in the introduction, we conjecture that our low-degree polynomial lower bounds is, up to polylog, sharp for general Gaussian mixture models. To establish Conjecture~\ref{conject:optimal_clustering_condition}, we would need to introduce and analyze a polynomial-time  procedure that is able to recover the clusters as long as the square separation $\Delta_*$ satisfies
\[
\Delta_*^2 \gtrsim_{\log}  \left[1 + \left( \sqrt{\frac{dK^2}{n}}+  \frac{K}{n^{1/4}} \right)  \wedge \sqrt{d} \right]\ .
\]
In this work, we only dealt with the case where the $\mu_k$'s are either either orthogonal or opposite to each other.  We doubt that our procedure is able to deal with general $\mu_k$'s. 
Although the iterative tensor projection method of~\cite{LiuLi2022} does not allow to recover the desired condition, we believe that such an approach is a promising research direction.  Finally, our lower bound techniques are only tailored to $d\geq K$. For small dimension problems $d< K$, pinpointing the optimal separation condition for polynomial-time procedures would require  new ideas.

\section{Important multigraph notations}\label{sec:multigraph}

We first introduce some general multigraph notations as this will be central in all the proofs.

%{\color{red} Bazar de notation ici; pas toute}

%{\color{red} On essaye avec $V=\{v_1,\ldots v_{r}\}$ avec des ensembles generique.   }

\subsection{Half-edges, equivalence, and automorphisms}

Although the multigraphs are defined through a set of nodes and through a multiset of edges, it is more convenient here to work with half-edges. Consider any multigraph $G=(V,E)$. Although the multigraph is undirected, for $s=1,\ldots |E|$, we write the $s$-th edge as $(i_s,i'_s)$ where $i_s$ and $i'_s$ belong to $V$. 
Each edge $e\in E$ is defined through its extremeties $(l(e), r(e))\in V\times V$. We define the half corresponding edges as $E^{\mathrm{\mathrm{half}}}=  (\underline{e})$ where $i(\underline{e})\in V$ is the node incident to $e$ - we call them {\bf half-edges}. The degree $\mathrm{deg}_G(i)$ of a node $i$ is the number of half-edges that are incident to $i$. Given two different half-edges $\underline{e}$ and  $\underline{e}'$ in $E^{\mathrm{\mathrm{half}}}$ we write $\underline{e}\stackrel{edge}{\sim} \underline{e}'$ if $\underline{e}$ and $\underline{e}'$ come from the same edge. We say that two half-edges $\underline{e}$ and $\underline{e}'$ are incident $i(\underline{e})= i(\underline{e}')$.

We now introduce a formal definitiion of multi-graph equivalence through the lens of their half-edges. We say that $G^{(1)}$ and $G^{(2)}$ are \emph{isomorphic} and write $G^{(1)}\simeq G^{(2)}$ if there exists a bijection $\sigma: E^{\mathrm{\mathrm{half}},(1)} \mapsto E^{\mathrm{\mathrm{half}},(2)}$ such that that $\sigma$ preserves (i) the equivalence  $\stackrel{edge}{\sim}$, the (ii) the incidence relation, and (iii)  $i(\sigma(\underline{e}))=v^{(2)}_1$ (resp. $v^{(2)}_2$) when  $i(\underline{e})= v^{(1)}_1$ (resp. $v^{(1)}_2$). %If $G^{(1)}$ and $G^{(2)}$ are simple multigraphs, then this notion corresponds to classical isomorphisms between undirected multigraph. This notion also matches the standard notion of isomorphism between multigraph. 

Such a bijection $\sigma$ is called an automorphism of $G^{(1)}$ if $G^{(1)}=G^{(2)}$. We recall that  $\mathrm{Aut}(G)$ stands for the group of automorphisms of a multigraph $G$. In particular, observe that, with our definitions of automorphism, if we consider the multigraph $G$ with two nodes $v_1$ and $v_2$, only one self-edge at $v_1$, we have $|\mathrm{Aut}(G)|=2$.

\subsection{Node matchings, half-edge pairings}\label{pairmatch}

\paragraph{Node matchings.} Given two multigraphs $G^{(1)}$ and $G^{(2)}$ in $\mathcal G_{\leq D}$, we define a matching of nodes $\mathbf{M}=\{(v_{l}^{(1)},v_{l'}^{(2)}), \ldots, \}$ as a subset $V^{(1)}\times V^{(2)}$ that satisfies the three following properties
\begin{itemize} 
    \item[(a)] any element of $V^{(1)}$ (resp.  $V^{(2)}$) occurs at most once,
    \item[(b)] $(v^{(1)}_1,v^{(1)}_2)\in \mathbf{M}$, $(v^{(1)}_2,v^{(2)}_2)\in \mathbf{M}$,
    \item[(c)] all degree $2$ nodes in $V^{(1)}$  and in $V^{(2)}$ occur in  $\mathbf M$.
\end{itemize}
Henceforth, we write $\mathcal{M}$ the collection of possible matching of nodes. In what follows, we also defined $\mathcal{M}^{\star} \subset \mathcal{M}$ as the subcollection of  all node matchings that ensure that each connected components of $G^{(1)}$ and of $G^{(2)}$ has at least one node that arises in $\mathbf{M}$.

\paragraph{Half-edge pairings.} Then, given a node matching $\mathbf{M}$,  we define a pairing $\mathbf{P}$ of the half-edges as a subset of $E^{\mathrm{\mathrm{half}},(1)}\times E^{\mathrm{\mathrm{half}},(2)}$ such that (i) each half-edge in $E^{\mathrm{\mathrm{half}},(1)}$ or $E^{\mathrm{\mathrm{half}},(2)}$ occurs at most once, (ii) $(\underline{e}, \underline{e}')$ belongs to $\mathbf{P}$ implies that $(i(\underline{e}),i(\underline{e}'))$ belongs to $\mathbf{M}$. The second condition ensures that it is only possible to match half-edges that are incident to a node that has been paired.  Finally, we write $\mathcal P:=\mathcal P(\mathbf{M})$ for the collection of all half-edge pairings associated to a node matching $\mathbf{M}$. We also abreviate by $\mathcal{M}\mathcal{P}$ for the collections of all such matching of nodes and associated half-edges pairings of $G^{(1)}$ and $G^{(2)}$.  Similarly, we abreviate by $\mathcal{M}^{\star} \mathcal P$ for the set of pairing and matching associated to the matchings $\mathcal{M}^{\star}$ (instead of $\mathcal{M}$).

Define also $\mathcal{P}_{\mathrm{full}}(\mathbf{M})$ the collection of pairing $\mathbf{P} \in \mathcal P(\mathbf{M})$ such that $|\mathbf{P}|= |E^{\mathrm{\mathrm{half}},(1)}|= |E^{\mathrm{\mathrm{half}},(2)}|$ - namely such that all half-edges adjacent to nodes in $\mathbf M$ are paired. This set is in fact empty, unless $|V^{(1)}|=|V^{(2)}$, all node are matched, and all matched nodes have the same degree. 
%can be empty, e.g.~it is always empty if the number of half-edges incident to some matched node are not the same in both graphs. 

%Note that $\mathcal{P}_{\mathrm{full}}(\mathbf{M})$ is non-empty if and only if $G^{(1)} = G^{(2)}$ and $|\mathbf{M}|= |V^{(1)}|= |V^{(2)}|$. \alex{A AJUSTER - JE CROIS QUE LA CONVENYTION EST PAS LA MEME PARTOUT}

We write $\mathbf M_{\mathrm{full}}$ for the set of pairs of matched nodes - namely in $\mathbf M$ - namely that have all half-edges connected to them that are paired - namely in $\mathbf P$. We say that these nodes are {\bf fully matched}. Figure~\ref{tildeG} displays an example of matching and pairing where one pair of nodes is fully matched. Note that the corresponding nodes are isolated in the multigraph $G_{\Delta}$.

\subsection{Construction of the multigraph $G_{\Delta}$ resulting from matching and pairing}

We introduce a few notations related to the matching and pairing of $G^{(1)}$ and $G^{(2)}$ that will be instrumental in the proofs.

\paragraph{Merged multigraph $G_{\cup}[\mathbf{M};\mathbf{P}]$.} We consider two templates $G^{(1)}$ and $G^{(2)}$, a node matching $\mathbf{M}\subset V^{(1)}\times V^{(2)}\in \mathcal{M}$. Then, we consider a pairing $\mathbf{P}\subset E^{\mathrm{\mathrm{half}},(1)}\times E^{\mathrm{\mathrm{half}},(2)}$ of half-edges.

Define $G_{\cup}[\mathbf{M};\mathbf{P}]=(V,E^{\mathrm{\mathrm{half}}})$ as the multi-graph that merges $G^{(1)}$ and $G^{(2)}$ by identifying all nodes in $\mathbf{M}$. As a consequence, $G_{\cup}[\mathbf{M};\mathbf{P}]$ contains $|V^{(1)}|+|V^{(2)}|-|\mathbf{M}|$ nodes and $| E^{\mathrm{\mathrm{half}},(1)}| +|E^{\mathrm{\mathrm{half}},(2)}|$ edges. 

We write $G^{(1)}[\mathbf{M};\mathbf{P}]$ (resp. $G^{(2)}[\mathbf{M};\mathbf{P}]$) the submultigraph $G_{\cup}[\mathbf{M};\mathbf{P}]$ that only contains the nodes and the half-edges arising from $G^{(1)}$ (resp. $G^{(2)}$). Obviously, $G^{(1)}[\mathbf{M};\mathbf{P}]$ is isomorphic to $G^{(1)}$.

\begin{figure}
    \begin{center}
        \includegraphics[scale=1.3]{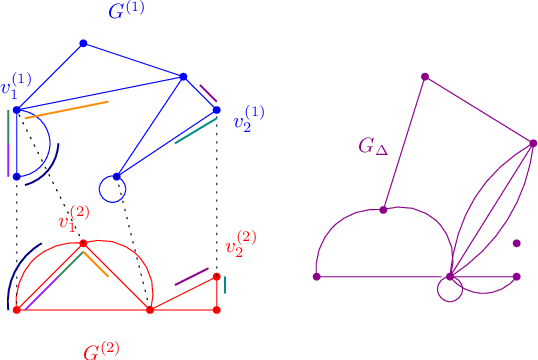}
    \end{center}
    \caption{{\bf Example of construction of $G_{\Delta}$.} The two multigraphs $G^{(1)}$ and $G^{(2)}$ are represented in respectively blue and red. The node matchings are indicated by dotted lines between the concerned nodes. The half-edge pairings are note by segments of colour. Two half-edges flagged with the same colour are paired together. The nodes $v_1^{(1)}, v_1^{(2)}$ resp.~$v_2^{(1)},v_2^{(2)}$ are matched together. $G_{\Delta}$ is then displayed in purple. Note that matched nodes appear only once, and that some edges are removed, and some are added - corresponding to the half-edge pairings. One pair of nodes is {\bf fully matched} and belongs to $\mathbf M_{\mathrm{full}}$, namely all half-edges connected to it are paired (in $\mathbf P$) - and it is therefore isolated in $G_{\Delta}$. %One edge is {\bf fully paired} - namely has both its corresponding half-edges paired (in $\mathbf P$) - and it belongs to $E_{\mathrm{full}}$.
    }\label{tildeG}
\end{figure}

\paragraph{Paths and cycles.}  In what follows, we {\bf color} each of the half-edges that correspond to $\mathbf P$ with a different color - each half-edge belonging to a pair are associated to a given, unique colour. See Figure~\ref{tildeG} for an example. We can then define cycles and paths along coloured paths.
\begin{itemize}
    \item When matching and pairing the multigraphs, we can have an {\bf open path} that is obtained by following a connected path of coloured half-edges in $\mathbf P$ such that it has open extremities - namely it starts and finished in non-paired half-edges. In what follows, we call the \emph{length} of the open path the number of pairs of half-edges that are involved in this path. We write respectively $\mathrm{Op}_{\mathrm{odd}}$ and $\mathrm{Op}_{\mathrm{even}}$ 
    for the sets of open paths obtained from our matching and pairing, that have an odd, resp.~an even, number of pairs of half-edges involved in the construction. We will mostly use their cardinality $|\mathrm{Op}_{\mathrm{odd}}|, |\mathrm{Op}_{\mathrm{even}}|$. Figure~\ref{paths} illustrates an example of both open paths.
    \item In a related way, we can have a {\bf cycle} that is obtained by following a connected path of coloured half-edges in $\mathbf P$ such that it has no open extremities and forms a cycle. We write $\mathrm{Cyc}$ 
    for the set of cycles obtained when matching and pairing the multigraphs. We will mostly use its cardinality $|\mathrm{Cyc}|$. Figure~\ref{cycle} illustrates an example of both open paths.
\end{itemize}

\begin{figure}
    \begin{center}
        \includegraphics[scale=0.8]{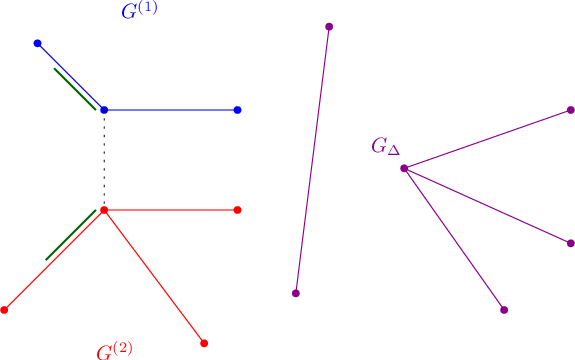}~~~~~~~~\includegraphics[scale=0.8]{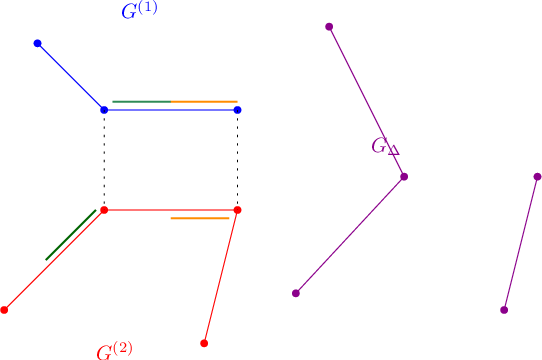}
    \end{center}
    \caption{{\bf Examples of open paths.} The picture on the left depicts a simple example of an open path that has an odd number of pairs of half-edges involved, while the picture on the right depicts an example of an open path that has an even number of pairs of half-edges involved. In each case, to create $G_{\Delta}$, the edges such that both half-edges are paired are removed, and we draw an edge through the open extremities.}\label{paths}
\end{figure}

\begin{figure}
    \begin{center}
        \includegraphics[scale=0.8]{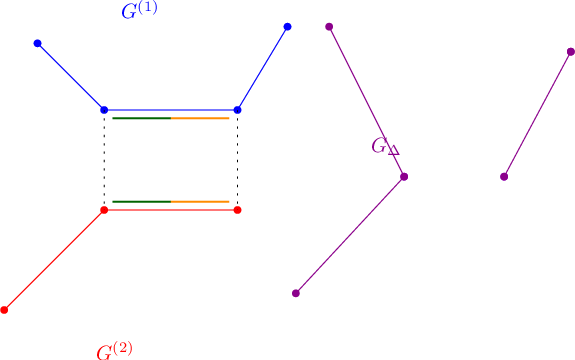}~~~~~~~~\includegraphics[scale=0.8]{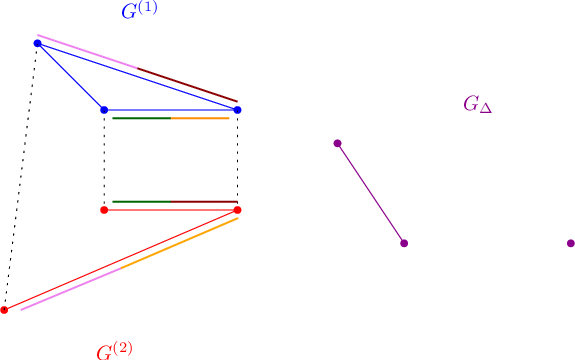}
    \end{center}
    \caption{{\bf Examples of cycles.} The picture on the left depicts an example of a simple cycle generated by two pairs of half-edges. %Note that the two corresponding edges are perfectly paired. 
    The picture on the right presents a more complicated case of a longer cycle. In each case, the cycles are removed to create $G_{\Delta}$.}\label{cycle}
\end{figure}

\paragraph{Construction of $G_{\Delta}$.} Finally, we define $G_{\Delta}:=G_{\Delta}[\mathbf{M},\mathbf{P}]=(V_{\Delta},E_{\Delta})$ as the pruned multi-graph of $G_{\cup}[\mathbf{M},\mathbf{P}]$ as follows: 
\begin{itemize} 
\item We remove all the edges that belong to a cycle.
\item For any maximum open path with colored half-edges,  we remove all these-half-edges and replace them by an edge between the two extremities.
\end{itemize}
Finally, we write $\mathrm{CC}(G_{\Delta})$ for the set of  connected components of $G_{\Delta}$.
Figure~\ref{tildeG} illustrates an example of construction - note that other examples are found in Figures~\ref{cycle} and Figure~\ref{paths}.

%In what follows, we also write $\mathcal{M}^{\star} \subset \mathcal{M}$ for the set of all node matchings that ensure that each connex components of $G^{(1)}, G^{(2)}$ have a least one which is matched, and such that all nodes or degree $2$ in $G^{(1)}, G^{(2)}$ are matched. We write then $\mathcal{M}^{\star} \mathcal P$ for the corresponding set of pairing and matching.

%Consider any $G^{(1)}$ and $G^{(2)}$ in $\mathcal G_{\leq D}$. Consider a subset $\mathbf{M} \subset \mathcal{M}$ of matched nodes. Define $\mathcal{P}_{\mathrm{full}}(\mathbf{M})$ the collection of pairing $\mathbf{P} \in \mathcal P(\mathbf{M})$ such that $|\mathbf{P}|= |E^{\mathrm{\mathrm{half}},(1)}|= |E^{\mathrm{\mathrm{half}},(2)}|$. Note that $\mathcal{P}_{\mathrm{full}}(\mathbf{M})$ is non-empty if and only if $G^{(1)} = G^{(2)}$ and $|\mathbf{M}|= |V^{(1)}|= |V^{(2)}|$.

%\include{polynotations.tex}

\section{Control of the first two moments of the polynomials}\label{sec:mom}

\subsection{Control of first moments}

In this section, we compute $\mathbb{E}[x\widetilde{\Psi}_G]$ for a multigraph $G=(V,E)$.

\begin{lemma}\label{lem:first_moment_tilde_psi}
Consider any $G\in \mathcal{G}_{\leq D}$. As long $G$ contains at least one edge, we have
\[
\mathbb{E}\left[\widetilde{\Psi}_G\right] = 0\ . 
\]
If $G$ is the graph with only two isolated nodes, $\widetilde{\Psi}_G= 1$ almost surely. 
\end{lemma}

\begin{lemma}\label{lem:first_moment_degreeuneven}
    Consider any $G\in \mathcal{G}_{\leq D}$.  If $G$ contains at least one node of odd degree, then 
    \[
    \mathbb{E}[x\overline{\Psi}_G]=0 \ ; \quad \quad \quad \mathbb{E}[x\widetilde{\Psi}_G]=0 
    \]
    \end{lemma}

\begin{lemma}\label{lem:first_moment_degree2}
Consider any $G\in \mathcal{G}_{\leq D}^{even}$.  If, possibly aside from the nodes $v_1$ and $v_2$, $G$ contains at least one node of degree $2$, then 
$$\mathbb{E}[\overline{\Psi}_G]=0\ ; \quad \quad \mathbb{E}[x\overline{\Psi}_G]=0\ ;\quad \quad \mathbb{E}[x\widetilde{\Psi}_G]=0 \enspace .$$
\end{lemma}

In the sequel, we write $\overline{G}$, the graph $G$ where we add an edge between $1$ and $2$.

\begin{lemma}\label{lem:first_moment} Consider any $G\in \mathcal{G}_{\leq D}^{even}$. 
    If, possibly aside from the node $1$ and $2$,  $G$ only contains nodes of even degree larger than or equal to $2$, then
    \begin{equation}\label{eq:moment_first_degree}
    \mathbb{E}[x\overline{\Psi}_G]= |\Pi_V| \Delta^{2|E|}\frac{1}{K^{|V|-|\mathrm{CC}(\overline{G})|}}\ ; \quad \mathbb{E}[\overline{\Psi}_G]= |\Pi_V| \Delta^{2|E|}\frac{1}{K^{|V|-|\mathrm{CC}(G)|}}\enspace . 
    \end{equation}
\end{lemma}

\begin{corollary}\label{cor:first_moment}
     Consider any $G\in \mathcal{G}_{\leq D}^{even}$. If $G$ does not contain any edge, then  $\mathbb{E}[x\widetilde{\Psi}_G]=1/K$
      If, possibly aside from the nodes $v_1$ and $v_2$,  $G$ only contains nodes of even degree larger than or equal to $4$ and if $G$ is connected, then
    \begin{equation}\label{eq:moment_first_degreetilde}
    \mathbb{E}[x\widetilde{\Psi}_G]= |\Pi_V| \Delta^{2|E|}\frac{1}{K^{|V|-1}} \left(1 - \frac{1}{K}\right)\enspace .
    \end{equation}
    Otherwise, we have $\mathbb{E}[x\widetilde{\Psi}_G]= 0$.
\end{corollary}

%\newpage

\subsection{Control of the variance}

For any two graphs $G^{(1)}$ and $G^{(2)}$, we compute in this section the cross-moments $\mathbb{E}[\overline{\Psi}_{G^{(2)}} \overline{\Psi}_{G^{(2)}}]$ and $\mathbb{E}[\widetilde{\Psi}_{G^{(1)},\pi^{(1)}} \widetilde{\Psi}_{G^{(2)},\pi^{(2)}}]$.

\begin{lemma}\label{lem:covariance_degree_odd}
    Consider any $G^{(1)}\in \mathcal{G}_{\leq D}\setminus \mathcal{G}^{even}_{\leq D}$ and any $G^{(2)}\in \mathcal{G}^{even}_{\leq D}$.  For any injective maps $\pi^{(1)}$ and $\pi^{(2)}$, we have 
   \[
   \mathbb{E}[\overline{\Psi}_{G^{(1)},\pi^{(1)}} \overline{\Psi}_{G^{(2)},\pi^{(2)}}]=0\ ; \quad \quad  \mathbb{E}[\widetilde{\Psi}_{G^{(1)},\pi^{(1)}} \widetilde{\Psi}_{G^{(2)},\pi^{(2)}}]=0.
    \]
   \end{lemma}

\begin{lemma}\label{lem:covariance_degree2}
 Consider a graph $G^{(1)}\in \mathcal{G}^{even}_{\leq D}$ that contains at least a degree $2$ node say $v$ with $v\notin\{v^{(1)}_1,v^{(1)}_2\}$, and consider any graph $G^{(2)}$. For any $\pi^{(1)}$ and $\pi^{(2)}$ such that $\pi^{(1)}(v)$ is not in the image of $\pi^{(2)}$, we have 
\[
\mathbb{E}[\overline{\Psi}_{G^{(1)},\pi^{(1)}} \overline{\Psi}_{G^{(2)},\pi^{(2)}}]=0\  ; \quad \quad  \mathbb{E}[\widetilde{\Psi}_{G^{(1)},\pi^{(1)}} \widetilde{\Psi}_{G^{(2)},\pi^{(2)}}]=0\enspace . 
 \]
\end{lemma}

\begin{proposition}\label{prop:covariance_without_degree_2}
    Consider any $G^{(1)}$ and $G^{(2)}$ in $\mathcal{G}_{\leq D}^{\mathrm{even}}$ that do not contain any degree $2$ node, possibly aside from their first two nodes $v_1$ and $v_2$. Fix any injective mapping $\pi^{(1)}$ and $\pi^{(2)}$ for $G^{(1)}$ and $G^{(2)}$ with $\pi^{(1)}(v^{(1)}_1)=\pi^{(2)}(v^{(2)}_1)=1$ and   $\pi^{(2)}(v^{(1)}_2)=\pi^{(2)}(v^{(2)}_2)=2$. Define the subset of matched nodes $\mathbf{M}\subset V^{(1)}\times  V^{(2)}$ by 
    \[
    \mathbf{M}:= \{(v^{(1)}_i,v^{(2)}_j): \pi^{(1)}(v_i^{(1)})= \pi^{(2)}(v_j^{(2)}) \} \ . 
    \]
 Given a pairing $\mathbf{P}\in \mathcal{P}[\mathbf{M}]$, recall the pruned graph $G_{\Delta}[\mathbf{M},\mathbf{P}]= (V_{\Delta},E_{\Delta})$ and the number $|\mathrm{Cyc}|$ of pruned cycles. Then, we have 
    We have 
        \[
            \mathbb{E}[\overline{\Psi}_{G^{(1)},\pi^{(1)}} \overline{\Psi}_{G^{(2)},\pi^{(2)}}]= \sum_{\mathbf{P}\in \mathcal{P}[\mathbf{M}]}\Delta^{2|E_{\Delta}|}d^{|\mathrm{Cyc}|}\frac{1}{K^{|V_{\Delta}|-|\mathrm{CC}(G_{\Delta})|}}
        \]
    \end{proposition}

\begin{proposition}\label{prop:covariance_general}
    Consider any $G^{(1)}=(V^{(1)},E^{(1)})$ and $G^{(2)}=(V^{(2)},E^{(2)})$ in $\mathcal{G}_{\leq D}^{\mathrm{even}}$. Fix any injective mapping $\pi^{(1)}$ and $\pi^{(2)}$ for $G^{(1)}$ and $G^{(2)}$ such that  with $\pi^{(1)}(v^{(1)}_1)=\pi^{(2)}(v^{(2)}_1)=1$ and   $\pi^{(2)}(v^{(1)}_2)=\pi^{(2)}(v^{(2)}_2)=2$ and such that all degree $2$ nodes in $G^{(1)}$ and in $G^{(2)}$ are matched, that is they arise in $\mathbf{M}$. 
Given a pairing $\mathbf{P}\in \mathcal{P}(\mathbf{M})$, we define $V_{2,np}(\mathbf{P})$ the subset of degree $2$ nodes in $V^{(1)}\cup V^{(2)}\setminus\{v^{(1)}_1,v^{(1)}_2,v^{(2)}_1,v^{(2)}_2\}$ such that none of their incident half-edges belongs to $\mathbf{P}$.
\begin{eqnarray*}
\left|\mathbb{E}[\overline{\Psi}_{G^{(1)},\pi^{(1)}} \overline{\Psi}_{G^{(2)},\pi^{(2)}}] - \sum_{\mathbf{P}\in  \mathcal{P}_{\mathrm{full}}(\mathbf{M})}\frac{\Delta^{2|E_{\Delta}|}d^{|\mathrm{Cyc}|}}{K^{|V_{\Delta}|-|\mathrm{CC}(G_{\Delta})|}}  \right|\leq   \sum_{\mathbf{P}\in \mathcal{P}[\mathbf{M}]\setminus \mathcal{P}_{\mathrm{full}}(\mathbf{M})} 2^{|V_{2,np}(\mathbf{P}) |} \frac{\Delta^{2|E_{\Delta}|}d^{|\mathrm{Cyc}|}}{K^{|V_{\Delta}|-|\mathrm{CC}(G_{\Delta})|}}\ .
\end{eqnarray*}
   \end{proposition}
   We refer to Subsection~\ref{pairmatch} for the definition of $\mathcal{P}_{\mathrm{full}}(\mathbf{M})$, namely the set of pairings according to the matchings $\mathbf M$ such that all half-edges attached to nodes in $\mathbf M$ are paired. In particular, $\mathcal{P}_{\mathrm{full}}(\mathbf{M})$ is empty unless $|V^{(1)}|=|V^{(2)}|$, all nodes are matched are matched, and all half-edges are paired. This disctinction between $\mathbf{P}\in \mathcal{P}_{\mathrm{full}}(\mathbf{M})$ and $\mathbf{P}\in \mathcal{P}[\mathbf{M}]\setminus \mathcal{P}_{\mathrm{full}}(\mathbf{M})$ will be of interest later only in the case where $G^{(1)}=G^{(2)}$.

   In the following proposition, recall that $\mathbf{M}_{\mathrm{full}}\subset \mathbf{M}$, is the subset of node matching such that all half-edges incident to the corresponding nodes are paired.

\begin{proposition}\label{prop:covariance_generaltilde}
    Consider any two graphs $G^{(1)}$ and $G^{(2)}$ in $\mathcal G_{\leq D}^{\mathrm{even}}$. 
    We have  
    \begin{eqnarray}\nonumber
    \lefteqn{\left|\mathbb{E}[\overline{\Psi}_{G^{(1)}} \overline{\Psi}_{G^{(2)}}] - \sum_{\mathbf{M}\in \mathcal{M}}\sum_{\mathbf{P}\in  \mathcal{P}_{\mathrm{full}}(\mathbf{M})}\frac{(n-2)!}{(n - |V_{\Delta}| )!}\Delta^{2|E_{\Delta}|}d^{|\mathrm{Cyc}|}\frac{1}{K^{|V_{\Delta}|-|\mathrm{CC}(G_{\Delta})|}}  \right|} &\\
    &\leq  \sum_{\mathbf{M} \in \mathcal{M}} \sum_{\mathbf{P}\in \mathcal{P}[\mathbf{M}]\setminus \mathcal{P}_{\mathrm{full}}(\mathbf{M})}2^{2(|\mathbf{M}|-|\mathbf{M}_{\mathrm{full}}|) } \frac{(n-2)!}{(n-|V_{\Delta}|)!} \Delta^{2|E_{\Delta}|}d^{|\mathrm{Cyc}|}\frac{1}{K^{|V_{\Delta}|-|\mathrm{CC}(G_{\Delta})|}} \enspace .\label{eq:moment_order_2}
    \end{eqnarray}
Consider any $G \in \mathcal G_{\leq D}^{\mathrm{even}}$ such that all its nodes have degree larger than or equal to $4$. Then, we have
\begin{equation}\label{eq:formula_variance}
    \mathbb{E}[\overline{\Psi}_{G}^2] 
    =  \sum_{\mathbf{M}\mathbf{P} \in \mathcal{MP}}  \frac{(n-2)!}{(n-|V_{\Delta}|)!} \Delta^{2|E_{\Delta}|}d^{|\mathrm{Cyc}|}\frac{1}{K^{|V_{\Delta}|-|\mathrm{CC}(G_{\Delta})|}}\enspace .
\end{equation}
If $G^{(1)}$ or $G^{(2)}$ is the graph with two isolated nodes, then $\mathbb{E}[\widetilde{\Psi}_{G^{(1)}} \widetilde{\Psi}_{G^{(2)}}] = \mathbf{1}\{G^{(1)}\backsimeq  G^{(2)}\}$. 
If, neither $G^{(1)}$ nor $G^{(2)}$ is the graph with two isolated nodes, we have 
\begin{eqnarray}\nonumber
    \lefteqn{\left|\mathbb{E}[\widetilde{\Psi}_{G^{(1)}} \widetilde{\Psi}_{G^{(2)}}] - \sum_{\mathbf{M}\in \mathcal{M}^{\star}}\sum_{\mathbf{P}\in  \mathcal{P}_{\mathrm{full}}(\mathbf{M})}\frac{(n-2)!}{(n - |V_{\Delta}|)!}\Delta^{2|E_{\Delta}|}d^{|\mathrm{Cyc}|}\frac{1}{K^{|V_{\Delta}|-|\mathrm{CC}(G_{\Delta})|}}  \right|} &\\
    &\leq  \sum_{\mathbf{M} \in \mathcal{M}^{\star}} \sum_{\mathbf{P}\in \mathcal{P}[\mathbf{M}]\setminus \mathcal{P}_{\mathrm{full}}(\mathbf{M})} 2^{4[|\mathbf{M}|- |\mathbf{M}_{\mathrm{full}}|]} \frac{(n-2)!}{(n-|V_{\Delta}|)!} \Delta^{2|E_{\Delta}|}d^{|\mathrm{Cyc}|}\frac{1}{K^{|V_{\Delta}|-|\mathrm{CC}(G_{\Delta})|}}\enspace .\label{eq:covariance_tilde} 
    \end{eqnarray}
   \end{proposition}
   We refer to Subsection~\ref{pairmatch} for the definition of $\mathcal{M}^{\star}$, namely the set of matchings such that all connected component have a matched node and such that all degree $2$ nodes are matched.

\section{Proof of the low-degree  lower bound - Theorem~\ref{thm:LB_polynomial}}\label{sec:proof_LB}

We first reduce the case the larger-dimensional regime $d\geq K$ to the case where $d=K$. We write $\mathrm{Corr}_{\leq 2D;d;K}$ for the low-degree correlation in dimension $d$ with $2K$ groups. We claim that, for $d> K$, we have  $\mathrm{Corr}_{\leq 2D;d;K}\leq \mathrm{Corr}_{\leq 2D;K;K}$. Indeed, the $d$-dimensional problem where the learner has access both to the matrix $Y$ and the space $U= \mathrm{span}(\mu_k, k \leq K)$ is exactly equivalent to a $K$-dimensional problem.

\medskip 

Hence, from now on, we restrict ourselves to the regime $d=K$. In this section, all the proofs of the technical lemmas are postponed to Appendix~\ref{sec:proof:technical:UB}.

The crux of the proof is to show that the family $\widetilde{\Psi}_G$, with $G\in \mathcal{G}_{\leq D}^{\mathrm{even}}$ is almost orthogonal, in the sense that the covariance matrix associated to this family have eigenvalues that are close to one. This result is stated in Theorem~\ref{thm:Riez_constant}.
Then, it follows from this theorem that 
\begin{align*}
    \mathrm{Corr}^2_{\leq 2D} &\leq 
     \sup_{(\alpha_G)_{G\in \mathcal G^{\mathrm{even}}_{\leq D}}} \frac{1}{1-u}\cdot \frac{\left(\sum_G \alpha_G \mathbb E\left[x \cdot \frac{1}{\sqrt{\mathbb V(G)}}\widetilde{\Psi}_{G}\right]\right)^2}{\|\alpha\|_2^2} = \frac{1}{1-u}\cdot \sum_{G\in \mathcal G^{\mathrm{even}}_{\leq D}} \frac{\left[\mathbb{E}(x  \widetilde{\Psi}_{G})\right]^2}{\mathbb{V}(G)}\ , 
\end{align*}
by Cauchy-Schwarz equality. First, we know from Corollary~\ref{cor:first_moment} that $\mathbb{E}(x \cdot \widetilde{\Psi}_{G})=1/K$ if $G$ has no edge. We set aside this graph henceforth. 
 By Corollary~\ref{cor:first_moment} again, $\mathbb{E}(x \cdot \widetilde{\Psi}_{G})=0$ unless $G$ is connected, the degree of all its nodes is even, and for any node aside from $v_1$ and $v_2$, its degree is higher than $4$. Let us call $\mathcal G^{\mathrm{even,4}}_{\leq D}$ the corresponding collection of graph. For any $G\in \mathcal G^{\mathrm{even,4}}_{\leq D}$, we again deduce from Corollary~\ref{cor:first_moment} and from the definition~\eqref{eq:definition:V(G)} of $\mathbb{V}(G)$ that 
\[
\frac{\left[\mathbb{E}(x  \widetilde{\Psi}_{G})\right]^2}{\mathbb{V}(G)}= \frac{(n-2)!}{(n-|V|)!|\mathrm{Aut}(G)|^2}\frac{\Delta^{4|E|}}{K^{|E|}K^{2|V|-2}} [1 - K^{-1}]^2\leq \left(\frac{\sqrt{n}\Delta^4}{K^2}\right)^{2(|V|-2)}  \left(\frac{\Delta^4}{K}\right)^{|E|-2(|V|-2)}  \frac{1}{K^2}\enspace . 
\]
Observe that $|E|-2(|V|-2)$ is non-negative as $G$ belongs to $ \mathcal G^{\mathrm{even,4}}_{\leq D}$.
Given $v\geq 2$, and $e\geq 1$ such that $e\geq 2(v-1)$, we denote $N_{v,e}$ the number of graph in $ \mathcal G^{\mathrm{even,4}}_{\leq D}$ with $v$ nodes and $e$ edges. A rough bound of $N_{v,e}$ is $v^{2e}$. Then, we arrive at 
\begin{align*}
    \mathrm{Corr}^2_{\leq 2D} 
    &\leq \frac{1}{(1-u)K^2}+ \frac{1}{K^2(1-u)}\sum_{v=2}^{D}\sum_{e = 2(|v|-2)\vee 1}^{D} v^{2e} \left(\frac{\sqrt{n}\Delta^4}{K^2}\right)^{2(v-2)} \left(\frac{\Delta^{4}}{K}\right)^{[e - 2(v-2)]}\\
    &\leq \frac{1}{K^2}+ \frac{2}{K^2} \left[u + \frac{4\Delta^4}{K}+\sum_{v=3}^{D}\sum_{e = 2(|v|-2)}^{D}  \left(\frac{D^2\sqrt{n}\Delta^4}{K^2}\right)^{2(v-2)} \left(\frac{D^2\Delta^{4}}{K}\right)^{[e - 2(v-2)]}  \right]\\
    &\leq \frac{1}{K^2}+   \frac{8}{K^2} \left[\left(\frac{8}{D}\right)^{c/16}
        + \frac{\Delta^4}{K}+  \frac{D^4n \Delta^8}{K^4}  \right]\ , 
\end{align*}
where we used that $u\leq 1/2$, the definition of $u$, $D^2\frac{\Delta^{4}}{K} \leq 1/2$ and $D^4n\Delta^8/K^4 \leq 1/2$. This concludes the proof of the theorem.

\section{Proof of Theorem~\ref{thm:Riez_constant}}\label{sec:proof:Riez}

Define $\Gamma$ as the covariance matrix indexed by $\mathcal{G}_{\leq D}^{\mathrm{even}}$ whose general term is $\Gamma_{G^{(1)},G^{(2)}}:=\frac{\mathbb{E}[\widetilde{\Psi}_{G^{(1)}} \widetilde{\Psi}_{G^{(2)}}]}{\sqrt{\mathbb V(G^{(1)}){\mathbb V}(G^{(2)})}}$. We follow the same general proof strategy as in~\cite{CGGV25} by establishing that the matrix $\Gamma$ is diagonal dominant to control its eigenvalues. First, we consider the case where $G$  is the graph without an edge. In this case, we know from Proposition~\ref{prop:covariance_generaltilde} that $\Gamma_{G,G}=1$ and that $\Gamma_{G,G'}=0$ for any $G\neq G'$. Hence, we only have to focus in the following on graphs $G^{(1)}$ and $G^{(2)}$ that contain at least one edge. Most of the proofs (Steps 1 up to Step 4) amount to deriving a sharp bound of  $\Gamma_{G^{(1)},G^{(2)}}$. For that purpose, we build on the moment bounds from Section~\ref{sec:mom}.

\subsection{Step 1: Expression of the pseudo-correlations}

    Consider any $G^{(1)}$ and $G^{(2)}$ in $\mathcal G_{\leq D}^{\mathrm{even}}$ such that $G^{(1)} \neq G^{(2)}$ and also assume that both $G^{(1)}$ and $G^{(2)}$ have at least one edge.  Since $G^{(1)}\neq G^{(2)}$, both templates are not isomorphic and, for all node matching $\mathbf{M}$, $\mathcal{P}_{\mathrm{full}}(\mathbf{M})$ is empty. Hence, we easily deduce from Proposition~\ref{prop:covariance_generaltilde} that
\begin{equation}\label{eq:control_covariance:PT}
    \left| \mathbb{E}[\widetilde{\Psi}_{G^{(1)}} \widetilde{\Psi}_{G^{(2)}}]\right|\leq  \sum_{\mathbf{M},\mathbf{P}\in \mathcal{M}^{\star}\mathcal{P}} 2^{4[|\mathbf M|- |\mathbf M_{\mathrm{full}}|]} \frac{(n-2)!}{(n-|V_{\Delta}|)!} \Delta^{2|E_{\Delta}|}d^{|\mathrm{Cyc}|}\frac{1}{K^{|V_{\Delta}|-|\mathrm{CC}(G_{\Delta})|}} \enspace . 
\end{equation}
We now use that $d=K$ and the definition~\eqref{eq:definition:V(G)} of $\mathbb{V}(G^{(1)})$ and $\mathbb{V}(G^{(2)})$. Hence, it follows that
\begin{eqnarray}\lefteqn{\frac{ \left| \mathbb{E}[\widetilde{\Psi}_{G^{(1)}} \widetilde{\Psi}_{G^{(2)}}]\right|}{\sqrt{\mathbb{V}(G^{(1)})\mathbb{V}(G^{(2)})}}[|\mathrm{Aut}(G^{(1)})||\mathrm{Aut}(G^{(2)})|]^{1/2} \leq} & &\nonumber\\ 
    & &\sum_{\mathbf{M},\mathbf{P}\in \mathcal{M}^{\star}\mathcal{P}}  2^{4[|\mathbf M|- |\mathbf M_{\mathrm{full}}|]}n^{-|V^{(1)}|/2-|V^{(2)}|/2+|V_{\Delta}|} \Delta^{2|E_{\Delta}|}{K^{|\mathrm{Cyc}|-|V_{\Delta}|-|\mathrm{CC}(G_{\Delta})|-(|E^{(1)}|+|E^{(2)}|)/2}} \nonumber \\   \label{eq:control_covariance:PT:bis}
&\leq &   \sum_{\mathbf{M},\mathbf{P}\in \mathcal{M}^{\star}\mathcal{P}} 2^{4[|\mathbf M|- |\mathbf M_{\mathrm{full}}|]} \left(\frac{n^{1/2}\Delta^4}{K^{2}}\right)^{b_0} \left(\frac{\Delta^{2}}{\sqrt{K}}\right)^{b_1} \left(\frac{1}{\sqrt{K}}\right)^{b_2} \ ,
\end{eqnarray}
where $b_0$, $b_1$, and $b_2$ are respectively defined as 
\begin{align}\label{eq:expression_b0}
b_0&:= 2|V_{\Delta}|-|V^{(1)}|- |V^{(2)}| = |V^{(1)}| + |V^{(2)}| - 2|\mathbf{M}|\enspace \ ;\\
\label{eq:defintion:b1}
    b_1& := -2b_0 + |E_{\Delta}| \enspace ,\\
\label{eq:defintion:b2}
b_2 &:= \left[|E^{(1)}|+ |E^{(2)}| - 2|\mathrm{Cyc}|\right] -|E_{\Delta}| - 2|\mathrm{CC}(G_{\Delta})| + 2 |V_{\Delta}|  - 2b_0\enspace . 
\end{align}
In the remainder of the proof, we shall rearrange and bound $b_0$, $b_1$, and $b_2$. Still, we might already observe that $b_0$ is non-negative by definition of $\mathbf{M}$.

\subsection{Step 2: bounds on $b_1,b_2$}

In order to bound $b_1$ and $b_2$, we need further notation. We remind the reader that $|\mathbf{M}_{\mathrm{full}}|$ is equal to the number of isolated nodes in $G_{\Delta}$. Equivalently, $|\mathbf{M}_{\mathrm{full}}|$ is the number of paired nodes $(v^{(1)}_{i}, v^{(2)}_{j})$ such  that all half-edges in $E^{\mathrm{\mathrm{half}},(1)}$ and $E^{\mathrm{\mathrm{half}},(2)}$ that are incident to $v^{(1)}_{i}$ in $G^{(1)}$  or $v^{(2)}_{j}$ in $G^{(2)}$ are paired, that is they arise in $\mathbf{P}$. 
We also recall the reader that  $|\mathrm{Cyc}|$, $|\mathrm{Op}_{\mathrm{even}}|$, and $|\mathrm{Op}_{\mathrm{odd}}|$ respectively stand for the number of cycles, of open-paths of even length, and open-path of odd length that have been pruned in the construction of $G_{\Delta}$. 
The two following lemmas are the main combinatorial tools of the proofs. Their proofs are postponed to the next section. 
\begin{lemma}\label{lem:connected_components}
        For any $\mathbf{M}\in \mathcal{M}^{\star}$, we have 
        \[
            2 |\mathrm{CC}(G_{\Delta})| \leq 2|\mathrm{Op}_{\mathrm{even}}|+ |\mathrm{Op}_{\mathrm{odd}}|+ 2|\mathbf{M}|.
        \]
\end{lemma}

Define 
\begin{eqnarray}\label{eq:definition_B}
    B&:= &|\mathbf{P}| -  2|\mathrm{Cyc}| - 2|\mathrm{Op}_{\mathrm{even}}|- |\mathrm{Op}_{\mathrm{odd}}| \ ;\\
\label{eq:definition:C}
    C &:= &|E^{(1)}|+ |E^{(2)}| - [2|V^{(1)}| + 2|V^{(2)}|-3|\mathbf{M}|-|\mathbf{M}_{\mathrm{full}}|] - |\mathbf{P}|\ .
\end{eqnarray}

\begin{lemma}\label{lem:pairing}
   For any $\mathbf{M}$ and $\mathbf{P}$, we have $B\geq 0$ and $C\geq 0$.
    \end{lemma}
We deduce from the definition of $C$ and from the fact $|E_{\Delta}|= |E^{(1)}|+|E^{(2)}|- |\mathbf{P}|$ --see the construction of $G_{\Delta}$ 
 that 
\begin{eqnarray}
b_1 &= &3 [|\mathbf{M}|- |\mathbf{M}_{\mathrm{full}}|]+  C \enspace , 
\label{eq:b_1:bis} 
\end{eqnarray}
which is non-negative by Lemma~\ref{lem:pairing}. Similarly, we deal with $b_2$ by noting that $|V_{\Delta}|=|V^{(1)}|+|V^{(2)}|-|\mathbf{M}|$, and by applying Lemmas~\ref{lem:connected_components} and~\ref{lem:pairing}
\begin{eqnarray}\nonumber
    b_2 &= &2 |\mathbf{M}| + |\mathbf{P}| - 2|\mathrm{Cyc}| - 2|\mathrm{CC}(G_{\Delta})|  \geq  B  \label{eq:b_2:bis}\enspace ,
    \end{eqnarray}
    which is again non-negative.

\subsection{Step 3: Bound on the pseudo-correlation}

Define 
\begin{equation}\label{eq:definition:psi_P_T}
\phi(\mathbf{M},\mathbf{P}):=   2[|V^{(1)}| + |V^{(2)}| -2 |\mathbf{M}|] + 2B+C +  3(|\mathbf{M}|-|\mathbf{M}_{\mathrm{full}}|)\enspace . 
\end{equation}
Using Lemma~\ref{lem:pairing} to deduce that $\phi(\mathbf{M},\mathbf{P}) \geq 2(|\mathbf{M}|-|\mathbf{M}_{\mathrm{full}}|)$.

By Assumption, we have $(\frac{n^{1/4}\Delta^2}{K})\vee \frac{\Delta^2}{\sqrt{K}}\vee \frac{1}{K^{1/4}}\leq  (4D)^{-c}$\ , 
for $c\geq c_0\geq 1$.  Combining the previous bounds on $b_0$, $b_1$, and $b_2$ with~\eqref{eq:control_covariance:PT:bis}, we arrive at 
    \begin{align}
        \frac{ \left| \mathbb{E}[\widetilde{\Psi}_{G^{(1)}} \widetilde{\Psi}_{G^{(2)}}]\right|}{\sqrt{\mathbb{V}(G^{(1)})\mathbb{V}(G^{(2)})}}&\leq  [|\mathrm{Aut}(G^{(1)})||\mathrm{Aut}(G^{(2)})|]^{-1/2} \sum_{\mathbf{M},\mathbf{P}\in \mathcal{M}^{\star}\mathcal{P}} 2^{4[|\mathbf M|- |\mathbf M_{\mathrm{\mathrm{full}}}|]} (4D)^{- c \phi(\mathbf{M},\mathbf{P}) } \nonumber \\
        &\leq  [|\mathrm{Aut}(G^{(1)})||\mathrm{Aut}(G^{(2)})|]^{-1/2} \sum_{\mathbf{M},\mathbf{P}\in \mathcal{M}^{\star}\mathcal{P}}  D^{- c \phi(\mathbf{M},\mathbf{P}) }\enspace .     \label{eq:control_covariance:PT:3}
\end{align}
In order to control this sum over $\mathbf{M}$ and $\mathbf{P}$, we shall group matching and pairings that have some common structure that we shall call henceforth a shadow. This strategy is somewhat inspired from the arguments of~\cite{CGGV25} for node matching of simple graphs, but the arguments have to be refined here by also taking into account pairing between the half-edges. Given $(\mathbf{M}, \mathbf{P})$, we say that two edges $e$ and $e'$ of $G^{(1)}$ and $G^{(2)}$ are {\bf perfectly paired} if the two corresponding half-edges of $e$  are paired in $\mathbf{P}$  with the corresponding half-edges of $e'$. Similarly, we say that two half-edges are {\bf perfectly paired} if they belong to perfectly paired edges. Define $\underline{\mathbf{P}}\subseteq \mathbf{P}$ the subset of $\underline{P}$ where we have removed all perfectly paired half-edges. Let $\underline{PE}^{(1)}$ (resp. $\underline{PE}^{(2)}$) be the subset of half-edges of $G^{(1)}$ (resp. $G^{(2)})$ that are perfectly paired. Finally, we say that two nodes in $(v,v')$ in $\mathbf{M}$ are {\bf perfectly matched} if all their incident half-edges belong to $\underline{PE}^{(1)}$ (resp. $\underline{PE}^{(2)}$). Finally, we define $\underline{\mathbf{M}}$ as the subset of $\mathbf{M}$ where we removed all perfectly matched nodes. We call $(\underline{\mathbf{M}}, \underline{\mathbf{P}}, \underline{PE}^{(1)},\underline{PE}^{(2)})$ the {\bf shadow} of $(\mathbf{M},\mathbf{P})$. See Figure~\ref{fig:shadow} for an illustration.

\begin{figure}
    \begin{center}
        \includegraphics[scale=1.3]{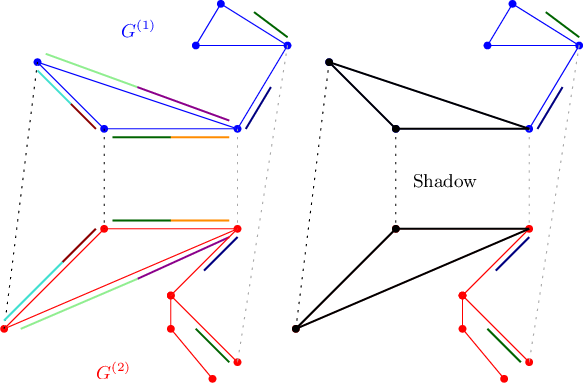}
    \end{center}
    \caption{{\bf Example of a shadow.} Here, the perfectly paired edges and perfectly mactched nodes in the two triangles are then displayed in black in the shadow, because the precise information on the matchings and pairings leading to these perfect matches is lost in the shadow (i.e.~node label and precise pairing). All the remaining information is kept intact in the shadow. In this example, there are only two matching-pairings in $\mathcal{M}\mathcal{P}$ that share this shadow. Indeed, we can only switch the matching between the two perfectly matched nodes and adjust the corresponding paired half-edges to obtain the same shadow.}\label{fig:shadow}
\end{figure}

Write $\underline{\mathcal{M}\mathcal{P}}_{\mathrm{shadow}}$ for the collection of all such $(\underline{\mathbf{M}}, \underline{\mathbf{P}}, \underline{PE}^{(1)},\underline{PE}^{(2)})$  that are the shadow of one element in $\mathcal{M}^{\star}\mathcal{P}$. For any element $(\mathbf{M}^{'},\mathbf{P}^{'})$ of $\mathcal{M}^{\star}\mathcal{P}$ we write $(\underline{\mathbf{M}}, \underline{\mathbf{P}}, \underline{PE}^{(1)},\underline{PE}^{(2)}) \triangleleft  (\mathbf{M}^{'},\mathbf{P}^{'})$ if $(\underline{\mathbf{M}}, \underline{\mathbf{P}}, \underline{PE}^{(1)},\underline{PE}^{(2)})$ is the shadow of $(\mathbf{M}^{'},\mathbf{P}^{'})$. Note that $\phi(\mathbf{M},\mathbf{P})$ defined in~\eqref{eq:definition:psi_P_T} is the same for all $(\mathbf{M},\mathbf{P})$ that share the same shadow. Hence, with a slight abuse of notation, we can define $\phi(\underline{\mathbf{M}},\underline{\mathbf{P}},\underline{PE}^{(1)},\underline{PE}^{(2)})$ for  $(\underline{\mathbf{M}},\underline{\mathbf{P}})\in \underline{\mathcal{M}\mathcal{P}}_{\mathrm{shadow}}$.  The following lemma, proved in the appendix, bounds the number of matchings-pairings that share the same shadow.

\begin{lemma}\label{lem:counting_isomorphism}
    For any $(\underline{\mathbf{M}}, \underline{\mathbf{P}}, \underline{PE}^{(1)},\underline{PE}^{(2)})\in \underline{\mathcal{M}\mathcal{P}}_{\mathrm{shadow}}$, we have 
    \[
    \left|\Big\{(\mathbf{M}^{'},\mathbf{P}^{'})\in \mathcal{M}\mathcal{P} :  (\underline{\mathbf{M}}, \underline{\mathbf{P}}, \underline{PE}^{(1)},\underline{PE}^{(2)})\triangleleft  (\mathbf{M}^{'},\mathbf{P}^{'}) \Big\}\right|\leq \min (|\mathrm{Aut}(G^{(1)})|,|\mathrm{Aut}(G^{(2)})|)\enspace .
    \]
\end{lemma}
In Light of Lemma~\ref{lem:counting_isomorphism}, we deduce from~\eqref{eq:control_covariance:PT:3} that 
\begin{eqnarray}\label{eq:control_covariance:PT:4}
    \frac{ \left| \mathbb{E}[\widetilde{\Psi}_{G^{(1)}} \widetilde{\Psi}_{G^{(2)}}]\right|}{\sqrt{\mathbb{V}(G^{(1)})\mathbb{V}(G^{(2)})}}  \leq   \sum_{\underline{\mathbf{M}},\underline{\mathbf{P}},\underline{PE}^{(1)},\underline{PE}^{(2)}\in \underline{\mathcal{MP}}_{\mathrm{shadow}}} D^{-c \phi(\underline{\mathbf{M}},\underline{\mathbf{P}},\underline{PE}^{(1)},\underline{PE}^{(2)})} \enspace . 
\end{eqnarray}

\paragraph{Control of the sum over $\underline{\mathcal{MP}}_{\mathrm{shadow}}$.}

Note that the definition of $(\underline{\mathbf{M}},\underline{\mathbf{P}},\underline{PE}^{(1)},\underline{PE}^{(2)})$ is completely characterized by (i) the set of half-edges that are not-perfectly paired, (ii) the possible pairing among these half-edges and (iii) the matching $\underline{\mathbf{M}}$ of nodes that are not perfectly matched.
In what follows, we write 
$m$
for the number of  half-edges that are not perfectly paired. Since the total number of half-edges of $G^{(1)}$ and of $G^{(2)}$ is at most $2D$, we deduce from the 
previous property that there are at most 
\begin{equation}\label{card:m}
   (4 D)^{m}\cdot m^{m}\cdot m^{m}\leq (4D)^{3m}\ ,
\end{equation}
    shadows $(\underline{\mathbf{M}},\underline{\mathbf{P}},\underline{PE}^{(1)},\underline{PE}^{(2)})$ with $m$ half-edges that are not perfectly paired. Indeed,  (1) this makes less than $(4D)^m$ possibilities for choosing these half-edges. Then, we choose whether those are paired and to which half-edge they are paired, hence less $m^m$ possibilities. Finally, for nodes incident to non-paired half-edges, we must decide whether they are matched and to which node they are matched, hence less than $m^m$ possibilities.

\begin{lemma}\label{lem:lower_bound_psi}
Consider any $\underline{\mathbf{M}}$, $\underline{\mathbf{P}}$, $\underline{PE}^{(1)}$, $\underline{PE}^{(2)}$. Writing $m$ for the number of half-edges that are not perfectly paired we have 
\[
    \phi(\underline{\mathbf{M}},\underline{\mathbf{P}},\underline{PE}^{(1)},\underline{PE}^{(2)})\geq \frac{1}{8} m 
\]
\end{lemma}

For any two multigraphs $G^{1}$ and $G^{(2)}$, define 
\begin{equation}\label{eq:definition:m*}
    m^*(G^{(1)},G^{(2)}) := \min_{\underline{\mathbf{M}},\underline{\mathbf{P}},\underline{PE}^{(1)},\underline{PE}^{(2)}} 
    m \ , 
\end{equation}
as the minimum over all shadows $(\underline{\mathbf{M}},\underline{\mathbf{P}},\underline{PE}^{(1)},\underline{PE}^{(2)})$ of the number of $m$ of half-edges that are not perfectly paired. Obviously, we have $m^*(G^{(1)},G^{(2)})\geq 0$. Besides, one readily checks that $m^*(G^{(1)},G^{(2)})=0$ implies that that $G^{(1)}$ and $G^{(2)}$ are isomorphic. 

So that we deduce from Equation~\eqref{eq:control_covariance:PT:4}, Equation~\eqref{card:m} and Lemma~\ref{lem:lower_bound_psi} that
\begin{eqnarray}\label{eq:control_covariance:PT:5}
\Gamma_{G^{(1)},G^{(2)}} =     \frac{ \left| \mathbb{E}[\widetilde{\Psi}_{G^{(1)}} \widetilde{\Psi}_{G^{(2)}}]\right|}{\sqrt{\mathbb{V}(G^{(1)})\mathbb{V}(G^{(2)})}}  \leq   \sum_{m\geq m^*(G^{(1)},G^{(2)})} (4D)^{3m} D^{-cm/8} \leq  \left(\frac{4}{D}\right)^{cm^*(G^{(1)},G^{(2)})/16}\left(\frac{1}{1- 4/D}\right) \enspace ,
\end{eqnarray}
for $c \geq 48$.
\subsection{Step 4: Special case of the rescaled variance}

We now consider the diagonal terms of $\Gamma$. Consider any graph $G=G^{(1)} = G^{(2)}$ in $\mathcal{G}_{\leq D}^{\mathrm{even}}$. We argue quite similarly to the non-diagonal terms except that we use slightly differently Proposition~\ref{prop:covariance_generaltilde} to get a tighter bound for perfect matchings and pairings. Let us call $\mathcal{MP}^{\mathrm{perf}}$ the collection of $(\mathbf{M},\mathbf{P}$) that correspond to perfect pairings of all edges. Note that in this case, the graph $G_{\Delta}$ does not contain any edge. Hence, we have from Proposition~\ref{prop:covariance_generaltilde} that 
\begin{equation}\label{eq:upper_variance}
\left|\mathbb{E}[\widetilde{\Psi}^2_{G^{(1)}}] -  \big|\mathcal{MP}^{\mathrm{perf}}\big|\frac{(n-2!)}{(n-|V_{\Delta}^{(1)}|!)}d^{|E^{(1)}|} \right|\leq  \sum_{\mathbf{M},\mathbf{P}\in \mathcal{M}^{\star}\mathcal{P}\setminus \mathcal{MP}^{\mathrm{perf}}}  2^{4[|\mathbf M|- |\mathbf M_{\mathrm{full}}|]} \frac{(n-2)!}{(n-|V_{\Delta}|)!} \Delta^{2|E_{\Delta}|}d^{|\mathrm{Cyc}|}\frac{1}{K^{|V_{\Delta}|-|\mathrm{CC}(G_{\Delta})|}} \enspace .
\end{equation}
One easily checks that $\big|\mathcal{MP}^{\mathrm{perf}}\big|= |\mathrm{Aut}(G^{(1)})|$ so that we recognize $\mathbb{V}(G^{(1)})$ in the above inequality. Note that, if $(\mathbf{M},\mathbf{P})$ does not belong to to $\mathcal{MP}^{\mathrm{perf}}$, there are at least $m\geq 2$ half-edges that are not perfectly paired. Hence, arguing exactly as in Step 2 and in Step 3, we arrive at 
\begin{eqnarray}\label{eq:control_covariance:PT:5'}
\left|\frac{ \mathbb{E}[\widetilde{\Psi}^2_{G^{(1)}}]}{\mathbb{V}(G^{(1)})} -1 \right| \leq   \sum_{m\geq 2} (4D)^{3m} D^{-cm/8} \leq  \left(\frac{4}{D}\right)^{c/8}\left(\frac{1}{1- 4/D}\right) \enspace . 
\end{eqnarray}

\subsection{Step 5: Conclusion on the pseudo-correlation matrix}

Define $H = \Gamma - \mathrm{Id}$. To control the eigenvalues of $\Gamma$, we bound the $l_1$ norm of each row of $H$.
We fix a multigraph $G^{(1)} \in \mathcal G_{\leq D}^{\mathrm{even}}$. Note that, by Equations~\eqref{eq:control_covariance:PT:5} and~\eqref{eq:control_covariance:PT:5'}, we have
\begin{align*}
    \sum_{G^{(2)} \in \mathcal G_{\leq D}^{\mathrm{even}}}\left|H_{G^{(1)}, G^{(2)}}\right|
    &\leq \sum_{G^{(2)} \in \mathcal G_{\leq D}^{\mathrm{even}}} \left(\frac{4}{D}\right)^{cm^*( G^{(1)},  G^{(2)})/16} \left(\frac{1}{1- 4/D}\right)+ \left(\frac{4}{D}\right)^{c/8}\left(\frac{1}{1- 4/D}\right)\\
    &\leq 2  \left(\frac{4}{D}\right)^{c/8}+ 2\sum_{m\geq 2} |\{G^{(2)}: m^*( G^{(1)}, G^{(2)}) = m\}| \left(\frac{4}{D}\right)^{cm/16}\enspace ,
\end{align*}
for $D \geq 6$.

\begin{lemma}\label{lem:combinatoire_m*}
    Consider any graph $G^{(1)} \in \mathcal G_{\leq D}^{\mathrm{even}}$. Then, for any $m\geq 2$, we have  
    $$|\{G^{(2)}\in \mathcal G_{\leq D}^{\mathrm{even}} : m^*( G^{(1)},  G^{(2)}) = m\}| \leq (2D)^{m}\enspace .$$
\end{lemma}
\begin{proof}[Proof of Lemma~\ref{lem:combinatoire_m*}]
    Consider any multigraph $G^{(2)}$ in $\mathcal{G}_{\leq D}$ such that $|\{G^{(2)}\in \mathcal G_{\leq D}^{\mathrm{even}}: m^*( G^{(1)},  G^{(2)}) = m\}|=m$. We claim that there exists an isomorphism $G'^{(2)}$ of $G^{(2)}$ such that the edit distance  between $G^{'(2)}$ and $G^{(1)}$ is less than $m/2$. Up to isomorphisms, we conclude that $|\{G^{(2)}: m^*( G^{(1)},  G^{(2)}) = m\}|$  is smaller than the set of multigraphs within edit distance $m/2$ from $G^{(1)}$. It is therefore smaller or equal to $(2D)^{m}$. 
    \end{proof}
We conclude from this lemma that 
\begin{eqnarray}\label{eq:control_covariance:PT:6}
    \sum_{G^{(2)} \in \mathcal G_{\leq D}^{\mathrm{even}}}\left|H_{G^{(1)}, G^{(2)}}\right| \leq 8 \left(\frac{8}{D}\right)^{c/16}\enspace ,
\end{eqnarray}
provided that $c\geq 32$ and $D\geq 16$. Hence, the matrix $\Gamma$ is diagonal dominant and its eigenvalue lies in $[1-u, 1+u]$ where $u = 8 \left(\frac{8}{D}\right)^{c/16}$. 
This concludes the proof of Theorem~\ref{thm:Riez_constant}.

\section{Proof of the upper bound}\label{sec:proof:UB}
\subsection{Proof of Theorem~\ref{thm:final_upper_bound}}

Define $t_0:= \frac{[(n-2)/\Lambda]!}{[2+ (n-2)/\Lambda-|V|]!} \cdot \frac{\Delta^{2|E|}}{\Lambda^{|E|} K^{|V|-2}}$. 
It follows from Theorem~\ref{thm:moment:rappeur} that, under the conditions of the theorem, we have
\[
\mathbb{P}\left[T^{(\ell)}_{G^*}\leq \frac{t_0}{2}|\mu,b_1,b_2,k^*(1),k^*(2), x=1\right]\leq \frac{1}{4} \ ; \mathbb{P}\left[T^{(\ell)}_{G^*}\geq \frac{t_0}{2}|\mu,,b_1,b_2,k^*(1),k^*(2), x=0\right]\leq \frac{1}{4}\enspace . 
\]
As conditionally to $\mu$ and $k^*(1)$, $k^*(2)$, $b_1$, and $b_2$, the random variables $T^{(\ell)}_{G^*}$ are i.i.d., we conclude that 
\[
\mathbb{P}[\hat{x}\neq x|\mu, x] \leq \mathbb{P}[\mathrm{Bin}(\Lambda,1/4)\geq \Lambda/2] \leq n^{-3}\ . 
\]
This concludes the proof.

%\begin{proof}[Proof of Theorem~\ref{thm:cross_variance}\]
\subsection{Proof of Theorem~\ref{thm:moment:rappeur}}
The following lemmas provide explicit expressions of the first and second moments of $\overline{\Psi}_{G^*}$. The statements and the proofs are closely related to Lemmas~\ref{lem:first_moment} and  Propositions~\ref{prop:covariance_without_degree_2} and~\ref{prop:covariance_generaltilde}.
\begin{lemma}
    \label{lem:upper_boundexpectation}
        Let $G^*$ be the multigraph defined in Section~\ref{sec:upper}. Then
      \begin{equation}\label{eq:expectation:psi_G^*}
            \mathbb{E}\left[\overline{\Psi}_{G^*}|\mu, k^*(1),k^*(2),b_1,b_2\right]= x \frac{(n-2)!}{(n-|V|)!} \frac{\Delta^{2|E|}}{K^{|V|-2}} \enspace . 
      \end{equation}
    \end{lemma}
\begin{remark}\label{remark:expection:G}
In fact, the above formula~\eqref{eq:expectation:psi_G^*} only depends on the specific form of $G^*=(V,E)$ through $|V|$ and $|E|$. Hence, a corresponding expression also holds for any connected graph $\bar{G}=(\bar{V},\bar{E})$ that contains only nodes with even degree at least $4$. 
\end{remark}

To compute the covariance, we consider two replicas of $G^*$, say $G^{*(1)}=(V^{(1)},E^{(1)})$ and $G^{*(2)}=(V^{(2)},E^{(2)})$. In the sequel, $\mathbf{M}$ will correspond to a matching of $V^{(1)}$ and $V^{(2)}$ as introduced in Section~\ref{sec:multigraph}. Since, by definition,  we always have $(v_1^{(1)}, v_1^{(2)})\in \mathbf{M}$ and $(v_2^{(1)}, v_2^{(2)})\in \mathbf{M}$, the nodes $v_1^{(1)}$ and  $v_1^{(2)}$ (resp. $v_2^{(1)}$ and  $v_2^{(2)}$) are identified in the merged graph $G_{\cup}[\mathbf{M},\mathbf{P}]$ and in the pruned graph $G_{\Delta}[\mathbf{M},\mathbf{P}]$ and we refer, for short, to these nodes as $v_1$ (resp. $v_2$). In this proof, we write $v_1 \sim_{G_\Delta} v_2$  (resp. $v_1 \nsim_{G_\Delta} v_2$) when $v_1$ and $v_2$ belong (resp. do not belong) to the same community in $G_{\Delta}$. 

\begin{lemma}\label{lem:UBva}
The conditional variances of $\overline{\Psi}_{G^*}$ write as follows. 
    \begin{align*}
 \mathbb{E}\left[\left[\overline{\Psi}_{G^*} - \mathbb E[\overline{\Psi}_{G^*}|\mu,x=0]\right]^2 |\mu,b_1,b_2,k^*(1),k^*(2), x=0\right]\hspace{3cm}\\ 
 =  \sum_{\mathbf{M},\mathbf{P}\in \mathcal{MP}:~\mathbf{P} \neq \emptyset} \frac{(n-2)!}{(n-|V_{\Delta}|)!} \Delta^{2|E_{\Delta}|}d^{|\mathrm{cyc}|}\frac{1}{K^{|V_{\Delta}|-|\mathrm{CC}(G_{\Delta})|}}  \mathbf{1}\{v_1 \nsim_{G_\Delta} v_2\}\enspace , 
    \end{align*}
and 
\begin{align*}
 \mathbb{E}\left[\left[\overline{\Psi}_{G^*} - \mathbb E[\overline{\Psi}_{G^*}|\mu,x=1]\right]^2 |\mu,b_1,b_2,k^*(1),k^*(2), x=1\right]\hspace{3cm}\\
 =  \sum_{\mathbf{M},\mathbf{P}\in \mathcal{MP}:~\mathbf{P} \neq \emptyset} \frac{(n-2)!}{(n-|V_{\Delta}|)!} \Delta^{2|E_{\Delta}|}d^{|\mathrm{cyc}|}\frac{1}{K^{|V_{\Delta}|-|\mathrm{CC}(G_{\Delta})|-\mathbf{1}\{v_1 \sim_{G_\Delta} v_2\}}}   \\ 
  \quad \quad \quad \quad + \mathbb E^2[\overline{\Psi}_{G^*}|\mu,x=1]   \sum_{\mathbf{M}\in  \mathcal{M} }   \frac{(n-|V|)!^2}{(n-2)!(n-2|V|+|\mathbf{M}|)! } \left(K^{|\mathbf{M}|- 2}-1\right)
 \enspace . \noindent 
\end{align*}
%
 %   Let $G$ be the multigraph defined above. %Let $\pi^{(1)},\pi^{(2)}$ be any two injective mapping from $V$ to $[n]$ such that $\pi^{(1)}(1) = \pi^{(2)}(1)=1, \pi^{(1)}(2) = \pi^{(2)}(2)=2$. 
 %   Then:
 %  \begin{align*}
  %      &\mathbb{E}\left[\left[\overline{\Psi}_{G} - \mathbb E[\overline{\Psi}_{G}|\mu,k^*]\right]^2 |\mu, k^*\right]\\
  %      &=  \sum_{\mathbf{M},\mathbf{P}\in \mathcal{MP}:~\mathbf{P} \neq \emptyset} \frac{(n-2)!}{(n-|V_{\Delta}|)!} \Delta^{2|E_{\Delta}|}d^{|\mathrm{cyc}|}\frac{1}{K^{|V_{\Delta}|-|\mathrm{CC}(G_{\Delta})|-|\mathrm{CC}(G_{\Delta}, 1,2)|}}  \mathbf 1\{|\mathrm{CC}(G_{\Delta}, 1,2)| =0~~\mathrm{OR}~~x=1\},
  %  \end{align*}
  %  using the notations from the previous section, and where $|\mathrm{CC}(G_{\Delta}, 1,2)|$ is the number of connected components of $G_{\Delta}$ that contain both $1$ and $2$ (note that it can either be $0$ or $1$).
\end{lemma}
In particular, we readily check from this lemma that 
\[
 \mathbb{E}\left[\left[\overline{\Psi}_{G^*} - \mathbb E[\overline{\Psi}_{G^*}|\mu,x=0]\right]^2 |\mu,b_1,b_2,k^*(1),k^*(2), x=0\right]\leq \mathbb{E}\left[\left[\overline{\Psi}_{G^*} - \mathbb E[\overline{\Psi}_{G^*}|\mu,x=1]\right]^2 |\mu,b_1,b_2,k^*(1),k^*(2), x=1\right]\ , 
\]
and we mainly have to control $\mathrm{Var}_1(\overline{\Psi}_{G^*}):=\mathbb{E}\left[\left[\overline{\Psi}_{G^*} - \mathbb E[\overline{\Psi}_{G^*}|\mu,x=1]\right]^2 |\mu,b_1,b_2,k^*(1),k^*(2), x=1\right]$.

Given $\mathbf{M}$ with $|\mathbf{M}|\geq 3$, and $\mathbf{P}\in \mathcal{P}(\mathbf{M})$ define $A[\mathbf{M},\mathbf{P}]$ by 
\begin{align}\label{eq:definition:A:M:P}
 A[\mathbf{M},\mathbf{P}] :=  \frac{(n-2)!}{(n-|V_{\Delta}|)!} \Delta^{2|E_{\Delta}|}d^{|\mathrm{cyc}|}\frac{1}{K^{|V_{\Delta}|-|\mathrm{CC}(G_{\Delta})|-\mathbf{1}\{v_1 \sim_{G_\Delta} v_2\}}}\ ; \quad \quad A:=\sum_{\substack{\mathbf{M},\mathbf{P}\in \mathcal{MP}:\\|\mathbf{M}|\geq 3~\mathrm{and}~\mathbf{P} \neq \emptyset}} A[\mathbf{M},\mathbf{P}]\enspace .
\end{align}
Denote $\mathbf{M}_0$ as the only node matching such that $|\mathbf{M}_0|=2$. Define
\begin{align*}
B&: = \sum_{\mathbf{P}\in \mathcal{P}(\mathbf{M}_0)\neq \emptyset} \frac{(n-2)!}{(n-|V_{\Delta}|)!} \Delta^{2|E_{\Delta}|}d^{|\mathrm{cyc}|}\frac{1}{K^{|V_{\Delta}|-|\mathrm{CC}(G_{\Delta})|-\mathbf{1}\{v_1 \sim_{G_\Delta} v_2\}}}\ ; \\
C &: = \mathbb E^2[\overline{\Psi}_{G^*}|\mu,x=1]   \sum_{\mathbf{M}\in  \mathcal{M} }   \frac{(n-|V|)!^2}{(n-2)!(n-2|V|+|\mathbf{M}|)! } \left(K^{|\mathbf{M}|- 2}-1\right)\enspace . 
\end{align*}
Hence, it follows from Lemma~\ref{lem:UBva} that we have the decomposition 
\begin{align}\label{eq:decomposition_V1}
\mathrm{Var}_1(\overline{\Psi}_{G^*})= A  + B + C\enspace .
\end{align}

\begin{remark}
The decomposition~\eqref{eq:decomposition_V1} does not use explicitly the topology of $G^*=(V,E)$, except that $G^*$ is connected. In particular, a similar formula holds for any connected multigraph $\bar{G}=(\bar{V},\bar{E})$ such that also the degree of all nodes is even and at least equal to $4$. Also, observe that all the terms in the sum defining $A$, $B$, and $C$ are non-negative. Consider any such connected $\bar{G}$ such that 
$\bar{v}_1$ and $\bar{v}_2$ do not contain any self-edge.  Denote $d_{\bar{G}}(\bar v_1)$ and $d_{\bar{G}}(\bar v_2)$ for the degrees of $\bar{G}$. 

First, consider the specific term in $B$ where $\mathbf{M}=\mathbf{M}_0$ and the pairing $\mathbf{P}$ such that each half-edge $\underline{e}^{(1)}$ incident to $\bar{v}^{(1)}_1$ to $\bar{v}^{(2)}_1$ in the first replica of $\bar{G}$ is paired to the corresponding half-edge  $\underline{e}^{(2)}$ in the second replica of $\bar{G}$. Together with~\eqref{eq:expectation:psi_G^*} (and the corresponding remark), this leads us  to 
\begin{align}\nonumber
\mathrm{Var}_1(\overline{\Psi}_{\bar{G}})& \geq \frac{(n-2)!}{(n-2|V|+2)!}\frac{\Delta^{4|\bar{E}|-2d_{\bar{G}}(v_1)-2d_{\bar{G}}(v_2)}}{K^{2|V|-5}}\\
&\geq \mathbb E^2[\overline{\Psi}_{\bar{G}}|\mu,x=1] \left(1- \frac{2|\bar{V}|}{n}\right)^{|\bar{V}|} \frac{K}{\Delta^{2(d_{\bar{G}}(\bar v_1)+d_{\bar{G}}(\bar v_2))}}\enspace . \label{eq:lower_variance}
\end{align}
Second, consider the specific choice of ${\bf M}$ and ${\bf P}$ in $A[{\bf M},{\bf P}]$ such that $|V_{\Delta}|=V$ and $E_{\Delta}=\emptyset$, and $|\mathrm{Cyc}|= |\tilde{E}|$. This corresponding to perfectly matching and pairing the corresponding replicas $\bar{G}^{(1)}$ and $\bar{G}^{(2)}$. Then, we have 
\begin{align} 
\mathrm{Var}_1(\overline{\Psi}_{\bar{G}})& \geq \frac{(n-2)!}{(n-|V|)!}K^{|\bar{E}|}
\geq \mathbb E^2[\overline{\Psi}_{\bar{G}}|\mu,x=1]  n^{-|\bar{V}|+2}\frac{K^{|\bar{E}|+2[|\bar{V}|-2]}}{\Delta^{4|\bar E|}} \enspace . \label{eq:lower_variance_2}
\end{align}

\end{remark}

We control each of the terms $A$, $B$, and $C$. We start with $C$. By assumption, we have $n\geq 4|V|$. It then follows from the definition of $C$ that 
\begin{align}\label{eq:bound:C}
\frac{C}{\mathbb E^2[\overline{\Psi}_{G^*}|\mu,x=1] } &\leq  \sum_{\mathbf{M}\in  \mathcal{M}: |\mathbf{M}|\geq 3 } \left(\frac{2K}{n}\right)^{|\mathbf M|-2}\leq  \sum_{m=3}^{|V|} \left(\frac{2(|V|-2)^2 K}{n}\right)^{m-2}\leq \frac{4(|V|-2)^2K}{n} \ , 
\end{align}
since the number of matchings of size $m$ is equal to 
$(\binom{|V|-2}{m-2})^2 (m-2)!\leq (|V|-2)^{2(m-2)}$ and since we assume that $n\geq 4(|V|-2)^2 K$.

\begin{lemma}\label{prop:UBvar1}
    %Let $G$ be the multigraph defined above. %Let $\pi^{(1)},\pi^{(2)}$ be any two injective mapping from $V$ to $[n]$ such that $\pi^{(1)}(1) = \pi^{(2)}(1)=1, \pi^{(1)}(2) = \pi^{(2)}(2)=2$. 
Provided that $\Delta^2 \geq 2(M+1)^{4}$, we have 
    \begin{align*}
B 
&\leq \mathbb E^2[\overline{\Psi}_{G^*}|\mu,x=1]\Big[\frac{K(M+1)^{2(M+1)}}{\Delta^{4(M+1)}} + 2  \frac{(M+1)^{5}}{\Delta^{2}}\Big]\ . 
\end{align*}
\end{lemma}

%\subsection{Part 2: Bound on $B$, i.e. the term in the sum such that just $2$ nodes are paired}\label{ss:2no}
%µ
%We have by definition:
%\begin{align*}
%    B &= \left[\frac{K^{|V| - 2} }{\Delta^{2|E|}} \frac{(n - |V|)!}{(n-2)!}\right]^2  \times \left[\sum_{1 \leq m \leq 2M+4} \frac{(n-2)!}{(n-(2|V|-2))!} \binom{2M+4}{2m} 2^m m! \Delta^{4|E|-2m} \frac{1}{K^{2(|V|-2) - 2\times \mathbf 1\{m \geq M\}}} \right]\\
 %   &= \frac{[(n - |V|)!]^2}{(n-2)!(n-(2|V|-2))!} \times \left[\sum_{1 \leq m \leq 2M+4} \binom{2M+4}{2m} 2^m m! \Delta^{-2m} \frac{1}{K^{ - 2\times \mathbf 1\{m \geq M\}}} \right]\\
 %   & \leq \sum_{1 \leq m \leq 2M+4} \binom{2M+4}{2m} 2^m m! \Delta^{-2m} \frac{1}{K^{ - 2\times \mathbf 1\{m \geq M\}}}\\
 %   & \leq \sum_{1 \leq m \leq 2M+4} (2(2M+4)m)^{2m}   \Delta^{-2m} \frac{1}{K^{ - 2\times \mathbf 1\{m \geq M\}}}.
%\end{align*}
%Since $\Delta^{1/2} \geq 4(2M+4)^2$ and $\Delta^{M/2} \geq K^{2}$, we have %\alex{ICI}
%\begin{equation}\label{boun:B}
%    B \leq \Delta^{-1}.
%\end{equation}

\paragraph{Control of $A[\mathbf{M},\mathbf{P}]$}
Since $|E_{\Delta}|= 2|E| - |\mathbf{P}|$, $|V_{\Delta}|=2|V|-|\mathbf{M}|$, $d=K$, and the assumption $n\geq 4|V|$, we  deduce from Lemma~\ref{lem:upper_boundexpectation} that 
\begin{align}\nonumber
\frac{A(\mathbf{M},\mathbf{P})}{\mathbb E^2[\overline{\Psi}_{G^*}|\mu,x=1]} &= \frac{(n-|V|)!^2}{(n-2)!(n-2|V|+{\mathbf{M}})!}\cdot 
\frac{\Delta^{2|E_{\Delta}|-4|E|}}{K^{|V_{\Delta}|-2(|V|-2)-|\mathrm{CC}(G_{\Delta})| - |\mathrm{Cyc}|-\mathbf{1}\{v_1 \sim_{G_\Delta} v_2\}}}\\  \label{eq:AMP_0}
&\leq  \left(\frac{2}{n}\right)^{|\mathbf{M}|-2}\frac{K^{|\mathbf{M}|-4+ |\mathrm{CC}(G_{\Delta})| + |\mathrm{Cyc}|+\mathbf{1}\{v_1 \sim_{G_\Delta} v_2\}}}{\Delta^{2|\mathbf{P}|}}\\ 
&\leq \left(\frac{2K^4}{n\Delta^8}\right)^{|\mathbf{M}|-2}\frac{K^{-3|\mathbf{M}|+4 +|\mathrm{CC}(G_{\Delta})| + |\mathrm{Cyc}|+\mathbf{1}\{v_1 \sim_{G_\Delta} v_2\}}}{\Delta^{2|\mathbf{P}|-8(|\mathbf{M}|-2)}} \nonumber \\
&\leq \left(\frac{2K^4}{n\Delta^8}\right)^{|\mathbf{M}|-2} \left(\frac{K}{\Delta^4}\right)^{a_1} \frac{1}{\Delta^{2a_2}} \label{eq:AMP}
\end{align}

where $a_1:= -3|\mathbf{M}|+4 +|\mathrm{CC}(G_{\Delta})| + |\mathrm{Cyc}|+\mathbf{1}\{v_1 \sim_{G_\Delta} v_2\}$ and 
\[
a_2= |\mathbf{P}|- 4[|\mathbf{M}|-2]-2a_1 = |\mathbf{P}|+2\left[|\mathbf{M}|-|\mathrm{CC}(G_{\Delta})| - |\mathrm{Cyc}|-\mathbf{1}\{v_1 \sim_{G_\Delta} v_2\}\right]\enspace . 
\]
Next, we control both $a_1$ and $a_2$. 

\begin{proposition}\label{prop:template:chaine_rappeur}
For any $\mathbf{M}$ and $\mathbf{P}$, we have 
\[|\mathrm{Cyc}|+|\mathrm{CC}(G_{\Delta})|  - 1 \leq 2(|\mathbf{M}|-2) +|\mathbf{M}| - \mathbf 1\{(v_1^{(1)},v_1^{(2)})\notin \mathbf{M}_{\mathrm{full}}\} -\mathbf 1\{(v_2^{(1)},v_2^{(2)})\notin \mathbf{M}_{\mathrm{full}}\}  + \lfloor4(|\mathbf{M}|-2)/L\rfloor\enspace . 
\]
\end{proposition}
We deduce that 
\begin{align}\nonumber
a_1 & \leq \mathbf 1\{(v_1^{(1)},v_1^{(2)})\in \mathbf{M}_{\mathrm{full}}\}+ \mathbf{1}\{v_1 \sim_{G_\Delta} v_2\} - \mathbf 1\{(v_2^{(1)},v_2^{(2)})\notin \mathbf{M}_{\mathrm{full}}\}+ \lfloor4(|\mathbf{M}|-2)/L\rfloor \\
&\leq \mathbf 1\{(v_1^{(1)},v_1^{(2)})\in \mathbf{M}_{\mathrm{full}}\}+\lfloor4(|\mathbf{M}|-2)/L\rfloor \ , 
\label{eq:upper:a1}
\end{align}
since $\mathbf 1\{(v_2^{(1)},v_2^{(2)})\in \mathbf{M}_{\mathrm{full}}\}$ implies that the node $v_2$ is isolated in $G_{\Delta}$.

%In order  $\mathbf 1\{(v_1^{(1)},v_1^{(2)})\}$ to belong to $\mathbf{M}_{\mathrm{full}}$, all that $M+1$ half-edges  incident to $v_{1}^{(1)}$ must be paired, which in turn implies that $|\mathbf{P}|\geq |M|+1$. Hence, we deduce from that and from this proposition that 

\medskip 

When we consider the graph $G_{\Delta}$, the so-called "fastener edges" that connect $v_1^{(1)}$ (or $v_1^{(2)}$) to long-distance node play a specific role. For this purpose, we introduce $H_1$ and $H_2$ as follows.  
 Write $H_1\subset E^{\mathrm{half},(1)}$ (resp. $H_2\subset E^{\mathrm{half},(1)}$) for the subset of half-edges of $G^{(1)}$, that are (i) paired (ii) \emph{incident to the node $v_1^{(1)}$} (resp.~$v_2^{(2)})$,  and (iii) such that the corresponding edge is incident to a node that is \emph{unmatched}.

\begin{proposition}\label{prop2:template:chaine_rappeur}
    We have:
%    $$\sum_{i \in \mathcal{M}\setminus\{(v_1^{(1)},v_1^{(2)}),(v_2^{(1)},v_2^{(2)})\}} |\overline{\mathbf{P}}_i| +|S_1| + |S_2| +  2(|\mathbf{M}|-2)   \geq 2(|\mathrm{Cyc}|+|\mathrm{CC}(G_{\Delta})| -1)  - \lfloor 8(|\mathbf{M}|-2)/L\rfloor.$$
\[
    2(|\mathrm{Cyc}|+|\mathrm{CC}(G_{\Delta})| -1)\leq |\mathbf{P}| - |H_1| -  |H_2|
 + 2\left[|\mathbf{M}|-\mathbf{1}\{(v_1^{(1)},v_1^{(2)})\notin \mathbf{M}_{\mathrm{full}}\} -\mathbf{1}\{(v_2^{(1)},v_2^{(2)}) \notin \mathbf{M}_{\mathrm{full}}\}\right] \enspace . 
\]
\end{proposition}
It follows from this proposition and the definition of $a_2$ that 
\begin{align}\nonumber 
a_2 & \geq |H_1|+ |H_2| - 2\mathbf 1\{(v_1^{(1)},v_1^{(2)})\in \mathbf{M}_{\mathrm{full}}\}+ 2[\mathbf{1}\{(v_2^{(1)},v_2^{(2)}) \notin \mathbf{M}_{\mathrm{full}}\} -\mathbf{1}\{v_1 \sim_{G_\Delta} v_2\} ] \\ 
&\geq |H_1|+ |H_2| - 2\mathbf 1\{(v_1^{(1)},v_1^{(2)})\in \mathbf{M}_{\mathrm{full}}\}\enspace  \ , \label{eq:upper:a2}
\end{align}
where we used again that $v_2$ is isolated in $G_{\Delta}$ when $(v_2^{(1)},v_2^{(2)}) \in \mathbf{M}_{\mathrm{full}}$.

\medskip

We consider two cases depending on the values of $\Delta$:~\\
\noindent 
{\bf Case 1}: $\Delta^4 \leq K$. Gathering~\eqref{eq:upper:a1} with~\eqref{eq:upper:a2} and coming back to the expression~\eqref{eq:AMP}, we arrive at 
\[
\frac{A(\mathbf{M},\mathbf{P})}{\mathbb E^2[\overline{\Psi}_{G^*}|\mu,x=1]}
\leq \left(\frac{2K^4}{n\Delta^8}\right)^{|\mathbf{M}|-2} \left(\frac{K}{\Delta^4}\right)^{\lfloor 4(|\mathbf{M}|-2)/L\rfloor} \frac{K^{\mathbf 1\{(v_1^{(1)},v_1^{(2)})\in \mathbf{M}_{\mathrm{full}}\}}}{\Delta^{2(|H_1|+|H_2|)}}
\]
In order to have $(v_1^{(1)},v_1^{(2)})\in \mathbf{M}_{\mathrm{full}}$, we need all $2(M+1)$ half-edges incident to $v_1$ in $G_{\cup}[\mathbf{M},\mathbf{P}]$ arise in $\mathbf{P}$. Hence, we either have $|H_1|\geq (M+1)/2$, or that $|\mathbf{M}|-2\geq (M+1)/2$. This yields
\begin{align}
    \frac{A(\mathbf{M},\mathbf{P})}{\mathbb E^2[\overline{\Psi}_{G^*}|\mu,x=1]}
&\leq \left(\frac{2K^4}{n\Delta^8} \left(\frac{K}{\Delta^4}\right)^{4/L}K^{2/(M+1)} \right)^{|\mathbf{M}|-2}  \left(\frac{K^{2/[M+1]}}{\Delta^{2}} \right)^{|H_1|+|H_2|}
\label{eq:AMP2} \ . 
\end{align}
%Next, we sum this quantity over all $(\mathbf{M},\mathbf{P})$ such that $|\mathbf{M}|\geq 3$ and $|\mathbf{P}|\geq 1$. 

\medskip 

\noindent 
{\bf Case 2}: $\Delta^4 \geq K$. First, if $a_1$ is non-negative, then the term $(K/\Delta^{4})$ in~\eqref{eq:AMP} is bounded by one. Arguing as previously, we deduce that 
\begin{align}
    \frac{A(\mathbf{M},\mathbf{P})}{\mathbb E^2[\overline{\Psi}_{G^*}|\mu,x=1]}
&\leq \left(\frac{2K^4}{n\Delta^8} \Delta^{8/(M+1)} \right)^{|\mathbf{M}|-2}  \left(\frac{\Delta^{8/[M+1]}}{\Delta^{2}} \right)^{|H_1|+|H_2|}
\label{eq:AMP2-bis} \ . 
\end{align}
Let us turn to the situation where $a_1\leq 0$. This implies that 
\[
\kappa := |\mathrm{CC}(G_{\Delta})| + |\mathrm{Cyc}|+\mathbf{1}\{v_1 \sim_{G_\Delta} v_2\}- 2 \leq 3[|\mathbf{M}|-2]\ . 
\]
In this case, we come back to the expression \eqref{eq:AMP_0} of $A(\mathbf{M},\mathbf{P})$ and we use the different decomposition 
\begin{align}\nonumber
\frac{A(\mathbf{M},\mathbf{P})}{\mathbb E^2[\overline{\Psi}_{G^*}|\mu,x=1]} 
\leq \left(\frac{2K}{n}\right)^{[|\mathbf{M}|-2]-\kappa/3} \left(\frac{2K^4}{n\Delta^8}\right)^{\kappa/3} \frac{1}{\Delta^{2a'_2}} \ ,  \label{eq:AMP_0_bis}
\end{align}
where $a'_2:= |\mathbf{P}|-\frac{4\kappa}{3}$.  This quantity is bounded in the next lemma.
\begin{lemma}\label{lem:P:petit}
Recall the sets $H_1$ and $H_2$ defined above Proposition~\ref{prop2:template:chaine_rappeur}. We have 
\[
\kappa \leq \frac{3}{4}|\mathbf{P}| - \frac{|H_1|+|H_2|}{4}\enspace . 
\]
\end{lemma}
This yields $a'_2\geq (|H_1|+|H_2)/3$ and we arrive at the bound   
\begin{align}
    \frac{A(\mathbf{M},\mathbf{P})}{\mathbb E^2[\overline{\Psi}_{G^*}|\mu,x=1]}
&\leq \left(\frac{2K^4}{n\Delta^8} \bigvee \frac{2K}{n} \right)^{|\mathbf{M}|-2}  \left(\frac{1}{\Delta^{2}} \right)^{(|H_1|+|H_2|)/3}
\label{eq:AMP2-ter} \ . 
\end{align}

Gathering~\eqref{eq:AMP2},~\eqref{eq:AMP2-bis}, and~\eqref{eq:AMP2-ter}, we conclude that, in all regimes,  we have 
\begin{align}
    \frac{A(\mathbf{M},\mathbf{P})}{\mathbb E^2[\overline{\Psi}_{G^*}|\mu,x=1]}
&\leq \left(\frac{2K^4}{n\Delta^8} \left(\frac{K}{\Delta^4}\vee 1 \right)^{4/L}(K\vee \Delta^4)^{2/(M+1)}\bigvee \frac{2K}{n} \right)^{|\mathbf{M}|-2}  \left(\frac{(K\vee \Delta^4)^{6/[M+1]}}{\Delta^{2}} \right)^{(|H_1|+|H_2|)/3}
\label{eq:AMP2_final} \ . 
\end{align}

Then, we sum over all possibile $\mathbf{M}$ and all possible $\mathbf{P}$ to control $A$. 
\begin{align*}
\frac{A}{\mathbb E^2[\overline{\Psi}_{G^*}|\mu,x=1]}&\leq  \sum_{\substack{\mathbf{M},\mathbf{P}\in \mathcal{MP}:\\|\mathbf{M}|\geq 3~\mathrm{and}~\mathbf{P} \neq \emptyset} } \left(\frac{2K^4}{n\Delta^8} \left(\frac{K}{\Delta^4}\vee 1 \right)^{4/L}(K\vee \Delta^4)^{2/(M+1)}\bigvee \frac{2K}{n} \right)^{|\mathbf{M}|-2} \\ &\hspace{4cm}  \cdot \left(\frac{(K\vee \Delta^4)^{6/[M+1]}}{\Delta^{2}} \right)^{(|H_1|+|H_2|)/3} \  , \\
&\leq 
\sum_{m=3}^{|V|} \sum_{h_1,h_2=0}^{M+1}
  \left(\frac{2K^{4}(K\vee \Delta^4)^{\tfrac{2}{M+1}}}{n\Delta^8} \left(\frac{K}{\Delta^4}\vee 1 \right)^{\tfrac{4}{L}} \bigvee \frac{2K}{n} \right)^{m-2} \\ &\quad \quad \cdot  \left(\frac{(K\vee \Delta^4)^{\tfrac{6}{M+1}}}{\Delta^{2}} \right)^{(h_1+h_2)/3} \sum_{\substack{\mathbf{M},\mathbf{P}\in \mathcal{MP}:\\|\mathbf{M}|=m; |H_1|=h_1 ; |H_2|=h_2}} 1 \  .
\end{align*}
We have already argued that number of matchings $\mathbf{M}$ of size $m$ is at most equal to $(|V|-2)^{2(m-2)}$. 
\begin{lemma}\label{lem:number_mathching}
    We have
\[
\left|\{\mathbf{P}\in \mathcal{P}(\mathbf{M}): |H_1|=h_1,\ |H_2|=h_2\}\right|\leq 5^{4(|\mathbf{M}|-2)} \cdot (M+2)^{4(|\mathbf{M}|-2)}\cdot (M+1)^{2(h_1+h_2)}\enspace . 
\]
\end{lemma}
\begin{proof}[Proof of Lemma~\ref{lem:number_mathching}]
Pairings only involve half-edges incident to matched nodes. Consider a node matching $(v_i^{(1)},v_{i'}^{(2)})$ where $i,i'>2$. The degrees of those nodes in $G^{(1)}$ and $G^{(2)}$ is equal to 4. Each half-edge incident to $v_i^{(1)}$ is either not paired or is paired to one of $4$ half-edges incident to $v_i^{(2)}$, hence there are at most $5^{4}$ possibilities for each of the $|\mathbf{M}|-2$ nodes. It remains to consider the paired half-edges incident to $v_{i}^{(a)}$ for $i=1,2$ and $a=2$. By symmetry, we focus on $v_{1}^{(1)}$. First, 
there are at most $(M+1)^{2h_1}$ possibilities to choose the $h_1$ half-edges and their corresponding paired half-edges. All the other paired half-edges incident to $v_{1}^{(1)}$ must belong to an edge that is incident to another node that arises in the matching $\mathbf{M}$. Since there are at most $2(|\mathbf{M}|-2)$ such half-edges, there are less than $(M+2)^{2(|\mathbf{M}|-2)}$ possibilities for choosing whether they are paired and to which half-edge they are paired. Arguing similarly for $v_1^{(2)}$ concludes the proof.
\end{proof}

As a consequence, we arrive at 
\begin{align*}
\frac{A}{\mathbb E^2[\overline{\Psi}_{G^*}|\mu,x=1]}
&\leq \sum_{m=3}^{|V|} \sum_{h_1,h_2=0}^{M+1}
  \left(\frac{2\cdot 5^{4}(M+2)^4 K^{4} (K\vee \Delta^4)^{\tfrac{2}{M+1}}}{n\Delta^8} \left(\frac{K}{\Delta^4}\vee 1 \right)^{\tfrac{4}{L}} \bigvee \frac{2K\cdot 5^4 (M+2)^{4}}{n}\right)^{m-2}  \\ & \quad \quad \cdot \left(\frac{(M+1)^6 (K\vee \Delta^4)^{\tfrac{6}{M+1}}}{\Delta^{2}} \right)^{(h_1+h_2)/3} \\
  &\leq  4 \sum_{m=3}^{|V|}
  \left(\frac{2\cdot 5^{4}(M+2)^4 K^{4}(K\vee \Delta^4)^{\tfrac{2}{M+1}}}{n\Delta^8} \left(\frac{K}{\Delta^4}\vee 1 \right)^{\tfrac{4}{L}}\bigvee \frac{2K\cdot 5^4 (M+2)^{4}}{n} \right)^{m-2}  \\ 
  &\leq 10^4(M+2)^4 \left[\frac{ K^{4+\tfrac{2}{M+1}+ \tfrac{4}{L}}}{n\Delta^{8- \tfrac{8}{M+1}}}\bigvee \frac{K}{n}\right] \ , 
\end{align*}
where we used in the second line that $\Delta^2 \geq 8(M+1)^6 (K\vee \Delta^4)^{\tfrac{6}{M+1}}$ and in the last line that $ 10^4 (M+2)^4 K^{4+\tfrac{2}{M+1}+\tfrac{4}{L} }\leq n\Delta^{8(1-\tfrac{1}{M+1})}$ and $n\geq 10^{4} K (M+2)^4$. 
Gathering~\eqref{eq:decomposition_V1},~\eqref{eq:bound:C}, and Lemma~\ref{prop:UBvar1}, we conclude that 
\begin{align*}
\frac{\mathrm{Var}_1(\overline{\Psi}_{G^*})}{\mathbb E^2[\overline{\Psi}_{G^*}|\mu,x=1] }&\leq 10^4(M+2)^4   \left[\frac{ K^{4+\tfrac{2}{M+1}+ \tfrac{4}{L}}}{n\Delta^{8(1-\tfrac{1}{M+1})}}\bigvee \frac{K}{n}\right]+ \frac{K(M+1)^{2(M+1)}}{\Delta^{4(M+1)}} + 2  \frac{(M+1)^{5}}{\Delta^{2}}  +  \frac{4M^2L   ^2K}{n}\enspace . 
\end{align*}
This finishes the proof.

\subsection*{Acknowledgements}
We grateful to Christophe Giraud for many illuminating discussions. The work of A. Carpentier is partially supported by the Deutsche Forschungsgemeinschaft (DFG)- Project-ID 318763901 - SFB1294 "Data Assimilation", Project A03,  by the DFG on the Forschungsgruppe FOR5381 "Mathematical Statistics in the Information Age - Statistical Efficiency and Computational Tractability", Project TP 02 (Project-ID 460867398), and by  the DFG on the French-German PRCI ANR-DFG ASCAI CA1488/4-1 "Aktive und Batch-Segmentierung, Clustering und Seriation: Grundlagen der KI" (Project-ID 490860858). The work of N. Verzelen has  partially been supported by ANR-21-CE23-0035 (ASCAI, ANR). The work of A.~Carpentier and N.~Verzelen is also supported by the Universite franco-allemande (UFA) through the college doctoral franco-allemand CDFA-02-25 "Statistisches Lernen für komplexe stochastische Prozesse".

\printbibliography
\appendix

\section{Proofs of the invariance results from Section~\ref{sec:mom}}

\subsection{Proof of Lemma~\ref{lem:reduction:rotation:permutation}}

Let us consider $f_0$ any minimizer of $\min_{f: \mathrm{deg}(f)\leq D}\mathbb{E}[(f(Y)-x)^2]$. Fix any orthogonal $O$ matrix and denote $Y\cdot O$ the rotation of $Y$, that is the matrix $Y\cdot O$ defined by $(Y\cdot O)_{i,j}= \sum_{l=1}^{d} Y_{i,l}O_{l,j}$. Since the distribution of $X$ and the distribution of the noise $Z$ is invariant by rotation, we have $\mathbb{E}[(f_0(Y\cdot 0)-x)^2]= \mathbb{E}[(f_0(Y)-x)^2]$. Then, define $f_1= \int f_0(Y\cdot O)dO$ by integrating with the Haar measure over all possible $O$. By Jensen's inequality, we deduce that $\mathbb{E}[(f_1(Y)-x)^2]\leq \mathbb{E}[(f_0(Y)-x)^2]$. 

We claim that $f_1$ can be expressed as polynomial in $U=YY^T$. 

\noindent 
{\bf Proof of the claim}: Since $f_0$ is polynomial, it is expressed as linear combination of monomial in $(Y_{i,j})$. By linearity, $f_1$ is therefore a linear combination of terms  of the form  $\int [Y\cdot O]_{i_1j_1}\ldots [Y\cdot O]_{i_rj_r}dO$. Again expanding the terms $[Y\cdot O]_{i_sj_s}$, $f_1$ is a linear combination of terms of the form 
\begin{equation}\label{eq:integral_orthogonal_group}
\sum_{l_1,\ldots, l_r} Y_{i_1l_1},\ldots Y_{i_rl_r}\int O_{l_1,j_1}\ldots O_{l_r,j_r} dO  \ , 
\end{equation}
which are  integrals of polynomials over the orthogonal group.  Here, we rely on the landmark paper~\cite{collins2006integration} which provides an explicit form of such integrals. More precisely, Corollary 3.4 in~\cite{collins2006integration} states that~\eqref{eq:integral_orthogonal_group} is null when $r$ is odd. If $r$ is even, write $\underline{l}= (l_1,\ldots, l_r)$ and $\underline{j}= (l_1,\ldots, l_r)$. A pairing $p$ of $r$ is a partition of $[r]$ into $r/2$ pairs. Given a pairing $p$ and a sequence $\underline{l}$ we write $\delta_{\underline{l}}^{p}$ as the indicator function that is equal to one if all paired indices by $p$  share the same value of $\underline{l}$. Then, by Corollary 3.4 in~\cite{collins2006integration}, there exists a function $\phi$ that maps two pairings $p_1$, $p_2$ to a real number such that 
\[
    \int O_{l_1,j_1}\ldots O_{l_r,j_r} dO = \sum_{p_1, p_2}\delta_{\underline{l}}^{p_1}\delta_{\underline{j}}^{p_2} \phi(p_1,p_2)\ , 
\]
where the sum is over all pairing $p_1$ and $p_2$. Then, we can reorder the sum in~\eqref{eq:integral_orthogonal_group} as 
\[
\sum_{p_2}\delta_{\underline{j}}^{p_2}\sum_{p_1}\phi(p_1,p_2) \sum_{\underline{l}}  \delta_{\underline{l}}^{p_1} \prod_{s=1}^r Y_{i_s,l_s}\ . 
\]
In the above sum over $\underline{l}$, the indicator function $\delta_{\underline{l}}^{p_1}$ is non-zero, only if $p_1$ is compatible with the pairing. As a consequence, this sum is equal to a  product of terms of the form $\prod \langle Y_{i_s}, Y_{i_{s'}}\rangle$, where $s$ and $s'$ are paired by $p_1$. Thus, we have shown that $f_1$ expresses a polynomial of degree at most $D/2$ in $U$. 

Furthermore, the problem is invariant by permutation. Let $\Pi$ denote the collection of permutation over $[n]$ that keep $1$ and $2$ invariant. Then, define $Y\cdot \pi$ as the matrix $(Y_{\pi(i)},j)$. By invariance, the polynomial $f_2$ define by $f_2(Y)= |\Pi|^{-1} \sum_{\pi\in \Pi}f_1(Y\cdot \pi)$ still achieves the minimum risk and satisfies the desired properties.

\subsection{Proof of Lemma~\ref{lem:invariant:graph}}

    First, for any multigraph $G$, one easily checks that $P_{G}$ is invariant by permutation.  Besides, if $G$ and $G'$ are isomorphic, then $P_{G}=P_{G'}$. Indeed, there exists a bijection $\sigma$ from the half-edges of $G$ to that of $G'$ that preserves the incidence equivalence and that preserves the edges. This induces a bijection $\tilde{\pi}$ from $V$ to $V'$ that preserves the edges. As a consequence, 
    \[
    P_{G}= \sum_{\pi}P_{G',\pi}= \sum_{\pi}P_{G,\pi o \tilde{\pi}}= \sum_{\pi}P_{G,\pi}\ , 
    \]
    and the claim  follows. 

    Furthermore, the family $(P_G)$, $G\in \mathcal{G}_{\leq D}$ spans the space of invariant polynomials of degree at most $2D$. Indeed, if we consider any invariant polynomial and we decompose it into monomials, we easily check (i) the monomials can be grouped as a sum over permutation $\pi$ and, in turn, is equal $P_G$ for some multigraph $P_G$.

\subsection{Proof of Lemma~\ref{lem:reduction:degree_pair}}

First, we deal with the family $(\overline{\Psi}_G)$ before showing that $(\widetilde{\Psi}_G)$ spans the suitable space of polynomials. Finally, we deal with polynomials with odd degree.

\begin{lemma}\label{lem:reduction:psi_bar}
The family $(\overline{\Psi}_G)$ with $G\in \mathcal{G}_{\leq D}$ spans the space $\mathcal{P}^{\mathrm{\mathrm{inv}}}_{\leq 2D}$.
\end{lemma}

This lemma is shown below. Next, we prove that the family $(\widetilde{\Psi}_G)$ with $G\in \mathcal{G}_{\leq D}$ spans the the same space. First, observe that if $G$ has only isolated nodes, the  $\widetilde{\Psi}_{G,\pi}= \overline{\Psi}_{G,\pi}=1$.  Consider a multigraph $G$ with at least one edge. By expanding $\widetilde{\Psi}_G$ in the definition~\eqref{eq:definition_psi_tilde}, we observe $\widetilde{\Psi}_G- \overline{\Psi}_G$ is a linear combination of $\overline{\Psi}_{G'}$ for some strict submultigraph $G'$ of $G$. As a consequence, (i)  the polynomials $\widetilde{\Psi}_G$ satisfy the desired invariance properties and (ii) the family $(\widetilde{\Psi}_{G})$ with $G\in \mathcal{G}_{\leq D}$ spans the space $\mathcal{P}^{\mathrm{\mathrm{inv}}}_{\leq 2D}$ of invariant polynomials. 

Thus, we have proved that 
\[
MMSE_{\leq 2D}= \inf_{f\in \mathrm{Vect}((\widetilde{\Psi}_{G})_{G\in \mathcal{G}_{\leq D}})}\mathbb{E}[(f(Y)-x)^2]\ . 
\]
It remains to reduce the space to polynomials with even degrees. 
For that, we take any $f$ in $\mathrm{Vect}((\widetilde{\Psi}_{G})_{G\in \mathcal{G}_{\leq D}})$. We decompose it into $f_1(Y)+ f_2(Y)$ where $f_1(Y)$ is in the span of the $\widetilde{\Psi}_{G}$'s with $G\in \widetilde{\Psi}^{even}_{G}$ and where $f_2(Y)$ is in the span of the  $\widetilde{\Psi}_{G}$'s where the graphs $G$ have at least a node with odd degree. 
Then, 
\[
\mathbb{E}[(f(Y)-x)^2]= \mathbb{E}[(f_1(Y)-x)^2] + \mathbb{E}[f^2_2(Y)] + 2 \mathbb{E}[f_2(Y)(f_1(Y)-x)]
\]
By Lemma~\ref{lem:first_moment_degreeuneven}, we have $\mathbb{E}[f_2(Y)x]= 0$, whereas Lemma~\ref{lem:covariance_degree_odd} states that $\mathbb{E}[f_1(Y)f_2(Y)]= 0$. Hence, we have $ \mathbb{E}[(f_1(Y)-x)^2]\leq \mathbb{E}[(f(Y)-x)^2]$. The result follows.

\begin{proof}[Proof of Lemma~\ref{lem:reduction:psi_bar}]

It is clear from the definition that $\overline{\Psi}_{G}$ is invariant by permutation. As a consequence, we mainly need to prove that $\overline{\Psi}_{G,\pi}$ is invariant by right rotation $O$ of $Y$.  As a warm-up, let us consider a slight variation  $\Psi_{G,\pi}$ of $\overline{\Psi}_{G,\pi}$ where, in~\eqref{eq:Hermite_graph} and \eqref{eq:psi_bar}, we replace the functions $\overline{\psi}$ by the functions $\psi$ -- in other words, we dot not make any specific correction to degree 2 nodes. We write $\Psi_{G,\pi}(Y)$ for emphasizing its dependency on $Y$. We denote $P_{G,\pi}[X]$ for the polynomial defined as $P_{G,\pi}$ but with the  $Y_i$ replaced by the $X_i$. In the remainder, of this proofs, we consider the matrix $X$ as fixed. 
Since $Y-X$ is made of independent Gaussian entries, it follows from Lemma~\ref{lem:hermite} that $\mathbb{E}[\Psi_{G,\pi}(Y)|X]= P_{G,\pi}[X]$. Consider any orthogonal matrix $O$; since $P_{G,\pi}[X]$ is invariant by rotation and the Gaussian distribution is invariant by rotation we have 
\begin{equation}\label{eq:nullite_rotation}
\mathbb{E}[\Psi_{G,\pi}(Y)- \Psi_{G,\pi}(Y\cdot O)|X]= P_{G,\pi}[X] - P_{G,\pi}[X\cdot O]=0
\end{equation}
$X$ almost surely.
Let us show that this implies that 
\begin{equation}\label{eq:equality_rotation}
\Psi_{G,\pi}(Y)= \Psi_{G,\pi}(Y\cdot O) \ . 
\end{equation}
Introduce $\varphi$ as the endomorphism of the space of polynomials in $dn$ variables such, given a polynomial $P$, $\varphi(P)$ is such that $\varphi(P)[X]= \mathbb{E}[P(X+Z)]$, where matrix $Z$ is made of independent standard Gaussian entries. One easily checks that $\varphi(P)$ is polynomial and that $\varphi$ is linear. Besides, the total degree of $\varphi(P)$ is the same as that of $P$. As a consequence, $\varphi$ is injective. Since~\eqref{eq:nullite_rotation} states that the image of  $\Psi_{G,\pi} - \Psi_{G,\pi}(\cdot O)$ by $\varphi$ is the null polynomial, this implies~\eqref{eq:equality_rotation}. Since this holds for any rotation $O$, we have proved that $\Psi_{G,\pi}$ is invariant by rotation. 

Let us now turn to $\overline{\Psi}_{G,\pi}$. In the definition \eqref{eq:psi_bar} of $\overline{\psi}_{{\beta_{\pi^{-1}(i),j},G}}$, $\overline{\psi}_{{\beta_{\pi^{-1}(i),j},G}}= \psi_{{\beta_{\pi^{-1}(i),j},G}}$ unless the degree of $\pi^{-1}(i)$ in $G$ is exactly equal to $2$  (and $i\notin\{i,2\}$). In this case, $\overline{\psi}_{\beta_{\pi^{-1}(i),j},G}$ has been defined in such a what that $\mathbb{E}[\overline{\psi}_{\beta_{\pi^{-1}(i),j},G}[x+ Z'_{\pi^{-1}(i)j}]]= x^{\beta_{\pi^{-1}(i),j}}$ where $Z'_{\pi^{-1}(i),j}$ is mean normal variable with variance $1+\Delta^2/K$. Then, let us work conditionally to $X$ and define the centered Gaussian matrix $Z'$ such that the variance of $Z'_{i,j}$ is one unless the degree $\pi^{-1}(i)$ equals $2$ and $i\notin\{i,2\}$, in which case, the variance is equal to $1+\Delta^2/K$. Then, we have, similarly to~\eqref{eq:nullite_rotation}, that for any orthogonal matrix $O$. 
\[
\mathbb{E}[\overline{\Psi}_{G,\pi}(X+Z')- \overline{\Psi}_{G,\pi}((X+Z')\cdot O)|X]= P_{G,\pi}[X] - P_{G,\pi}[X\cdot O]=0
\]
By arguing as previously, but with the different morphism $\varphi'$ such that $\varphi'(P)[X]= \mathbb{E}[P(X+Z')]$, we again conclude that 
$\overline{\Psi}_{G,\pi}(Y)= \overline{\Psi}_{G,\pi}(Y\cdot O)$ for any rotation $O$. We have proved that $\overline{\Psi}[G]$ satisfies the desired invariance properties.

To show that the family $\overline{\Psi}_G$ spans the space $\mathcal{P}^{\mathrm{\mathrm{inv}}}_{\leq 2D}$, it suffices to prove that any $P_G$ expresses as a linear combination of $\overline{\Psi}_{G}$. This is done by induction on the total degree of the polynomial. From the definition of $\overline{\psi}$, we know  that the total degree $Q_G= \overline{\Psi}_{G}- P_{G}$ is smaller than that of $P_{G}$. Then, $Q_G$ is a polynomial generated by linear combination of $P_{G'}$ with total degree smaller than that of $G$, which by induction hypothesis, is spanned by the $\overline{\Psi}_{G}$'s. The result follows. 
\end{proof}

\section{Proofs of the moment bounds from Section~\ref{sec:mom}}\label{sec:proof:moment}

\subsection{Proof of Lemma~\ref{lem:first_moment_tilde_psi}}

    The case where $G$ only has only two isolated nodes is a consequence of the definition. Let us turn to the general case. 
    \[
    \widetilde{\Psi}_{G,\pi}= \prod_{l=1}^c \left[\overline{\Psi}_{G^*_l,\pi} - \mathbb{E}[\overline{\Psi}_{G^*_l,\pi} ]\right] \enspace . 
    \]
    Conditionally to the $\mu_k$'s, the random variables $\overline{\Psi}_{G^*_l,\pi}$ are independent. Thus, 
    \[
    \widetilde{\Psi}_{G,\pi}=\mathbb{E}\left[\prod_{l=1}^c \mathbb{E}[\overline{\Psi}_{G^*_l,\pi}|\mu] - \mathbb{E}[\overline{\Psi}_{G^*_l,\pi}] \right]\ . 
    \]
    Since $\overline{\Psi}_{G^*_l,\pi}$ is invariant by rotation on $Y$ in $\mathbb{R}^d$, we have for any orthogonal matrix $O$ that 
    $\overline{\Psi}_{G^*_l,\pi}(Y)= \overline{\Psi}_{G^*_l,\pi}(Y\cdot O)$. Besides the conditional distribution  $Y\cdot O$ to $\mu$ is the same as the conditional distribution to $\mu\cdot O$, which corresponds to the orthogonal transformation $O$ of all $\mu_k$. Hence,
    $\mathbb{E}[\overline{\Psi}_{G^*_l,\pi}|\mu]$ is invariant by any orthogonal transformation of $\mu$ and therefore does not depend on the value of $\mu$. We have therefore $\mathbb{E}[\overline{\Psi}_{G^*_l,\pi}|\mu]= \mathbb{E}[\overline{\Psi}_{G^*_l,\pi}]$, which, in turn, implies that $\mathbb{E}[\widetilde{\Psi}_{G,\pi}]=0$.

    \subsection{Proof of Lemma~\ref{lem:first_moment_degreeuneven}}

        First, we show that, for any $\pi$, we have $\mathbb{E}[x\overline{\Psi}_{G,\pi}]=0$. Let $v$ be any node with odd degree. Define $\tilde{Y}\in \mathbb{R}^{n\times d}$ by $\tilde{Y}_{i,j}= Y_{i,j}$ if $i\neq \pi(v)$ and $\tilde{Y}_{i,j}= -Y_{i,j}$ if $i= \pi(v)$. This amounts to reversing $b_{\pi(v)}$ and the noise $(Z_{\pi(v),j})_{j=1,\ldots,d}$. Note that reversing $b_{\pi(v)}$ does not change the functional $x$. Since the noise distribution is symmetric and since the $b_i$'s are Rademacher variables, $\tilde{Y}$ has the same distribution as $Y$. Since the degree of $v$ is odd, we have $\overline{\Psi}_{G,\pi}(Y)= -\overline{\Psi}_{G,\pi}(\tilde{Y})$. Hence, we conclude that 
        \[
        \mathbb{E}[x\overline{\Psi}_{G,\pi}(Y)] = - \mathbb{E}[x\overline{\Psi}_{G,\pi}(\tilde{Y})] =- \mathbb{E}[x\overline{\Psi}_{G,\pi}(Y)] =0\ .
        \]
        The first result follows. Similarly, we prove also that  $\mathbb{E}[\widetilde{\Psi}_{G},\pi]=0$.
        Let us turn to $\widetilde{\Psi}_{G}$. Decompose $G$ into its connected components $(G_1,\ldots, G_c)$.
        Without loss of generality, we assume that the first connected component $G_1$  contains a node with odd degree. Since we have proved that $\mathbb{E}[\overline{\Psi}_{G^*_1}]= 0$, we deduce that 
        \[
        \widetilde{\Psi}_{G,\pi}= \overline{\Psi}_{G^*_1,\pi}\prod_{l=2}^c\left[\overline{\Psi}_{G^*_l,\pi} - \mathbb{E}[\overline{\Psi}_{G^*_l,\pi}] \right]\ . 
        \]
        As previously, we transform $Y$ into $\tilde{Y}$ by reversing the row $\pi(v)$ so that $x\widetilde{\Psi}_{G,\pi}(Y)= - x\widetilde{\Psi}_{G,\pi}(\tilde{Y})$. By distribution invariance, the expectation is therefore again equal to zero.

\subsection{Proof of Lemma~\ref{lem:first_moment_degree2}}

Let us denote $v$ any such node of degree $2$, which is distinct from $v_1$ or $v_2$. We first prove that  $\mathbb{E}[x\overline{\Psi}_{G}]=0$, the result for $\overline{\Psi}_{G}$ being shown analogously. 
We shall prove the stronger result $\mathbb{E}[x\overline{\Psi}_{G,\pi}|(\mu, k^*(i)_{i\neq \pi(v)}) ]=0$,   $(\mu, k^*(i)_{i\neq \pi(v)})$ almost surely. For that, we expand $\overline{\Psi}_{G,\pi}$ over the features $\underline{j}:= (j_1,\ldots, j_{|E|})$, that is $\overline{\Psi}_{G,\pi}= \sum_{\underline{j}}
\overline{\Psi}_{G,\pi,\underline{j}}$ where, as in~\eqref{eq:Hermite_graph}, $\overline{\Psi}_{G,\pi,\underline{j}}= \overline{\psi}\left[\prod_{e\in E} Y_{\pi(l(e)),j_e}Y_{\pi(r(e)),j_e}\right]$. We shall show that, for   each of these terms, we have  $\mathbb{E}\left[x\overline{\Psi}_{G,\pi,\underline{j}}|(\mu, k^*(i)_{i\neq \pi(v)})\right]=0$. Let us denote $s$ and $s'$ the indices of the edges that are incident to $v$ --note that $s'$ and $s$ can be identical if $v$ has a self-edge. The random variable $\overline{\Psi}_{G,\pi,\underline{j}}$ is then of the form $(Y_{\pi(v),j_{s}}Y_{\pi(v),j_{s'}}- (1+\frac{\Delta^2}{d})\mathbf{1}\{j_{s}=j_{s'}\}) H((Y_l)_{(l\neq \pi(v))})$ for some function $H$. As a consequence, if we integrate $x\overline{\Psi}_{G,\pi,\underline{j}}$ with respect to $Z_{\pi(v)}$, we arrive at $[\mu_{k^*(\pi(v)),j_s}\mu_{k^*(\pi(v)),j'_s}- \frac{\Delta^2}{d}\mathbf{1}\{j_{s}=j_{s'}\}]xH((Y_i)_{(i\neq \pi(v))})$. Also, only the left-hand side term of the latter expression actually depends on $k^*(\pi(v))$. Integrating with respect to $k^{*}(\pi(v))$ (while conditioning on all the rest) we arrive at 
\[
\mathbb{E}\left[\overline{\Psi}_{G,\pi,\underline{j}}|\mu, (Y_l)_{(l\neq \pi(v))},(k^*(l))_{(l\neq \pi(v))} \right] = \left[\frac{1}{K}\sum_{k=1}^K \mu_k \mu_k^T - \frac{\Delta^2}{d} I_d\right]_{j_{s},j_{s}'}x H((Y_l)_{(l\neq \pi(v))}) \ , 
\]
Since the $\mu_k$'s form almost surely an orthogonal basis in $\mathbb{R}^d$ and their norm is almost surely equal $\Delta$, the matrix $\frac{1}{K}\sum_{k=1}^K \mu_k \mu_k^T - \frac{\Delta^2}{d} I_d$ is exactly equal to zero and the result follows.% for $\overline{\Psi}_G$. The result also follows straightforwardly for $\widetilde{\Psi}_G$ by definition.

We prove analogously that $\mathbb{E}[x\widetilde{\Psi}_{G,\pi}|(\mu, k^*(i)_{i\neq \pi(v)}) ]=0$, which concludes the proof. 

\subsection{Proof of Lemma~\ref{lem:first_moment}}

    As $G$ does not contain any node of degree $2$ aside from $1$ and $2$, it follows from standard properties of Hermite polynomials that 
    \begin{equation}\label{eq}
    \mathbb{E}[x\overline{\Psi}_{G,\pi}|k^*,\mu]= xP_{G,\pi}([\mu_{k^*(1)};\mu_{k^*(2)};\ldots;\mu_{k^*(n)}]^T )\ , 
    \end{equation}
    where $P_{G,\pi}([\mu_{k^*(1)};\mu_{k^*(2)};\ldots;\mu_{k^*(n)}]^T)$ stands for the polynomial $P_{G,\pi}$ applied to the $n\times d$ matrix ~\\ $[\mu_{k^*(1)};\mu_{k^*(2)};\ldots;\mu_{k^*(n)}]^T$. This leads us to 
    $\mathbb{E}[x\overline{\Psi}_{G,\pi}]= \mathbb{E}\left[x P_{G,\pi}([\mu_{k^*(1)};\mu_{k^*(2)};\ldots;\mu_{k^*(n)}]^T )\right]$
    and 
    \[
    \mathbb{E}[\overline{\Psi}_{G}]= \mathbb{E}\left[P_{G}([\mu_{k^*(1)};\mu_{k^*(2)};\ldots;\mu_{k^*(n)}]^T )\right]\enspace . 
    \]
    Since the $\mu_k$'s are orthogonal, $P_{G}([\mu_{k^*(1)};\mu_{k^*(2)};\ldots;\mu_{k^*(n)}]^T)$ is non-zero if and only we have  $k^*(i)= k^*(j)$ whenever $\pi^{-1}(\{i\})$ and $\pi^{-1}(\{j\})$ belong to the same connected component in $G$. The probability of that event is $K^{-|V|+|\mathrm{CC}(G)|}$. Since $\langle \mu_{k^*(i)},\mu_{k^*(i)}\rangle= \Delta^2$, this leads us to 
    \[
    \mathbb{E}[\overline{\Psi}_G]= |\Pi_V| \Delta^{2|E|}\frac{1}{K^{|V|-|\mathrm{CC}(G)|}}\enspace .
    \]
    Turning to $x\overline{\Psi}_G$, we observe $xP_{G}([\mu_{k^*(1)};\mu_{k^*(2)};\ldots;\mu_{k^*(n)}]^T )$ is non-zero if and only we have  $k^*(i)= k^*(j)$ whenever $\pi^{-1}(\{i\})$ and we have $k^*(1)=k^*(2)$. The second result follows.

    \subsection{Proof of Corollary~\ref{cor:first_moment}}
        The case where $G$ has no edge is simple as $\widetilde{\Psi}_G=1$. 
        If there exists a connected component $G_1$ in $G$ that does not contain both nodes  $v_1$ and $v_2$ then $\mathbb{E}[x\widetilde{\Psi}_{G,\pi}]= 0$. Indeed, conditionally to $\mu$, $(\overline{\Psi}_{G_1,\pi}-\mathbb{E}[\overline{\Psi}_{G_1,\pi}])$ is centered and is independent of $x$ and of the other $(\overline{\Psi}_{G_l,\pi}-\mathbb{E}[\overline{\Psi}_{G_k,\pi}])$. We have dealt with the case where $G$ contains at least a degree $2$ node (aside from $v_1$ and $v_2$) in Lemma~\ref{lem:first_moment_degree2}. 
        
        Hence, we can restrict ourselves to the cases where $G$ is connected and all the nodes of $G$, to the possible exception of $v_1$ and $v_2$, have a degree larger or equal to $4$. Then, by definition of $\widetilde{\Psi}_G$, we have 
        \[
        \mathbb{E}[x\widetilde{\Psi}_G]= \mathbb{E}[x\overline{\Psi}_G]- \mathbb{E}[x]\mathbb{E}[\overline{\Psi}_G]\ . 
        \]
        The result then follows from Lemma~\ref{lem:first_moment} and from the fact that $\overline{G}=G$.

    \subsection{Proof of Lemma~\ref{lem:covariance_degree_odd}}
        The proof is similar to that of  Lemma~\ref{lem:first_moment_degreeuneven}. Denote $v$ a node with odd degree in $G^{(1)}$ and define the transformation $\tilde{Y}$ of $Y$ by reversing the row $\pi^{(1)}(v)$. Then, $\tilde{Y}$ and $Y$ have similar distributions. Since, if it exists, the degree of $(\pi^{(2)})^{-1}(\{v\})$ is even, we deduce that $\overline{\Psi}_{G^{(1)},\pi^{(1)}}(Y) \overline{\Psi}_{G^{(2)},\pi^{(2)}}(Y)= -\overline{\Psi}_{G^{(1)},\pi^{(1)}}(\tilde{Y}) \overline{\Psi}_{G^{(2)},\pi^{(2)}}(\tilde{Y})$. The first result follows by distribution invariance. The second result is proved analogously

    \subsection{Proof of Lemma~\ref{lem:covariance_degree2}}

    The proof of this lemma is analogous to that of Lemma~\ref{lem:first_moment_degree2}. We only prove that  $\mathbb{E}[\overline{\Psi}_{G^{(1)},\pi^{(1)}} \overline{\Psi}_{G^{(2)},\pi^{(2)}}]=0$, the second result being similar. For that, we establish the stronger result 
    \[
    \mathbb{E}[\overline{\Psi}_{G^{(1)},\pi^{(1)}} \overline{\Psi}_{G^{(2)},\pi^{(2)}}|(\mu, k^*(i)_{i\neq \pi^{(1)}(v)}), (Y_{i})_{i\neq \pi^{(1)(v)}} ]=0
    \]
    $(\mu, k^*(i)_{i\neq \pi^{(1)}(v)}, (Y_{i})_{i\neq \pi^{(1)(v)}})$ almost surely. We expand $\overline{\Psi}_{G^{(1)},\pi^{(1)}}$ over the features $\underline{j}:= (j_1,\ldots, j_{|E|})$, that is $\overline{\Psi}_{G^{(1)},\pi^{(1)}}= \sum_{\underline{j}}
\overline{\Psi}_{G^{(1)},\pi^{(1)},\underline{j}}$ where, as in~\eqref{eq:Hermite_graph}, $\overline{\Psi}_{G^{(1)},\pi^{(1)},\underline{j}}= \overline{\psi}\left[\prod_{e\in E^{(1)}} Y_{\pi^{(1)}(l(e)),j_e}Y_{\pi^{(1)}(r(e)),j_e}\right]$. Let us denote $s$ and $s'$ the indices of the edges that are incident to $v$ --note that $s'$ and $s$ can be identical if $v$ has a self-edge. The random variable $\overline{\Psi}_{G^{(1)},\pi^{(1)},\underline{j}}$ is then of the form $(Y_{\pi^{(1)}(v),j_{s}}Y_{\pi^{(1)}(v),j_{s'}}- (1+\frac{\Delta^2}{d})\mathbf{1}\{j_{s}=j_{s'}\}) H((Y_l)_{(l\neq \pi^{(1)}(v))})$ for some function $H$.
Integrating the latter with respect to $E_{\pi^{(1)}(v)}$, we arrive at $[\mu_{k^*(\pi^{(1)}(v)),j_s}\mu_{k^*(\pi^{(1)}(v)),j'_s}- \frac{\Delta^2}{d}\mathbf{1}\{j_{s}=j_{s'}\}]H((Y_i)_{(i\neq \pi^{(1)}(v))}) \overline{\Psi}_{G^{(2)},\pi^{(2)}}$. Also, only the left-hand side term of the latter expression actually depends on $k^*(\pi^{(1)}(v))$. Integrating with respect to $k^{*}(\pi^{(1)}(v))$ while conditioning on all the rest, we arrive at 
\[
\mathbb{E}\left[\overline{\Psi}_{G^{(1)},\pi^{(1)}} \overline{\Psi}_{G^{(2)},\pi^{(2)}}|\mu, (Y_l)_{(l\neq \pi^{(1)}(v))},(k^*(l))_{(l\neq \pi^{(1)}(v))} \right] = \left[\frac{1}{K}\sum_{k=1}^K \mu_k \mu_k^T - \frac{\Delta^2}{d} I_d\right]_{j_{s},j_{s}'}H((Y_l)_{(l\neq \pi^{(1)}(v))})\overline{\Psi}_{G^{(2)},\pi^{(2)}} \ , 
\]
Since the $\mu_k$'s form almost surely an orthogonal basis in $\mathbb{R}^d$, the matrix $\frac{1}{K}\sum_{k=1}^K \mu_k \mu_k^T - \frac{\Delta^2}{d} I_d$ is exactly equal to zero and the result follows.

            \subsection{Proof of Proposition~\ref{prop:covariance_without_degree_2}}

            Without loss of generality, we assume that $\pi^{(1)}(v^{(1)}_i)=\pi^{(2)}(v^{(2)}_i)$ for $i=1,\ldots, r_0$ and $\pi^{(1)}(v_i)\cap \mathrm{Im}(\pi^{(2)})=\emptyset$ for $i>r_0$ and $\pi^{(2)}(v_i)\cap \mathrm{Im}(\pi^{(1)})=\emptyset$ for $i\geq r_0$, where $\mathrm{Im}(\pi^{(a)})$ stands for the image of $\pi^{(a)}$. With our notation, this corresponds to the case where $\mathbf{M}=\{(v_1^{(1)},v_1^{(2)}),\ldots, (v_{r_0}^{(1)},v_{r_0}^{(2)})\}$. We start by computing the conditional expectation with respect to $k^*(.)$ and to $\mu_k$'s. We  use the decomposition $\overline{\Psi}_{G^{(1)},\pi^{(1)}}= \sum_{\underline{j}^{(1)}}\overline{\Psi}_{G^{(1)},\pi^{(1)},\underline{j}^{(1)}}$ and $\overline{\Psi}_{G^{(2)},\pi^{(2)}}= \sum_{\underline{j}^{(2)}}\overline{\Psi}_{G^{(2)},\pi^{(2)},\underline{j}^{(2)}}$, where, for $a=1,2$, 
                \[
                \overline{\Psi}_{G^{(a)},\pi^{(a)},\underline{j}^{(a)}} = \overline{\psi}\left[\prod_{e\in E^{(a)}} Y_{\pi^{(a)}(l(e)),j^{(a)}_e}Y_{\pi^{(a)}(r(e)),j^{(a)}_e}  \right]\ . 
                \]
            It is more convenient here to work with half-edges instead of the edges. Given a half-edge $\underline{e}\in E^{\mathrm{half},a}$, we define $j^{(a)}_{\underline{e}}$ to be equal to $j^{(a)}_{e}$ where $e$ is the half-edge that correspond to $\underline{e}$. 
            Then, we define 
            \[
            \omega_{\underline{j}^{(1)},\underline{j}^{(2)}}= \prod_{\underline{e}\in E^{\mathrm{half},(1)}}\mu_{k^*(i(\underline{e})),j^{(1)}_{\underline{e}}} \prod_{\underline{e}'\in E^{\mathrm{half},(2)}} \mu_{k^*(i(\underline{e}')),j^{(2)}_{\underline{e}'}}\ . 
            \]
            
            For $i= 1,\ldots, r_0$, we rely on the decomposition~\eqref{eq:recursion_hermite} of Hermite polynomials. Then, for any $i\in[r_0]$, any $l\in [d]$, we introduce $\mathcal{S}^{(1)}_{i,l}$ as the set of half-edges $\underline{e}\in E^{\mathrm{half},(1)}$ that are incident to $v_i$ and such that $j^{(1)}_{\underline{e}}=l$. We define $\mathcal{S}^{(2)}_{i,l}$ similarly. In the following, we sum over all sets $(S^{(1)}_{i,l},S^{(1)}_{i,l})$ in $\mathcal{S}^{(1)}_{i,l}\times \mathcal{S}^{(2)}_{i,l}$ that have \emph{the same size}. Then, it follows from~\eqref{eq:recursion_hermite} that 
            \begin{eqnarray} \nonumber
            \lefteqn{    \omega_{\underline{j}^{(1)},\underline{j}^{(2)}}^{-1}  \mathbb{E}[\overline{\Psi}_{G^{(1)},\pi^{(1)},\underline{j}^{(1)}}\overline{\Psi}_{G^{(2)},\pi^{(2)},\underline{j}^{(2)}}|k^*,\mu]} &&\\ & = \sum_{(S^{(1)}_{1,1},S^{(2)}_{1,1},\ldots),\ldots, (S^{(1)}_{r_0,1},S^{(2)}_{r_0,1},\ldots)}\left(\prod_{i=1}^{r_0}\prod_{l=1}^d |S^{(1)}_{i,l}|!\right) \prod_{i\in [r_0]}\prod_{l=1}^d 
            \mu^{-2|S^{(1)}_{i,l}|}_{k^*(\pi^{(1)}(v_i)),l}
            %\prod_{e\in S^{(1)}_{i,l}}\mu_{k^*(\pi^{(1)}(v_i)),j^{(1)}_e}\prod_{e\in S^{(2)}_{i,l}}\mu_{k^*(\pi^{(2)}(v_i)),j^{(2)}_e}
            \ . \label{eq:cross1}
            \end{eqnarray}

            \begin{lemma}\label{lem:claim}
                
                Using the same notation as in the statement of the proposition, we have 
                \begin{eqnarray}\label{eq:cross12}
                    \mathbb{E}\left[\overline{\Psi}_{G^{(1)},\pi^{(1)}} \overline{\Psi}_{G^{(2)},\pi^{(2)}}\big|(k^*,\mu)\right]&=& \sum_{\mathbf{P} \in \mathcal{P}[\mathbf{M}]} \omega_{\mathbf{P} } \ ; \\
                    \omega_{\mathbf{P} }&:= &   \sum_{\underline{j}^{(1)}}\sum_{\underline{j}^{(2)}}\omega_{\underline{j}^{(1)},\underline{j}^{(2)}}\prod_{(\underline{e},\underline{e'} )\in\mathbf{P} }\mu^{-1}_{k^*(\pi^{(1)}(i(\underline{e}))),j^{(1)}_{\underline{e}}}\mu^{-1}_{k^*(\pi^{(2)}(i(\underline{e}'))),j^{(2)}_{\underline{e}'}} \ , 
                \end{eqnarray}
                where the sum over $\underline{j}^{(1)}$ and $\underline{j}^{(2)}$ is such that feature $j^{(1)}_{\underline{e}}$ and $j^{(2)}_{\underline{e}'}$ are the same whenever $(\underline{e},\underline{e}')\in \mathbf{P}$. 
                %For any $i\in [1:r_0]$, we write $d_i= \min(\mathrm{deg}_G(i),\mathrm{deg}_{G'}(i))$. 
                
                %In what follows $T_i$ is a matching between $t_i$ edges incident to $i$ in $G$ and $t_i$ edges incident to $i$ in $G'$.
                %\begin{eqnarray}
                 %   U=  \sum_{i=1}^{r_0}\sum_{t_i=0}^{r_i} \sum_{T_1,\ldots, T_{r_0}}\omega_{T_1,\ldots, T_{r_0}}\\
                 %   \omega_{T_1,\ldots, T_{r_0}}:= \sum_{j,j'} \prod_{s,s'} 
                 %   \mu_{k^*(\pi(i_s)),j_s}^{\mathbf{1}_{s\notin T_{i_s}}}
                 %   \mu_{k^*(\pi(i'_s)),j'_s}^{\mathbf{1}_{s\notin T_{i'_s}}}
                 %   \mu_{k^*(\pi'(i_{s'})),j_{s'}}^{\mathbf{1}_{s'\notin T_{i_{s'}}}}
                 %   \mu_{k^*(\pi'(i'_{s'})),j'_{s'}}^{\mathbf{1}_{s'\notin T'_{i'_{s'}}}} \ , 
                %\end{eqnarray}
                %where the sum over $j$ and $j'$ only runs such that matched half-edges tha share the same features.
                
               \end{lemma}

            It remains to quantify $\omega_{\mathbf{P}}$. For that purpose, we build, as described previously, the labelled graph $G_{\Delta}=(V_{\Delta},E_{\Delta})$ obtained from $G^{(1)}$ and $G^{(2)}$. Let us recall its construction to clarify the arguments. 
            First, we merge $G^{(1)}$ and $G^{(2)}$ to create $G_{\cup}[\mathbf{M}, \mathbf{P}]$. We add a distinct color to each of the half- edges that are paired. We denote $\mathrm{Cyc}$ the set of cycles induced by these colored edges. Next, for any maximum open path induced by colored half-edges,  we remove all these half-edges and replace them by an edge between the two extremities of these open paths. Coming back to the definition of $  \omega_{\mathbf{P}}$, we observe that paired half-edges $\underline{e}$ and $\underline{e}'$ constrain that $j^{(1)}_{e}=j^{(2)}_{e'}$, for the corresponding edges. Consider a cycle in $\mathrm{Cyc}$. Then, all the edges $e\in E^{(1)}$ or $e'\in E^{(2)}$ involved in the cycle share the same features $j^{(1)}_e$ or $j^{(2)}_{e'}$. Similarly, for an open colored path that has been transformed into an edge in $G_{\Delta}$, we have the constrain that the features of the extremities are identical. As a consequence, for a given $\mathbf{P}$, if we sum over all possible sequence $\underline{j}_1$ and $\underline{j}_2$, we obtain that  
            \begin{eqnarray*}
                \omega_{\mathbf{P}} &:= d^{|\mathrm{Cyc}|}\prod_{e\in E_{\Delta}} \langle \mu_{k^*(l(e))},  \mu_{k^*(r(e))}\rangle\ .
            \end{eqnarray*}
            Hence, we have 
            \begin{align}\label{eq:upper_cross_variance:conditional}
               \mathbb{E}\left[\overline{\Psi}_{G^{(1)},\pi^{(1)}} \overline{\Psi}_{G^{(2)},\pi^{(2)}}|(k^* ,\mu) \right]=  \sum_{\mathbf{P} \in \mathcal{P}[\mathbf{M}]} 
                d^{|\mathrm{Cyc}|} \mathbb{E}\left[\prod_{e\in E_{\Delta}} \langle \mu_{k^*(l(e))},  \mu_{k^*(r(e))}\rangle \big|(k^* ,\mu) \right]\ .
            \end{align}
            To conclude, we observe that $\langle \mu_{k^*(l(e))},  \mu_{k^*(r(e))}\rangle=\Delta^2$ if $k^*(l(e))= k^*(r(e))$ and is zero otherwise. Hence, 
            $\mathbb{E}\left[\prod_{e\in E_{\Delta}} \langle \mu_{k^*(l(e))},  \mu_{k^*(r(e))}\rangle \right]$ is equal to $\Delta^{2|E_{\Delta}|}$ times the probability that, inside each connected component of $G_{\Delta}$, the groups are the same. The result follows.

            \begin{proof}[Proof of Lemma~\ref{lem:claim}]
                We start from~\eqref{eq:cross12} and we will show that it equals~\eqref{eq:cross1}. In~\eqref{eq:cross12}, we first exchange the summation orders and we first fix  $\underline{j}^{(1)}$ and $\underline{j}^{(2)}$. 
            Then, we have  
            \[
             \sum_{\mathbf{P} \in \mathcal{P}[\mathbf{M}]} \omega_{\mathbf{P} } = \sum_{\underline{j}^{(1)}, \underline{j}^{(2)}} \sum_{\mathbf{P}} \omega_{\underline{j}^{(1)}\underline{j}^{(2)}}\left[\prod_{(\underline{e},\underline{e'}) \in \mathbf{P}} \mu_{k^*(\pi^{(1)}(i(\underline{e}))),j^{(1)}_{\underline{e}}}\mu_{k^*(\pi^{(2)}(i(\underline{e}'))),j^{(2)}_{\underline{e}} }\right]^{-1}
            \] 
            where the sum run over all $\mathbf{P}$ that are compatible with $\underline{j}^{(1)}$ and $\underline{j}^{(2)}$ in the sense that  $j^{(1)}_{\underline{e}}$ and $j^{(2)}_{\underline{e}'}$ are the same whenever $(\underline{e},\underline{e}')\in \mathbf{P}$. Hence, we only  have to prove that 
            \[
            \sum_{\mathbf{P}} \left[\prod_{(\underline{e},\underline{e'}) \in \mathbf{P}} \mu_{k^*(\pi^{(1)}(i(\underline{e}))),j^{(1)}_{\underline{e}}}\mu_{k^*(\pi^{(2)}(i(\underline{e}'))),j^{(2)}_{\underline{e}} }\right]^{-1} = \sum_{(S^{(1)}_{1,1},S^{(2)}_{1,1},\ldots)}\left(\prod_{i=1}^{r_0}\prod_{l=1}^d |S^{(1)}_{i,l}|!\right) \prod_{i\in [r_0]}\prod_{l=1}^d 
            \mu^{-2|S^{(1)}_{i,l}|}_{k^*(\pi^{(1)}(v_i)),l}\enspace ,
            \]
            where we recall that the sum runs over all sets $(S^{(1)}_{i,l},S^{(1)}_{i,l})$ in $\mathcal{S}^{(1)}_{i,l}\times \mathcal{S}^{(2)}_{i,l}$ that have \emph{the same size}. 

            To show the latter equality, it suffices to fix $(S^{(1)}_{i,l},S^{(2)}_{i,l}), i=1,\ldots, r_0$ and $l=1,\ldots, d$. The pairings $\mathbf{P}$ corresponding to those $(S^{(1)}_{i,l},S^{(2)}_{i,l})$ are such that any half-edge in $S^{(1)}_{i,l}$ is paired to any half-edge of $S^{(1)}_{i,l}$. As a consequence, we have $\prod_{i=1}^{r_0}\prod_{l=1}^d |S^{(1)}_{i,l}|!$ such possible pairings. Since, for such a pairing $\mathbf{P}$, the product $\prod_{(\underline{e},\underline{e'}) \in \mathbf{P}} \mu_{k^*(\pi^{(1)}(i(\underline{e}))),j^{(1)}_{\underline{e}}}\mu_{k^*(\pi^{(2)}(i(\underline{e}'))),j^{(2)}_{\underline{e}} }$ equals $ \prod_{i\in [r_0]}\prod_{l=1}^d 
            \mu^{2|S^{(1)}_{i,l}|}_{k^*(\pi^{(1)}(v_i)),l}$, the result follows. 
            % of the $\mathbf{P}_i=(T_i^{(1)},T_i^{(2)})$ such that the multi-set $M_i^{(1)}=\{j_e, e\in T_i^{(1)}\}$ and $M_i^{(2)}:=\{j_e, e\in R_i^{(2)}\}$ are equal.
            
            %Then, we sum terms of the form 
            % $\prod_{i=1}^{r_0}\omega_{i\underline{j}^{(1)}\underline{j}^{(2)}}\left[\prod_{(\underline{e},\underline{e'}) \in \mathbf{P}_{i}} \mu_{k^*(\pi(v_i)),j^{(*)}_e} \right]^{-1}$ of the $\mathbf{P}_i=(T_i^{(1)},T_i^{(2)})$ such that the multi-set $M_i^{(1)}=\{j_e, e\in T_i^{(1)}\}$ and $M_i^{(2)}:=\{j_e, e\in R_i^{(2)}\}$ are equal.
             
            % The number of $(T_i)$ that are associated to a specific expression  is the the following: for each distinct value of $j$ in $M_i^{(1)}$, we consider all subset of edges $e$ in $G^{(1)}$  incidents to $i$ such that $j^{(1)}_e=j$ with size equal to the multiplicity $m_j$ of $j$ in $M_i^{(1)}$, we consider also  all subsets of edges $e$ in  $G^{(2)}$ that are incident to $i$ such that $j_e=j$ with size equal to the multiplicity $m_j$ of $j$ in $M_i^{(2)}$ and finally  we multiply this by $m_j!$ to count all such matching. In view of~\eqref{eq:cross1}.
            \end{proof}

\subsection{Proof of Proposition~\ref{prop:covariance_general}}
    
    Let $v^{(1)}$ be any node of degree $2$ of $G^{(1)}$ which is distinct from $v^{(1)}_1$ and $v^{(2)}_2$. Let $v^{(2)}$ be the node in $G^{(2)}$ such that $\pi^{(1)}(v^{(1)})=\pi^{(2)}(v^{(2)})$. To start with, we assume that $\mathrm{deg}_{G^{(2)}}(v^{(2)})>2$. Without loss of generality, we write $\underline{e}$ and $\underline{e}'$ for the two half-edges incident to $v^{(1)}$ in $G^{(1)}$.    
    Given some vectors $\underline{j}^{(1)}$, $\underline{j}^{(2)}$, let us compute the term where  $Y_{\pi^{(1)}(v^{(1)})}$ arises in  $\mathbb{E}[\overline{\Psi}_{G^{(1)},\pi^{(1)},\underline{j}^{(1)}}\overline{\Psi}_{G^{(2)},\pi^{(2)},\underline{j}^{(2)}}|k^*,\mu]$. If $j^{(1)}_{\underline{e}}\neq j^{(1)}_{\underline{e}'}$, we apply Lemma~\ref{lem:hermite} and the expression is identical to that in the proof of Proposition~\ref{prop:covariance_without_degree_2}. 
    \[
    \sum_{S^{(1)}_{\pi^{(1)}(v),j^{(1)}_{\underline{e}}},S^{(2)}_{\pi^{(1)}(v^{(1)}),j^{(1)}_{\underline{e}}}, S^{(1)}_{\pi^{(1)}(v^{(1)}),j^{(1)}_{\underline{e}'}}, S^{(2)}_{\pi^{(1)}(v^{(1)}),j^{(1)}_{\underline{e}'}}}|S^{(1)}_{\pi^{(1)}(v^{(1)}),j^{(1)}_{\underline{e}}}|!|S^{(1)}_{\pi_1(v^{(1)}),j^{(1)}_{\underline{e}'}}|! \prod_{\underline{e''}} \mu_{k^{*}(\pi^{(1)}(v^{(1)})),j^{(*)}_{\underline{e}''}} \ , 
    \]
    where the product $\prod_{\underline{e}''}$ runs over half-edges in $E^{\mathrm{half},(1)}$ and $E^{\mathrm{half},(2)}$ that are incident to $v^{(1)}$ or $v^{(2)}$ and do not belong to the $ S^{(*)}_{\pi^{(1)}(v^{(1)}),j^{(1)}_{\underline{e}}}$ or  $S^{(*)}_{\pi^{(1)}(v^{(1)}),j^{(1)}_{\underline{e}'}}$ for $*=1,2$ and $j^{(*)}$ equals $j^{(1)}$ or $j^{(2)}$. 
    If $j^{(1)}_{\underline{e}}= j^{(1)}_{\underline{e}'}$, the expression is slightly different.
    \begin{eqnarray}\label{eq:e_second_moment}
    \left[\mu_{k^*(\pi(v^{(1)})),j^{(1)}_{\underline{e}}}^2 -\frac{\Delta^2}{K }\right]\prod_{\underline{e}''\in E^{\mathrm{half},(2)}(v^{(2)})}\mu_{k^{*}(\pi^{(1)}(v^{(1)})),j^{(2)}_{\underline{e}''}} + \sum_{(S^{(1)}_{\pi^{(1)}(v^{(1)}),j^{(1)}_{\underline{e}}}, S^{(2)}_{\pi^{(1)}(v^{(1)}),j^{(1)}_{\underline{e}}})\neq \emptyset}|S_{i,j^{(1)}_{e}}|!\prod_{\underline{e}''} \mu_{k^{*}(\pi^{(1)}(v^{(1)})),j^{(*)}_{\underline{e}''}}\ , 
    \end{eqnarray}
    where $E^{\mathrm{half},(2)}(v^{(2)})$ is the set of half-edges in $G^{(2)}$ that are incident to $v^{(2)}$ and where the product over $\underline{e}''$ is over half-edges not belonging $ S^{(1)}_{v^{(1)},j^{(1)}_{\underline{e}}}$ or $ S^{(2)}_{v^{(1)},j^{(1)}_{\underline{e}}}$. In comparison to nodes with degree $4$, the correction of degree $2$ nodes only arises in the case where $(S^{(1)}_{v^{(1)},j^{(1)}_{\underline{e}}}, S^{(2)}_{v^{(1)},j^{(1)}_{\underline{e}}})= \emptyset$. We have a similar expression when the degree of $v^{(2)}$ in $G^{(2)}$ is also equal to $2$.
    
    Next, we argue as in the proof of Proposition~\ref{prop:covariance_without_degree_2}. Consider any pairing $\mathbf{P}\in\mathcal{M}$. 
    Write $V_{2,np}\subset V^{(1)}\cup V^{(2)}$ the subset of degree $2$ nodes $v$ in $G^{(1)}$ and $G^{(2)}$ such that no half-edge incident to $v$  is involved in $\mathbf{P}$. Then, we define 
    $\overline{\Psi}^{(V_{2,np})}_{G^{(1)},\pi^{(1)},\underline{j}^{(1)}}$ similarly to $\overline{\Psi}_{G^{(1)},\pi^{(1)},\underline{j}^{(1)}}$ except that we do not make the correction in $\overline{\Psi}$ to degree $2$ nodes that are not $V_{2,np}$. Arguing as in the main claim in the proof of Proposition~\ref{prop:covariance_without_degree_2}, we arrive at 
    \begin{align}\label{eq:cross12_bis}
        \mathbb{E}[\overline{\Psi}_{G^{(1)},\pi^{(1)}} \overline{\Psi}_{G^{(2)},\pi^{(2)}}|(k^*,\mu)]&=  \sum_{\mathbf{P}\in \mathcal{P}} \underline{\omega}_{\mathbf{P}}\enspace ; \\
        \omega_{\underline{j}^{(1)},\underline{j}^{(2)};V_{2,np}}& := \mathbb{E}[\overline{\Psi}^{(V_{2,np})}_{G^{(1)},\pi^{(1)},\underline{j}^{(1)}} |(k^*,\mu)]\mathbb{E}[\overline{\Psi}^{(V_{2,np})}_{G^{(2)},\pi^{(2)},\underline{j}^{(2)}} |(k^*,\mu)]\enspace ; \\ 
        \underline{\omega}_{\mathbf{P}}&:=   \sum_{\underline{j}^{(1)}}\sum_{\underline{j}^{(2)}}\omega_{\underline{j}^{(1)},\underline{j}^{(2)};V_{2,np}}\prod_{(\underline{e},\underline{e'} )\in\mathbf{P} }\mu^{-1}_{k^*(\pi^{(1)}(i(\underline{e}))),j^{(1)}_{\underline{e}}}\mu^{-1}_{k^*(\pi^{(2)}(i(\underline{e}'))),j^{(2)}_{\underline{e}'}}\enspace , 
    \end{align}
    where the sum runs over $\underline{j}^{(1)}$ and $\underline{j}^{(2)}$ such that features $j^{(1)}_{\underline{e}}$ and $j^{(2)}_{\underline{e}'}$ are the same if $(\underline{e},\underline{e}')\in \mathbf{P}$. 
    It is clear that, whenever $V_{2,np}=\emptyset$, we have $\underline{\omega}_{\mathbf{P}}= \omega_{\mathbf{P}}$.
    
    \begin{lemma}\label{lem:terme_correctif}
    If $V_{2,np}\neq \emptyset$, we claim that 
    \[
    \left|   \mathbb{E}\left[ \underline{\omega}_{\mathbf{P}}\right]\right| \leq 2^{|V_{2,np}|}\mathbb{E}\left[\omega_{ \mathbf{P}}\right]\ . 
    \]
    \end{lemma}
    The result then follows by relying on the same bound $\mathbb{E}\left[\omega_{ \mathbf{P}}\right]$ as in the previ
    ous proof and by relying that $V_{2,np}=\emptyset$ whenever $\mathbf{P}\in \mathcal{P}_{\mathrm{full}}(\mathcal{M})$.  It remains to prove the claim.

    \begin{proof}[Proof of Lemma~\ref{lem:terme_correctif}]
    For any node $v$ in $V_{2,np}$, there is a correction term $-(\Delta^2/K) \mathbf{1}_{j^{(*)}_{\underline{e}}= j^{(*)}_{\underline{e}'}}$ where $(*)=(1),(2)$ depending on whether $v$ belongs to $V^{(1)}$ or $V^{(2)}$ and where $\underline{e}$ and $\underline{e}'$ are the two-half-edges incident to $v$. 
    Let us expand all these terms in $\underline{\omega}_{\mathbf{P}}$. For any $T\subseteq V_{2,np}$, we consider the term of $\underline{\omega}_{\mathbf{P}}$ where we consider the correction term 
    $-(\Delta^2/K) \mathbf{1}_{j^{(*)}_{\underline{e}}= j^{(*)}_{\underline{e}'}}$ for all $v\in T$ and, for any $v\notin T$ we consider the term 
    $\mu_{k^{*}(\pi^{(*)}(v)),j^{(*)}_{\underline{e}}} \mu_{k^{*}(\pi^{(*)}(v)),j^{(*)}_{\underline{e}'}}$ where $(*)=1,2$ depending whether $v\in V^{(1)}$ or $V^{(2)}$. In order to compute the corresponding quantity, we build a graph $G_{\Delta,T}=(V_{\Delta},E_{\Delta,T})$ from $G_{\Delta}$ as follows: first, for all $v\in T$, we color (with a color that depends on $v$) all half-edges incident to $v$ that were originally in the same graph $G^{(1)}$ or $G^{(2)}$. Then, we build $G_{\Delta,T}$ by removing all possible cycles formed by colored half-edges and removing all half-edges that formed an open path by connecting the two (non-colored) extremities of these open paths. Writing $|\mathrm{Cyc}(T)|$ the number of cycles that have been pruned when going from $G_{\Delta}$ to $G_{\Delta,T}$. Then, we have 
    \begin{eqnarray*}
        \underline{\omega}_{\mathbf{P}} &:= \sum_{T\subseteq V_{2,np}} (-1)^{|T|} d^{|\mathrm{Cyc}|} d^{|\mathrm{Cyc}(T)|} \frac{\Delta^{2|T|}}{K^{T}}\prod_{e\in E_{\Delta,T}} \langle \mu_{k^*(l(e))},  \mu_{k^*(r(e))}\rangle\ .
    \end{eqnarray*}
    Integrating with respect to $k^*(.)$ and using the orthogonality  of the $\mu_k$'s, we arrive at 
    \begin{eqnarray}\nonumber
    \mathbb{E}\left[\underline{\omega}_{\mathbf{P}}\right]&=& \sum_{T\subseteq V_{2,np}} (-1)^{|T|}  d^{|\mathrm{Cyc}(T)|} \frac{\Delta^{2|T|}}{K^{|T|}} \frac{\Delta^{2|E_{\Delta,T}|}}{K^{|V_{\Delta}|-|CC(G_{\Delta,T})|}} \\
    & = &   \frac{d^{|\mathrm{Cyc}(T)|} \Delta^{2|E_{\Delta}|}}{K^{|V_{\Delta}|-|CC(G_{\Delta})|}}\sum_{T\subseteq S} (-1)^{|T|}  \frac{d^{|\mathrm{Cyc}(T)|}}{K^{|T|+|CC(G_{\Delta})| -|CC(G_{\Delta,T})|}} \enspace . \label{eq:upper:corrective}
    \end{eqnarray}
    To conclude, we use that $d\leq K$ and we claim that  
    \begin{equation}\label{eq:upper_T}
    |T|+ |CC(G_{\Delta})| -|CC(G_{\Delta,T})|- \mathrm{Cyc}(T)\geq 0\enspace . 
    \end{equation}
    We first conclude the proof and then we show this claim. We deduce from~\eqref{eq:upper:corrective} and~\eqref{eq:upper_T} that 
    \[
    \left|\mathbb{E}\left[\underline{\omega}_{\mathbf{P}}\right]\right|\leq \frac{d^{|\mathrm{Cyc}(T)|} \Delta^{2|E_{\Delta}|}}{K^{|V_{\Delta}|-|CC(G_{\Delta})|}} 2^{|V_{2,np}|}\ , 
    \]
    which is the desired bound. 
    
    Let us now show~\eqref{eq:upper_T}  by induction on $T$. This is obviously true for $T=\emptyset$. Suppose this is true for $T\setminus\{v\}$. When we add $v$, we prune two more half-edges to go from $G_{\Delta,T\setminus \{v\}}$ to $G_{\Delta,T}$. It these two half-edges corresponded to a self-edge of $v$, then the number of cycles increases by one, we number of connected components is the same:  
    \[
    |T|- |CC(G_{\Delta})| -|CC(G_{\Delta,T})|- |\mathrm{Cyc}(T)| = |T\setminus\{v\}|- |CC(G_{\Delta})| -|CC(G_{\Delta,T\setminus\{v\}})|- |\mathrm{Cyc}(T\setminus\{v\})| \geq 0\ . 
    \]
    It these two half-edges do not correspond to a self-edge, then $\mathrm{Cyc}(T\setminus\{v\})= \mathrm{Cyc}(T)$. Besides, by going from $G_{\Delta,T\setminus\{v\}}$ to $G_{\Delta,T}$, we have replaced two edges incident to $v$ by an edge between the two corresponding neighbors of $v$. By doing this, we have increased the number of connected components by at most one. The result follows. 
    
\end{proof}

    \subsection{Proof of Proposition~\ref{prop:covariance_generaltilde}}

        The first inequality is obtained by Lemma~\ref{lem:covariance_degree_odd} and by summing the terms in Proposition~\ref{prop:covariance_general} over all possible matchings $\mathbf M$ in $\mathcal{M}$, and over all injections $\pi^{(1)}$, $\pi^{(2)}$ that are compatible with $\mathbf M$. The number of these injections  is $\frac{(n-2)!}{(n-|V_{\Delta}|)!}$.  For a matching ${\mathbf M}$ such that all degree-2 nodes are matched, the number $|V_{2,np}(\mathbf{P})|$ of degree $2$ nodes whose half-edges are not paired is smaller than $2[|\mathbf{M}| - |\mathbf{M}_{\mathrm{full}}|]$, where we recall $\mathbf{M}_{\mathrm{full}}$ the matching of nodes such that all incident edges are paired. We have proved the first identity. 

        The second equality is obtained analogously by applying Proposition~\ref{prop:covariance_without_degree_2}.

    \medskip 
    We now turn to the polynomials $\widetilde{\Psi}_G$. If $G^{(1)}$ is the graph with two isolated nodes, then $\widetilde{\Psi}_{G^{(1)}}=1$ and $\mathbf{E}[\widetilde{\Psi}_{G^{(1)}}\widetilde{\Psi}_{G^{(2)}}]=\mathbf{E}[\widetilde{\Psi}_{G^{(2)}}]$ and the result follows from Lemma~\ref{lem:first_moment_tilde_psi}. 
    The core of the proof is to establish~\eqref{eq:covariance_tilde}. 
    We write $G^{(1)}_1,\ldots, G^{(1)}_{r_1}$ and $G^{(2)}_1,\ldots, G^{(2)}_{r_2}$ for the respective connected components of $G^{(1)}$ and $G^{(2)}$. Consider any $\pi^{(1)}$ and $\pi^{(2)}$ such that the corresponding matching $\mathbf{M}$ belongs to $\mathcal{M}\setminus \mathcal{M}^{\star}$. Hence, there exists a connected components, say $G^{(1)}_1$ such that no node of $G^{(1)}_1$ is matched to any node of $G^{(2)}$. As a consequence, conditionally to $\mu$, $\overline{\Psi}_{G^{(1)}_1, \pi^{(1)}}$ is independent of $(\overline{\Psi}_{G^{(1)}_i, \pi^{(1)}})_{2 \leq i \leq r_1}, (\overline{\Psi}_{G^{(2)}_j, \pi^{(2)}})_{j \leq r_2}$. Hence, 
        \begin{align*}
        &\mathbb{E}\left[\prod_{1\leq i\leq r_i}\left[\overline{\Psi}_{G^{(1)}_i, \pi^{(1)}} - \mathbb E \overline{\Psi}_{G^{(1)}_i, \pi^{(1)}}\right]\prod_{1\leq j\leq r_2} \left[\overline{\Psi}_{G^{(2)}_j, \pi^{(2)}} - \mathbb E\overline{\Psi}_{G^{(2)}_j, \pi^{(2)}}\right]\Big|\mu\right]\\
        &= \mathbb{E}\left[\prod_{2\leq i\leq r_i}\left[\overline{\Psi}_{G^{(1)}_i, \pi^{(1)}} - \mathbb E \overline{\Psi}_{G^{(1)}_i, \pi^{(1)}}\right]\prod_{1\leq j\leq r_2} \left[\overline{\Psi}_{G^{(2)}_j, \pi^{(2)}} - \mathbb E\overline{\Psi}_{G^{(2)}_j, \pi^{(2)}}\right]\Big|\mu\right] \mathbb{E}\left[\overline{\Psi}_{G^{(1)}_1, \pi^{(1)}} - \mathbb E \overline{\Psi}_{G^{(1)}_1, \pi^{(1)}}\Big|\mu\right]
    \end{align*}
    which is $0$ since $\mathbb{E}[\overline{\Psi}_{G^{(1)}_1, \pi^{(1)}} - \mathbb E \overline{\Psi}_{G^{(1)}_1, \pi^{(1)}}|\mu] = 0$ as $\mathbb E[ \overline{\Psi}_{G^{(1)}_1, \pi^{(1)}}\mu|]$ does not depend on $\mu$ -- see the proof of Lemma~\ref{lem:first_moment_tilde_psi}. 
    
    As a consequence, only node matchings $\mathbf{M}$ in $\mathcal{M}^{\star}$ play a role in $\mathbb{E}[\widetilde{\Psi}_{G^{(1)}} \widetilde{\Psi}_{G^{(2)}}]$. 
    Now take a matching $\mathbf M \in \mathcal{M}^{\star}$ and consider any two injections $\pi^{(1)}$, $ \pi^{(2)}$ with node matching $\mathbf M$. We first develop the product
    \begin{align*}
    &\mathbb{E}\left[\prod_{1\leq i\leq r_1}\left[\overline{\Psi}_{G^{(1)}_i, \pi^{(1)}} - \mathbb E \overline{\Psi}_{G^{(1)}_i, \pi^{(1)}}\right]\prod_{1\leq j\leq r_2} \left[\overline{\Psi}_{G^{(2)}_j, \pi^{(2)}} - \mathbb E\overline{\Psi}_{G^{(2)}_j, \pi^{(2)}}\right]\Big|\mu\right]\\
    &= \sum_{S_1 \subset [r_1], S_2\subset [r_2]} (-1)^{|S_1|+|S_2|} \mathbb{E}\left[\prod_{i\not\in S_1}  \overline{\Psi}_{G^{(1)}_i, \pi^{(1)}} \prod_{j\not\in S_2} \overline{\Psi}_{G^{(2)}_j, \pi^{(2)}} \Big|\mu\right] \prod_{i \in S_1} \mathbb E \overline{\Psi}_{G^{(1)}_i, \pi^{(1)}} \prod_{j \in S_2} \mathbb E \overline{\Psi}_{G^{(2)}_j, \pi^{(2)}}\enspace . 
    \end{align*}

    \begin{lemma}\label{lem:cores}
      
      Fix any $S_1\subset [r_1]$, any $S_2\subset [r_2]$ and any $\mathbf{M}\in \mathcal{M}^{\star}$ and any $\pi^{(1)}$ and $\pi^{(2)}$ compatible with $\mathbf{M}$.  Define $\mathbf{M}_{S_1,S_2}$ the subset of $\mathbf{M}$ where we remove any pair of nodes if it involves at least a node from the connected components $G^{(1)}_i$ with $i\in S_1$ or from the connected components $G^{(1)}_2$ with $i\in S_2$.  Let $\overline{\pi}^{(1)}$ and $\overline{\pi}^{(2)}$ be any two injections that are compatible with $\mathbf{M}_{S_1,S_2}$. Then, we have 
        \begin{align*}
        &\mathbb{E}\left[\prod_{i\not\in S_1}  \overline{\Psi}_{G^{(1)}_i, \pi^{(1)}} \prod_{j\not\in S_2} \overline{\Psi}_{G^{(2)}_j, \pi^{(2)}} \right] \prod_{i \in S_1} \mathbb E \overline{\Psi}_{G^{(1)}_i, \pi^{(1)}} \prod_{j \in S_2} \mathbb E \overline{\Psi}_{G^{(2)}_j, \pi^{(2)}} =  \mathbb{E}\left[ \overline{\Psi}_{G^{(1)}, \overline{\pi}^{(1)}}  \overline{\Psi}_{G^{(2)}, \overline{\pi}^{(2)}} \right]\enspace . 
        \end{align*}
    \end{lemma}
    
    Hence, we deduce from Lemma~\ref{lem:cores} that 
    \begin{align}\nonumber
        &\left|\mathbb{E}\left[\prod_{1\leq i\leq r_i}\left[\overline{\Psi}_{G^{(1)}_i, \pi^{(1)}} - \mathbb E \overline{\Psi}_{G^{(1)}_i, \pi^{(1)}}\right]\prod_{1\leq j\leq r_2} \left[\overline{\Psi}_{G^{(2)}_j, \pi^{(2)}} - \mathbb E\overline{\Psi}_{G^{(2)}_j, \pi^{(2)}}\right]\right] - \mathbb E[\overline{\Psi}_{G^{(1)}, \pi^{(1)}}\overline{\Psi}_{G^{(2)}, \pi^{(2)}}]\right|\\ \label{eq:upper_mixed_moment}
        &\leq \sum_{S_1 \subset [r_1], S_2\subset [r_2]: S_1\cup S_2 \neq \emptyset}  \left|\mathbb{E}\left[ \overline{\Psi}_{G^{(1)}, \overline{\pi}^{(2)}}  \overline{\Psi}_{G^{(2)}, \overline{\pi}^{(2)}} \right]\right|\ .
        \end{align}
        Note that for $S_1,S_2$ such that $S_1\cup S_2 \neq \emptyset$ we have, by Proposition~\ref{prop:covariance_general} that
        \begin{eqnarray}\label{eq:upper_bound_1_pairing}
            \left|\mathbb{E}\left[ \overline{\Psi}_{G^{(1)}, \overline{\pi}^{(1) }}  \overline{\Psi}_{G^{(2)}, \overline{\pi}^{(2)}}\right]  \right|&\leq   \sum_{\mathbf{P}\in \mathcal{P}[\mathbf{M}_{S_1,S_2}]}2^{2[|\mathbf{M}_{S_1,S_2}| - |(\mathbf{M}_{S_1,S_2})_{\mathrm{full}}|] }\Delta^{2|E'_{\Delta}|}d^{|\mathrm{Cyc}'|}\frac{1}{K^{|V_{\Delta}|-|\mathrm{CC}(G'_{\Delta})|}}\enspace , 
        \end{eqnarray}
        where, in a the above bound,  $G'_{\Delta}=(V_\Delta,E'_{\Delta})$ is the graph  associated to $\mathbf{M}_{S_1,S_2}$ and to $\mathbf{P}$, whereas $|\mathrm{Cyc}'|$ is the number of cycles pruned in the construction of $G'_{\Delta}$. In fact, the previous quantity in~\eqref{eq:upper_bound_1_pairing} is no higher than if we replaced $G'_{\Delta}$ by $G_{\Delta}$ as described in the next proposition.
    
    \begin{lemma}\label{lem:core2}
    Define $\mathcal{P}[\mathbf{M}; S_1;S_2]$ the collection of pairings such that no half-edges incident to a node in the connected components $G^{(1)}_i$ with $i\in S_1$ or $G^{(2)}_i$ with $i\in S_2$ are paired. Then, we have  that 
       \begin{eqnarray}\label{eq:upper_bound_2_pairing}
            \left|\mathbb{E}\left[ \overline{\Psi}_{G^{(1)}, \overline{\pi}^{(1)}}  \overline{\Psi}_{G^{(2)}, \overline{\pi}^{(2)}}\right]  \right|&\leq   \sum_{\mathbf{P}\in \mathcal{P}[\mathbf{M};S_1;S_2]}2^{2[|\mathbf{M}| - |\mathbf{M}_{full}|]} \Delta^{2|E_{\Delta}|}d^{|\mathrm{Cyc}|}\frac{1}{K^{|V_{\Delta}|-|\mathrm{CC}(G_{\Delta})|}}\enspace .
        \end{eqnarray}
    \end{lemma}

    Coming back to~\eqref{eq:upper_mixed_moment}, summing over all $\mathbf{M}\in \mathcal{M}^{\star}$ and over all $\pi^{(1)}$ and $\pi^{(2)}$ that are compatible with $\mathbf{M}$, and applying Lemma~\ref{lem:core2}, we arrive at 
        \begin{align*}
     \lefteqn{       \left|\mathbf E[\widetilde{\Psi}_{G^{(1)}}\widetilde{\Psi}_{G^{(2)}}]- \mathbf E[\overline{\Psi}_{G^{(1)}}\overline{\Psi}_{G^{(2)}}]\right|} &\\ 
            &\leq \sum_{\mathbf M\in \mathcal{M}^{\star}} \frac{(n-2)!}{(n-|V_{\Delta}|)!}  \sum_{S_1 \subset [r_1], S_2\subset [r_2]: S_1\cup S_2 \neq \emptyset} \sum_{\mathbf{P}\in \mathcal{P}[\mathbf{M};S_1;S_2]}2^{2[|\mathbf{M}| - |\mathbf{M}_{full}|]} \Delta^{2|E_{\Delta}|}d^{|\mathrm{Cyc}|}\frac{1}{K^{|V_{\Delta}|-|\mathrm{CC}(G_{\Delta})|}} \enspace . 
        \end{align*}
    To conclude, it remains to reorganize the sum over pairing $\mathbf{P}\in \mathcal{P}[\mathbf{M}]$. Note that $\mathbf{P}$ belongs to $\mathcal{P}[\mathbf{M};S_1;S_2]$ only if all the nodes in the corresponding connected components are either not matched or are matched but none of their half-edges are paired. Given $\mathbf{M}\in \mathcal{M}^{\star}$, denote $V_0(\mathbf{P})$ the collection nodes  that are matched but are not incident to  any paired half-edges.
    Since, by definition of $\mathbf{M}\in \mathcal{M}^{\star}$, in any connected component, there is at least one matched node, this implies that, for a given $\mathbf{P}\in \mathcal{P}[\mathbf{M}]$, there are at most $2^{|V_0(\mathbf{P})|}-1$ sets  $(S_1,S_2)\neq \emptyset$ such that $\mathbf{P}\in \mathcal{P}[\mathbf{M},S_1,S_2]$. Since $|V_0(\mathbf{P})|\leq 2[|\mathbf{M}|- |\mathbf{M}_{\mathrm{full}}|]$, it follows that 
    \begin{align*}
         \lefteqn{       \left|\mathbf E[\widetilde{\Psi}_{G^{(1)}}\widetilde{\Psi}_{G^{(2)}}]- \mathbf E[\overline{\Psi}_{G^{(1)}}\overline{\Psi}_{G^{(2)}}]\right|} &\\ 
            &\leq \sum_{\mathbf M\in \mathcal{M}^{\star}} \frac{(n-2)!}{(n-|V_{\Delta}|)!}\sum_{\mathbf{P}\in \mathcal{P}[\mathbf{M}]\setminus \mathcal{P}_{\mathrm{full}}(\mathbf{M})}2^{2[|\mathbf{M}|- \mathbf{M}_{\mathrm{full}}|]}(2^{2[|\mathbf{M}|- \mathbf{M}_{\mathrm{full}}|]}-1)\Delta^{2|E_{\Delta}|}d^{|\mathrm{Cyc}|}\frac{1}{K^{|V_{\Delta}|-|\mathrm{CC}(G_{\Delta})|}}\enspace , 
             \end{align*}
    Finally, we apply~\eqref{eq:moment_order_2} to conclude the proof of~\eqref{eq:covariance_tilde}.

    \begin{proof}[Proof of Lemma~\ref{lem:cores}]
        Conditionally to $\mu$, the variables  $\overline{\Psi}_{G^{(1)}_i, \overline{\pi}^{(1)}}$ with $i\in S_1$ and the $\overline{\Psi}_{G^{(2)}_j, \overline{\pi}^{(2)}}$ with $j\in S_2$ are independent and they are independent of the other terms. Besides, the conditional expectation to $\mu$ of terms $\overline{\Psi}_{G^{(a)}_j, \underline{\pi}^{(a)}}$ does not depend on $\mu$.
        Hence, we have 
            \begin{align*}
                \mathbb{E}\left[ \overline{\Psi}_{G^{(1)}, \overline{\pi}^{(1)}}  \overline{\Psi}_{G^{(2)}, \overline{\pi}^{(2)}}|\mu \right]& =  \mathbb{E}\left[ \overline{\Psi}_{G^{(1)}, \overline{\pi}^{(1)}}  \overline{\Psi}_{G^{(2)}, \overline{\pi}^{(2)}}|\mu \right] \\
       &= \mathbb{E}\left[\prod_{i\not\in S_1}  \overline{\Psi}_{G^{(1)}_i, \overline{\pi}^{(1)}} \prod_{j\not\in S_2} \overline{\Psi}_{G^{(2)}_j, \overline{\pi}^{(2)}}|\mu \right]\prod_{i \in S_1} \mathbb E\left[\overline{\Psi}_{G^{(1)}_i, \overline{\pi}^{(1)}}|\mu\right] \prod_{j \in S_2} \mathbb E \left[\overline{\Psi}_{G^{(2)}_j, \overline{\pi}^{(2)}}|\mu\right]   \\
       &= \mathbb{E}\left[\prod_{i\not\in S_1}  \overline{\Psi}_{G^{(1)}_i, \overline{\pi}^{(1)}} \prod_{j\not\in S_2} \overline{\Psi}_{G^{(2)}_j, \overline{\pi}^{(2)}} \right] \prod_{i \in S_1} \mathbb E\left[\overline{\Psi}_{G^{(1)}_i, \overline{\pi}^{(1)}}\right] \prod_{j \in S_2} \mathbb E \left[\overline{\Psi}_{G^{(2)}_j, \overline{\pi}^{(2)}}\right]
       \enspace . 
            \end{align*}
        Then, we use that the first moments of the $\overline{\Psi}_{G^{(a)}_i, \overline{\pi}^{(a)}}$ for $a=1,2$ do not depend on the actual value of the injection $\overline{\pi}^{(a)}$. Similarly, the  $\mathbb{E}\left[\prod_{i\not\in S_1}  \overline{\Psi}_{G^{(1)}_i, \overline{\pi}^{(1)}} \prod_{j\not\in S_2} \overline{\Psi}_{G^{(2)}_j, \overline{\pi}^{(2)}}|\mu \right]$ only depends on $\overline{\pi}^{(1)}$ and $\overline{\pi}^{(2)}$ through their matching $\mathbf{M}_{S_1,S_2}$. The result follows. 
        \end{proof}

        \begin{proof}[Proof of Lemma~\ref{lem:core2}]
    Let us compare the graph $G'_{\Delta}$ arising from a pairing $(\mathbf{M}_{S_1,S_2},\mathbf{P})$ in~\eqref{eq:upper_bound_1_pairing} and the corresponding  graph  $G_{\Delta}$ arising from $(\mathbf{M},\mathbf{P})$. First, we have $|E_{\Delta}|=|E'_{\Delta}|$ and $|\mathrm{Cyc}|=|\mathrm{Cyc}'|$, as only the node matching is changed. 
     In $G'_{\Delta}$, the number $|V'_{\Delta}|$ of nodes is higher than $|V_{\Delta}|$ as less nodes are matched, but the number of connected components is higher. In fact, the new connected components in $G_{\Delta}'$ only arise because some nodes are not matched anymore. Removing a matching between two nodes creates at most one new connected component. Hence, we have  $|V'_{\Delta}|- |CC(G'_{\Delta})|\geq |V_{\Delta}|- |CC(G_{\Delta})|$ and the result follows. All in all, we have shown that 
     \[
     \Delta^{2|E'_{\Delta}|}d^{|\mathrm{Cyc}'|}\frac{1}{K^{|V'_{\Delta}|-|\mathrm{CC}(G'_{\Delta})|}}\leq \Delta^{2|E_{\Delta}|}d^{|\mathrm{Cyc}|}\frac{1}{K^{|V_{\Delta}|-|\mathrm{CC}(G_{\Delta})|}}\enspace ,
     \]
     which concludes the proof. 
        \end{proof}

\section{Technical proofs for the upper bound}\label{sec:proof:technical:UB}

\subsection{Proof of Lemma~\ref{lem:upper_boundexpectation}}
    Similarly to the proof of Lemma~\ref{lem:first_moment}, we have the following. For any labeling $\pi\in \Pi_V$ such that $\pi(v_1)=1$, $\pi(v_2)=2$, we have 
\[
        \mathbb{E}[\overline{\Psi}_{G,\pi}|\mu, k^*,b_1,b_2]= \Delta^{2|E|} \mathbf{1}\{k^*(\pi(v))\text{ is  contant for } v\in V  \}
\] 
It readily follows that  $\mathbb{E}[\overline{\Psi}_{G,\pi}|\mu,k^*(1),k^*(2),b_1,b_2, x=0]=0$. Besides, we have 
\[
        \mathbb{E}[\overline{\Psi}_{G,\pi}|\mu, k^*(1),k^*(2),b_1,b_2,  x=1]= \frac{\Delta^{2|E|}}{K^{2|V|-2}} \ , 
\]
and we conclude by summing over $\Pi_{V}$.

\subsection{Proof of Lemma~\ref{lem:UBva}}
We deduce from~\eqref{eq:upper_cross_variance:conditional} in the proof of Proposition~\ref{prop:covariance_without_degree_2} that, for any labeling $\pi^{(1)}$ and $\pi^{(2)}$ of $G^*$
and the corresponding matching $\mathbf{M}$ --as defined in the statement of Proposition~\ref{prop:covariance_without_degree_2}--, we have 
 \begin{align}\label{eq:upper_cross_variance:conditional2}
               \mathbb{E}\left[\overline{\Psi}_{G^{*},\pi^{(1)}} \overline{\Psi}_{G^{*},\pi^{(2)}}|\mu, b_1,b_2,k^*(1),k^*(2) \right]=  \sum_{\mathbf{P} \in \mathcal{P}[\mathbf{M}]} 
                d^{|\mathrm{Cyc}|} \mathbb{E}\left[\prod_{e\in E_{\Delta}} \langle \mu_{k^*(l(e))},  \mu_{k^*(r(e))}\rangle \big|\mu, x \right]\ .
\end{align}
Recall that $\mathbf{1}\{v_1 \sim_{G_\Delta} v_2\}$ (resp. $\mathbf{1}\{v_1 \nsim_{G_\Delta} v_2\}$) means that the node $1$ and $2$ belong (resp. do not belong) to the same connected component in $G_{\Delta}$. 
Since the $\mu_k$'s are orthogonal almost surely, we deduce from~\eqref{eq:upper_cross_variance:conditional} that 
\[
\mathbb{E}\left[\overline{\Psi}_{G^{*},\pi^{(1)}} \overline{\Psi}_{G^{*},\pi^{(2)}}|\mu, b_1,b_2,k^*(1),k^*(2) \right]=  \sum_{\mathbf{P} \in \mathcal{P}[\mathbf{M}]} 
 d^{|\mathrm{Cyc}|}  \frac{\Delta^{2|E_{\Delta}|} }{K^{|V_{\Delta}|- |\mathrm{CC}(G_{\Delta})|- \mathbf{1}\{v_1 \sim_{G_\Delta} v_2\}}} 
 \left[\mathbf{1}\{x=1\}+ \mathbf{1}\{x=0\}\mathbf{1}\{v_1 \nsim_{G_\Delta} v_2\}\right]
\]
Then, summing over all $\pi_1$ and $\pi_2$, we arrive at 
\begin{align*}
\mathbb{E}\left[\overline{\Psi}_{G^*}^2  |\mu, b_1,b_2,k^*(1),k^*(2)\right]&= \sum_{\mathbf{M},\mathbf{P}\in \mathcal{MP}} \frac{(n-2)!}{(n-|V_{\Delta}|)!} \frac{d^{|\mathrm{Cyc}|} \Delta^{2|E_{\Delta}|} }{K^{|V_{\Delta}|- |\mathrm{CC}(G_{\Delta})|- \mathbf{1}\{v_1 \sim_{G_\Delta} v_2\}}} 
 \left[\mathbf{1}\{x=1\}+ \mathbf{1}\{x=0\}\mathbf{1}\{v_1 \nsim_{G_\Delta} v_2\}\right]\\ 
 &= \sum_{\mathbf{M},\mathbf{P}\in \mathcal{MP}, \mathbf{P}\neq \emptyset } \frac{(n-2)!}{(n-|V_{\Delta}|)!} \frac{d^{|\mathrm{Cyc}|} \Delta^{2|E_{\Delta}|} }{K^{|V_{\Delta}|- |\mathrm{CC}(G_{\Delta})|- \mathbf{1}\{v_1 \sim_{G_\Delta} v_2\}}} 
 \left[\mathbf{1}\{x=1\}+ \mathbf{1}\{x=0\}\mathbf{1}\{v_1 \nsim_{G_\Delta} v_2\}\right] \\
  & \quad \quad \quad + \mathbf{1}\{x=1\} \sum_{\mathbf{M}\in \mathcal{M}}  \frac{(n-2)!}{(n-|V_{\Delta}|)!} \frac{\Delta^{4|E|} }{K^{2|V|-|\mathbf{M}|- 2}} \ .
\end{align*}
where we used in the last line, that for $\mathbf{P}=\emptyset$, we have $|\mathrm{Cyc}|=0$, $|E_{\Delta}|=2|E|$, $G_{\Delta}$ is connected, and $|V_{\Delta}|=2|V|-|\mathbf{M}|$. 
With Lemma~\ref{lem:upper_boundexpectation}, we conclude that 
\begin{align*}
 \mathbb{E}\left[\left[\overline{\Psi}_{G^*} - \mathbb E[\overline{\Psi}_{G^*}|\mu,x=0]\right]^2 |\mu, b_1,b_2,k^*(1),k^*(2), x=0\right]\hspace{4cm}\\
 =  \sum_{\mathbf{M},\mathbf{P}\in \mathcal{MP}:~\mathbf{P} \neq \emptyset} \frac{(n-2)!}{(n-|V_{\Delta}|)!} \Delta^{2|E_{\Delta}|}d^{|\mathrm{cyc}|}\frac{1}{K^{|V_{\Delta}|-|\mathrm{CC}(G_{\Delta})|}}  \mathbf{1}\{v_1 \nsim_{G_\Delta} v_2\}\enspace , 
    \end{align*}
and that 
\begin{align*}
 \mathbb{E}\left[\left[\overline{\Psi}_{G^*} - \mathbb E[\overline{\Psi}_{G^*}|\mu,x=1]\right]^2 |\mu, b_1,b_2,k^*(1),k^*(2), x=1\right]\hspace{4cm}\\ =  \sum_{\mathbf{M},\mathbf{P}\in \mathcal{MP}:~\mathbf{P} \neq \emptyset} \frac{(n-2)!}{(n-|V_{\Delta}|)!} \Delta^{2|E_{\Delta}|}d^{|\mathrm{cyc}|}\frac{1}{K^{|V_{\Delta}|-|\mathrm{CC}(G_{\Delta})|-\mathbf{1}\{v_1 \sim_{G_\Delta} v_2\}}}  \\ + \mathbb E^2[\overline{\Psi}_{G^*}|\mu,x=1]   \sum_{\mathbf{M}\in  \mathcal{M} }   \frac{(n-|V|)!^2}{(n-2)!(n-2|V|+|\mathbf{M}|!) } \left(K^{|\mathbf{M}|- 2}-1\right)
 \enspace . 
\end{align*}

\subsection{Proof of Lemma~\ref{prop:UBvar1}}
    Since $\mathbf{M}=\mathbf{M}_0$, we have $|V_{\Delta}|= 2|V|-2$, $|\mathrm{Cyc}|=0$. The only half edges that can be paired are as follows: half-edges incident to $v_{1}^{(1)}$ --there are $M+1$ of them-- with half-edges incident to $v_1^{(2)}$ --there are $M+1$ of them--  and half-edges incident to $v_2^{(1)}$  with half-edges incident to $v_2^{(2)}$. A a consequence $|\mathbf{P}|$ satisfies $|\mathbf{P}|\leq 2(M+1)$. Each of these pairings creates an open path of length 1 in $G_{\cup}[\mathbf{M}_0,\mathbf{P}]$. If there are $m_1$ pairings incident to the node $v_1$ that identifies $v_1^{(1)}$ and $v_1^{(2)}$ in $G_{\cup}[\mathbf{M}_0,\mathbf{P}]$, it decreases the degree of $v_1$ by $2m_1$ in $G_{\Delta}[\mathbf{M}_0,\mathbf{P}]$. 
    The graph $G_{\Delta}$ remains connected as long as $m_1< M+1$ and $m_2< M+1$. If  $m_1=M+1$ (resp. $m_2=M+1$), then the node $v_1$ that identifies $v_1^{(1)}$ and $v_1^{(2)}$ (resp. $v_2$ that identifies $v_2^{(1)}$ and $v_2^{(2)}$) becomes isolated in $G_{\Delta}$. Then, summing over the size of $\mathbf{P}$ and recalling the value of $\mathbb E[\overline{\Psi}_{G^*}|\mu,x=1]$ from Lemma~\ref{lem:upper_boundexpectation}, we arrive at 
\begin{align*}
B&= \sum_{\substack{0\leq m_1\leq M+1,\\ 0\leq m_2\leq M+1 , \\ m_1+m_2>0 }} \binom{M+1}{m_1}^2 m_1! \binom{M+1}{m_2}^2 m_2! \frac{(n-2)!}{(n-2|V|+2)!}\frac{\Delta^{4|E|-2m_1-2m_2}}{K^{2(|V|-2)-\mathbf{1}\{m_1+m_2=2M+2\}}}\\ 
&\leq  \mathbb E^2[\overline{\Psi}_{G^*}|\mu,x=1]\Big[\frac{K(M+1)^{2(M+1)}}{\Delta^{4(M+1)}} + \sum_{\substack{0\leq m_1\leq M+1,\\ 0\leq m_2\leq M+1 , \\ 0 <m_1+m_2<2(M+1)}} \frac{(M+1)^{3(m_1+m_2)} }{\Delta^{2(m_1+m_2)}}\Big]\\ 
&\leq \mathbb E^2[\overline{\Psi}_{G^*}|\mu,x=1]\Big[\frac{K(M+1)^{2(M+1)}}{\Delta^{4(M+1)}} + \sum_{m=1}^{2M+1} (m+1)  \left(\frac{(M+1)^{3}}{\Delta^{2}}\right)^{m  }\Big]\\
&\leq \mathbb E^2[\overline{\Psi}_{G^*}|\mu,x=1]\Big[\frac{K(M+1)^{2(M+1)}}{\Delta^{4(M+1)}} + 2  \frac{(M+1)^{5}}{\Delta^{2}}\Big]\ , 
\end{align*}
provided that $\Delta^2 \geq 2(M+1)^{4}$.

\subsection{Proof of Proposition~\ref{prop:template:chaine_rappeur}}
We consider  two replicates $G^{(1)}$ and $G^{(2)}$ of $G^*$, a node matching $\mathbf{M}$ and a paring $\mathbf{P}$.  For the purpose of this proof, we shall argue on the multigraphs $G_{\cup}[\mathbf{M},\mathbf{P}]$, $G^{(1)}[\mathbf{M},\mathbf{P}]$, and $G^{(2)}[\mathbf{M},\mathbf{P}]$. Recall that $G_{\cup}[\mathbf{M},\mathbf{P}]$ is built by  merging two replicates $G^{(1)}$ and $G^{(2)}$ of $G^*$ where we identify nodes in $\mathbf{M}$. Here, $G^{(1)}[\mathbf{M},\mathbf{P}]$ is isomorphic to $G$ and is interpreted as a sub-multigraph of $G_{\cup}[\mathbf{M},\mathbf{P}]$.

 Let us denote $\mathbf{M}^{(1)}$ the subset of nodes of $G^{(1)}[\mathbf{M},\mathbf{P}]$ that correspond to matched nodes. By definition, we have $|\mathbf{M}^{(1)}|=|\mathbf{M}|$. 
 An edge of $G^{(1)}[\mathbf{M},\mathbf{P}]$ is said to be \emph{fully paired} is both corresponding half-edges are paired, i.e. both half edges arise in the pairing $\mathbf{P}$. 
 Let us denote $FE^{(1)}[\mathbf{M},\mathbf{P}]$ the collection of fully paired edges in $G^{(1)}[\mathbf{M},\mathbf{P}]$. By definition, any fully-paired edge should be incident to nodes that are matched. For this reason, we introduce $E(\mathbf{M}^{(1)})$ as the multi-set of edges such that their two extremities belong to $\mathbf{M}^{(1)}$. So that we have $FE^{(1)}[\mathbf{M},\mathbf{P}]\subset E(\mathbf{M}^{(1)})$. We further partition the collection of fully-paired edges into $FE_c^{(1)}[\mathbf{M},\mathbf{P}]$ and $FE_o^{(1)}[\mathbf{M},\mathbf{P}]$, depending on whether they arise in cycles or open paths in the construction of $G_{\Delta}$.

We shall argue using the structure of $\mathbf{M}^{(1)}$. In particular, we decompose  $\mathbf{M}^{(1)}\setminus\{v^{(1)}_1,v^{(1)}_2\}$ into maximum  consecutive sequences of the form $\{v^{(1)}_l, v^{(1)}_{l+1},\ldots, v^{(1)}_{l+b}\}$. Henceforth, we write $c$ for the number of such consecutive sequences $\mathbf{M}^{(1)}\setminus\{v^{(1)}_1,v^{(1)}_2\}$. 

\begin{lemma}\label{lem:technique:0}
We have the following decomposition
\begin{align}\nonumber
    &|E(\mathbf{M}^{(1)})|\leq   \\ \label{eq:upper_bound:E}
& 2(|\mathbf{M}^{(1)}|-2) -2c + 2\mathbf{1}\{v^{(1)}_3\in \mathbf{M}^{(1)} \} + 2\mathbf{1}\{v^{(1)}_{LM+2} \in \mathbf{M}^{(1)}\}  + |\{m\in [M-1]: \{v^{(1)}_{mL+2},v^{(1)}_{mL+3}\}\cap \mathbf{M}^{(1)}\neq \emptyset\}|\ . 
\end{align}    
\end{lemma}

\begin{proof}[Proof of Lemma~\ref{lem:technique:0}]
Consider a maximum set $[v^{(1)}_l; v^{(1)}_{l+b}]$ that is included in $\mathbf{M}^{(1)}\setminus\{v^{(1)}_1,v^{(1)}_2\}$. If $l>2$ and $l+b\leq LM+2$, there are $2b- |\{m\in [M-1]: \{v^{(1)}_{mL+2},v^{(1)}_{mL+3}\}\subset[v^{(1)}_l; v^{(1)}_{l+b}] \}|$ edges in $G^{(1)}[\mathbf{M},\mathbf{P}]$ that connects the nodes in $[v^{(1)}_l; v^{(1)}_{l+b}]$. As an aside, there are also $|\{m\in [M-1]: v^{(1)}_{mL+2}\in [v^{(1)}_l; v^{(1)}_{l+b}] \}|+ |\{m\in [M-1]: v^{(1)}_{mL+3}\in [v^{(1)}_l; v^{(1)}_{l+b}] \}| $ edges that connect $[v^{(1)}_l; v^{(1)}_{l+b}]$ and $\{v^{(1)}_1,v^{(1)}_2\}$. The result follows by summing over all $c$ sequence and by considering properly the neighboring effects, that is the case where $v^{(1)}_3\in \mathbf{M}^{(1)}$ and $v^{(1)}_{LM+2}\in \mathbf{M}^{(1)}$
\end{proof}

Let us partition the connected components $\mathrm{CC}(G_{\Delta})$ of $G_{\Delta}$  into $\mathrm{CC}_{\mathrm{isol}}(G_{\Delta})$, $\mathrm{CC}_{\mathrm{matched}}(G_{\Delta})$, $\mathrm{CC}_{1}(G_{\Delta})$, $\mathrm{CC}_{2}(G_{\Delta})$, where 
$\mathrm{CC}_{\mathrm{isol}}(G_{\Delta})$ stands for the collection of isolated nodes without self-edges in $G_{\Delta}$, $\mathrm{CC}_{\mathrm{matched}}(G_{\Delta})$ is the collection of non-trivial connected components that intersect at least one matched node aside from $v_1$ and $v_2$.
$\mathrm{CC}_{1}(G_{\Delta})$ is the collection of non-trivial connected components that intersect $G^{(1)}[\mathbf{M},\mathbf{P}]$ and that do not intersect any matched node to the possible exception of $\{v_1,v_2\}$. Finally, $\mathrm{CC}_{2}(G_{\Delta})$ stands for the remaining connected components: those that neither intersect matched nodes nor intersect  $G^{(1)}[\mathbf{M},\mathbf{P}]$. 

All isolated nodes in $G_{\Delta}$ have to be matched and all of the incident half-edges have to be paired. Each component in $\mathrm{CC}_{\mathrm{matched}}(G_{\Delta})$ includes at least a matched node (aside from $v_1$ and $v_2$). Hence, we have   
\[
 |\mathrm{CC}_{\mathrm{isol}}(G_{\Delta})| - \mathbf{1}\{v_1\text{ isolated in }G_\Delta\} -  \mathbf{1}\{v_2\text{ isolated in }G_\Delta\} + |\mathrm{CC}_{\mathrm{matched}}(G_{\Delta})|\leq |\mathbf{M}| - 2\ , 
\]
which is equivalent to 
\begin{equation}\label{eq:upper_CC_0}
    |\mathrm{CC}_{\mathrm{isol}}(G_{\Delta})| + |\mathrm{CC}_{\mathrm{matched}}(G_{\Delta})|\leq |\mathbf{M}|-\mathbf{1}\{(v^{(1)}_1,v^{(2)}_1)\notin \mathbf{M}_{\mathrm{full}}\} -\mathbf{1}\{(v^{(1)}_2,v^{(2)}_2) \notin \mathbf{M}_{\mathrm{full}}\}   \enspace . 
\end{equation}

Besides, we claim that
\begin{equation}\label{eq:upper_CC_1}
    |\mathrm{CC}_{1}(G_{\Delta})| \leq c + 1 - \mathbf{1}\{v^{(1)}_3\in \mathbf{M}^{(1)} \} - \mathbf{1}\{v^{(1)}_{ML+2}\in \mathbf{M}^{(1)} \}  + \mathbf{1}\{\{v_1,v_2\}\in \mathrm{CC}(G_{\Delta})\}\enspace . 
\end{equation}
Let us prove this claim.
First consider the case where $\{v_1,v_2\}\notin CC(G_{\Delta})$. This enforces that any $C$ in $\mathrm{CC}_{1}(G_{\Delta})$ contains at least one node $v_{k}^{(1)}$ with $k\geq 3$. As  $v^{(1)}_k$ belongs to $C$, then  $[v^{(1)}_{l_1},v^{(1)}_{l_2}]\subset C$ where $l_1-1= \max\{l <k : v_l\in \mathbf{M}^{(1)}\}$ and $\max\{l >k : v_l\in \mathbf{M}^{(1)}\}$. Indeed, all the edges betwen nodes $[v^{(1)}_{l_1},v^{(1)}_{l_2}]$ cannot be pruned as the corresponding nodes are not matched. As a consequence, there are at most between $c-1$ and $c+1$ such connected components (depending on the boundary conditions). %It is possible that $v_1$ or $v_2$ belong to $C$, but as $C$ contains at least one set of the form $[v_{l_1},v_{l_2}]$ (otherwise $v_1$ or $v_2$ would be isolated and therefore not belong to $\mathrm{CC}_{1}(G_{\Delta})$ or would of the form $\{v_1,v_2\}$ which is impossible) this does not change the bound. 
If $\{v_1,v_2\}\in CC(G_{\Delta})$, then this connected components belongs to  $\mathrm{CC}_{1}(G_{\Delta})$ and the bound~\eqref{eq:upper_CC_1} accounts for it.

Finally, we focus on $\mathrm{CC}_{2}(G_{\Delta})$. Let us consider such a connected component $C\in \mathrm{CC}_{2}(G_{\Delta})$. Arguing as previously, we deduce that  $C$ is an union of nodes of the form $[v^{(2)}_{l_1},v^{(2)}_{l_2}]$ where $[v^{(2)}_{l_1},v^{(2)}_{l_2}]\cap \mathbf{M}^{(2)}=\emptyset$, $v^{(1)}_{l_1-1}\in \mathbf{M}^{(2)}$, and\footnote{In fact, if  $l_2=LM+2$, the condition is $v^{(2)}_2\in \mathbf{M}^{(2)}$} $v^{(2)}_{l_2+1}\in \mathbf{M}^{(2)}$. Since neither $v^{(2)}_{l_1-1}$, nor $v^{(2)}_{l_2+1}$ belong to $C$, this implies that the edges in $E^{(2)}$ of the form $(v^{(2)}_{l_1-1},v^{(2)}_{l_1})$ and $(v^{(2)}_{l_2},v^{(2)}_{l_2+1})$ have been pruned in $G_{\Delta}$ when removing the open paths of paired half-edges. Since the connected component $C$ does not intersect $G^{(1)}[\mathbf{M},\mathbf{P}]$, one sees that the length of these open paths is a least two --indeed, pruning an open path of length $1$ will connect a node of $G^{(1)}[\mathbf{M},\mathbf{P}]$ to a node of $G^{(2)}[\mathbf{M},\mathbf{P}]$. 
As a consequence, we can count two fully paired edges $G^{(1)}[\mathbf{M},\mathbf{P}]$ for $C$. Besides, let us consider the case where $\{v_1,v_2\}$ is a connected component in $G_{\Delta}$. This implies that there has been a pruned open path that allowed to connect $v_1$ to $v_2$ in $G_{\Delta}$. Since the graph distance from $v_1$ to $v_2$ in $G_{\cup}$ is $3$, this implies that at least one fully paired edge in $G^{(1)}[\mathbf{M},\mathbf{P}]$ has been used for this open path. In summary we have proved that
\begin{equation}\label{eq:upper_CC_2}
  2|\mathrm{CC}_{2}(G_{\Delta})|+ \mathbf{1}\{\{v_1,v_2\}\in CC(G_{\Delta})\}\leq   |FE_o^{(1)}[\mathbf{M},\mathbf{P}]|\enspace .
\end{equation}

Since each cycle involves at least one fully paired edge, we $|\mathrm{Cyc}|\leq |FE_c^{(1)}[\mathbf{M},\mathbf{P}]|$. Gathering~\eqref{eq:upper_CC_0},~\eqref{eq:upper_CC_1}, and~\eqref{eq:upper_CC_2}, we arrive at 
\begin{eqnarray*}
    |\mathrm{Cyc}| + |\mathrm{CC}(G_{\Delta})| -1 &\leq&  |\mathbf{M}|-\mathbf{1}\{(v^{(1)}_1,v^{(2)}_1)\notin \mathbf{M}_{\mathrm{full}}\} -\mathbf{1}\{(v^{(1)}_2,v^{(2)}_2)\notin \mathbf{M}_{\mathrm{full}}\}  + |FE_c^{(1)}[\mathbf{M},\mathbf{P}]| + |FE_o^{(1)}[\mathbf{M},\mathbf{P}]|\\ && + c   - \mathbf{1}\{v^{(1)}_3\in \mathbf{M}^{(1)} \} - \mathbf{1}\{v^{(1)}_{ML+2}\in \mathbf{M}^{(1)}\} 
\end{eqnarray*}
Then, since $|E(\mathbf{M}^{(1)})|\geq |FE_c^{(1)}[\mathbf{M},\mathbf{P}]| + |FE_o^{(1)}[\mathbf{M},\mathbf{P}]|$ we combine this with \eqref{eq:upper_bound:E} and we obtain
\begin{eqnarray*}
    |\mathrm{Cyc}| + |\mathrm{CC}(G_{\Delta})| -1 &\leq&  3|\mathbf{M}|-4-\mathbf{1}\{(v^{(1)}_1,v^{(2)}_1)\notin \mathbf{M}_{\mathrm{full}}\} -\mathbf{1}\{(v^{(1)}_2,v^{(2)}_2)\notin \mathbf{M}_{\mathrm{full}}\} \\ && - c +\mathbf{1}\{v_3^{(1)}\in \mathbf{M}^{(1)} \} + \mathbf{1}\{v^{(1)}_{ML+2}\in \mathbf{M}^{(1)}\}+  |\{m\in [M-1]: \{v^{(1)}_{mL+2},v^{(1)}_{mL+3}\}\cap \mathbf{M}^{(1)}\}| \ . 
\end{eqnarray*}
To conclude, it suffices to prove that 
\[
    -c   + \mathbf{1}\{v_3\in \mathbf{M}^{(1)} \} + \mathbf{1}\{v^{(1)}_{ML+2}\in \mathbf{M}^{(1)}\}+ |\{m\in [M-1]: \{v^{(1)}_{mL+2},v^{(1)}_{mL+3}\}\cap \mathbf{M}^{(1)}\neq \emptyset\}|\leq \lfloor (4(|\mathbf{M}|-2))/L\rfloor\ . 
\]
First, we consider the specific case where $|\mathbf{M}|=ML+2$ so that the left-hand side equals $M$, whereas the right-hand side is equal to $4M$ and the inequality therefore holds. Next, we focus on the case where  $|\mathbf{M}|< ML+2$. For any set  $[v^{(1)}_{l_1}; v^{(1)}_{l_1+x}]$ with $l_1>3$ and $l_1+x<  ML+2$, we have $|\{m\in [M-1]: \{v^{(1)}_{mL+2},v^{(1)}_{mL+3}\}\cap [v^{(1)}_{l_1};v^{(1)}_{l_1+x}]\neq \emptyset\}|\leq 1 + \lfloor (x+3)/L\rfloor \leq 1+ \lfloor 4(x+1)/L\rfloor$. For any interval of the form $[v^{(1)}_{l_1},v^{(1)}_{l_1+x}]$ with either $l_1=3$ or $l_1+x=ML+2$, we have 
$|\{m\in [M-1]: \{v^{(1)}_{mL+2},v^{(1)}_{mL+3}\}\cap [v^{(1)}_{l_1};v^{(1)}_{l_1+x}]\neq \emptyset\}|\leq  \lfloor 4(x+1)/L\rfloor$. Note that we cannot have $l_1=3$ and $l_1+x=LM+2$ as $|\mathbf{M}|\leq LM+2$. Summing over all these $c$ intervals in the decomposition  $\mathbf{M}^{(1)}\setminus\{v^{(1)}_1,v^{(1)}_2\}$, we arrive at the desired conclusion.

\subsection{Proof of Proposition~\ref{prop2:template:chaine_rappeur}}
We use the same notation as in the previous proof.         
We have observed in the previous proof that 
$|\mathrm{Cyc}|\leq |FE_c^{(1)}[\mathbf{M},\mathbf{P}]|$. 
Combining the partition of $\mathrm{CC}(G_{\Delta})$ into $\mathrm{CC}_{\mathrm{isol}}(G_{\Delta})$, $\mathrm{CC}_{\mathrm{matched}}(G_{\Delta})$, $\mathrm{CC}_{1}(G_{\Delta})$, and $\mathrm{CC}_{2}(G_{\Delta})$ with~\eqref{eq:upper_CC_0} and~\eqref{eq:upper_CC_2}, we get
\begin{eqnarray} \nonumber
|\mathrm{Cyc}| + |\mathrm{CC}(G_{\Delta})|-1 & \leq& |FE_c^{(1)}[\mathbf{M},\mathbf{P}]|+ |FE_o^{(1)}[\mathbf{M},\mathbf{P}]|+ |\mathrm{CC}_{1}(G_{\Delta})| - \mathbf{1}\{\{v_1,v_2\}\in \mathrm{CC}(G_{\Delta})\}\\ & &+ |\mathbf{M}|-\mathbf{1}\{(v^{(1)}_1,v^{(1)}_1)\notin \mathbf{M}_{\mathrm{full}}\} -\mathbf{1}\{(v^{(2)}_2,v^{(2)}_2) \notin \mathbf{M}_{\mathrm{full}}\} - 1 \enspace . \label{eq:prop2:chain}
\end{eqnarray}

Consider any connected component $C$ in $\mathrm{CC}_1(G_{\Delta})$ that is distinct from $\{v_1,v_2\}$. As argued in the previous proof, the restriction to $C$ of the nodes in $G^{(1)}[\mathbf{M},\mathbf{P}]\setminus\{v_1^{(1)},v_2^{(1)}\}$ is an union of intervals of the form $[v^{(1)}_l,v^{(1)}_{l+b}]$ where $v^{(1)}_{l-1}\in \mathbf{M}^{(1)}$ and $v^{(1)}_{l+b+1}\in  \mathbf{M}^{(1)}$ (except if $l+b=ML+2$). Suppose that $l>3$ and $l+b< LM+2$. Since neither $v^{(1)}_{l-1}$, nor $v^{(1)}_{l+b+1}$ belong to $C$ (by definition of $\mathrm{CC}_1(G_{\Delta})$), this means that half-edges $G^{(1)}[\mathbf{M},\mathbf{P}]$ that are incident to $v^{(1)}_{l-1}$ (resp. $v_{l+b+1}^{(1)}$) and correspond to edges incident to $v^{(1)}_{l}$ (resp. $v_{l+b}^{(1)}$) belong to the collection of paired half-edges. Note also that these half-edges, while paired,  do not belong to a fully paired edge as neither $v_l^{(1)}$ nor $v_{l+b}^{(1)}$ are matched. As a consequence, to each such component $C$ that does not intersect $v^{(1)}_{3}$ or $v^{(1)}_{LM+2}$, we can associate at least two paired half-edges that do not belong to a fully paired edge. If $v^{(1)}_3\in C$, $v^{(1)}_{LM+2}\in C$, but $v_1$ does not belong to $C$, we obtain by arguing similarly, that we can associate at least two paired half-edges that do not belong to a fully paired edge. In contrast, for $v^{(1)}_3\in C$, $v^{(1)}_{LM+2}\in C$, but $v_1\in C$, we can associate only one-such half-edges. Arguing similarly with the remaining cases, we arrive at 
%, we can associate two such paired half-edges. If $l=3$ and $l+b<LM+2$ or  If $l>3$ and $l+b=LM+2$, we can only associate one such paired half-edge. If $l=3$ and $l+b=LM+2$, we cannot associate any such paired half-edge. 
\begin{equation}\label{eq:upper:CC1:bis}
    2|\mathrm{CC}_{1}(G_{\Delta})|  \leq q + |\{C\in \mathrm{CC}_{1}(G_{\Delta}): C\cap \{v_1\}\neq \emptyset\}|+ |\{C\in \mathrm{CC}_{1}(G_{\Delta}): C\cap \{v_2\}\neq \emptyset\}|\leq q +2 \ , 
\end{equation}
where $q$ is the number of paired half-edges in $G^{(1)}[\mathbf{M},\mathbf{P}]$ that are not incident to $v_1$ or $v_2$ (except if they belong to an edge incident to $v_3^{(1)}$ or $v_{LM+2}^{(1)}$) and that do not belong to a fully paired edge. Together with~\eqref{eq:prop2:chain}, we arrive at 
\begin{align*}
    2(|\mathrm{Cyc}| + |\mathrm{CC}(G_{\Delta})|-1)  &\leq 2|FE_c^{(1)}[\mathbf{M},\mathbf{P}]|+ 2|FE_o^{(1)}[\mathbf{M},\mathbf{P}]|+ q \\ & \quad \quad + 2\left[|\mathbf{M}|-\mathbf{1}\{(v^{(1)}_1,v^{(2)}_1)\notin \mathbf{M}_{\mathrm{full}}\} -\mathbf{1}\{(v^{(1)}_2,v^{(2)}_2) \notin \mathbf{M}_{\mathrm{full}}\}\right] \ .
\end{align*}
This concludes the proof as the total number  $|\mathbf{P}|$ of paired half-edges in $G^{(1)}[\mathbf{M},\mathbf{P}]$ is equal to $|H_1|+|H_2|+ q + 2|FE_c^{(1)}[\mathbf{M},\mathbf{P}]|+ 2|FE_o^{(1)}[\mathbf{M},\mathbf{P}]|$. 

\subsection{Proof of Lemma~\ref{lem:P:petit}}

The quantity $|\mathrm{CC}(G_{\Delta})|+ \mathbf{1}\{v_1\sim_{G_{\Delta}} v_2\} -2$ counts the number of connected components that do not intersect $v_1$ and $v_2$. We decompose this set of connected components in $\mathrm{CC}^*_{1}$ and in  $\mathrm{CC}^*_2$ where $\mathrm{CC}^*_{1}$ is the set of connected components that intersect $G^{(1)}[\mathbf{M},\mathbf{P}]$ but neither intersect  $v_1$ nor $v_2$ and where $\mathrm{CC}^*_2$ is the set of connected components that do not intersect $G^{(1)}[\mathbf{M},\mathbf{P}]$. 

Since $|CC^*_2|\leq |CC_2(G_{\Delta})|$, where $CC_2(G_{\Delta})$ is defined in the proof of Proposition~\ref{prop:template:chaine_rappeur}, it follows from~\eqref{eq:upper_CC_2} that  
\begin{align}\label{eq:upper_CC*2}
|CC^*_2|\leq |FE^{(1)}_o[\mathbf{M},\mathbf{P}|]|/2\enspace .
\end{align}

Define $V^*:=[v_3^{(1)},\ldots, v_{ML+2}^{(1)}]$. Each connected component $C$ of  $\mathrm{CC}^*_{1}$ decomposes into an union of maximum intervals of the form $[v^{(1)}_l, v^{(1)}_{l+b}]$. Since neither $v_{l-1}^{(1)}$, $v_{b+l+1}^{(1)}$, $v_1$, nor $v_2$ belong to $C$, it follows that the edges going from $v_{l}^{(1)}$ to $v^{(1)}_{l-1}$ (or to $v_1$ or $v_2$) and from $v^{(1)}_{l+b}$ to $v^{(1)}_{l+b+1}$ or ($v_1$ and $v_2$) have been pruned in some way when building $G_{\Delta}$. This implies that at least $2$ edges incident to $v_{l-1}^{(1)}$ and $2$ distinct edges incident to $v^{(1)}_{l+b}$ have at least one of their two half-edges that has been paired.  

By summing over all connected components $C$ of $\mathrm{CC}^*_{1}$, it follows that all paired half-edges in $G^{(1)}[\mathbf{M},\mathbf{P}]$ have been  counted at most twice, except for the half-edges incident to $v_1$ or $v_2$, and except for half-edges involved in a fully paired edges which have counted at most once. 
We arrive at  
\[
4 |\mathrm{CC}^*_1|+ 2 |FE^{(1)}[\mathbf{M},\mathbf{P}]|+ |H_1|+|H_2| \leq 2|\mathbf{P}|\enspace . 
\]
Since $|\mathrm{Cyc}|\leq |FE^{(1)}_{c}[\mathbf{M},\mathbf{P}]|$, we deduce from~\eqref{eq:upper_CC*2}  that 
\[
|\mathrm{Cyc}| +  |\mathrm{CC}^*_1|+  |\mathrm{CC}^*_2| \leq |\mathbf{P}|/2 + |FE^{(1)}_{c}[\mathbf{M},\mathbf{P}]|/2 - (|H_1|+|H_2)/4\enspace . 
\]
Since each fully paired edge required $2$ paired half edges, we conclude that 
\[
|\mathrm{Cyc}| + |\mathrm{CC}(G_{\Delta})|+ \mathbf{1}\{v_1\sim G_{\Delta} v_2\} -2 \leq \frac{3}{4}|\mathbf{P}| - \frac{|H_1|+|H_2}{4}\enspace . 
\]

\section{Technical proofs for the lower bound}\label{sec:proof:technical:LB}

\subsection{Proof of Lemma~\ref{lem:connected_components}}

As in the proof of Proposition~\ref{prop:template:chaine_rappeur}, we partition the connected components $\mathrm{CC}(G_{\Delta})$ into several subsets, but the partition is slightly different. Consider $\mathrm{CC}_{\mathrm{matched}}(G_{\Delta})$, $\mathrm{CC}_{mixed}(G_{\Delta})$, and  $\mathrm{CC}_{pure}(G_{\Delta})$ such that $\mathrm{CC}_{\mathrm{matched}}(G_{\Delta})$ is the collection of connected components of $G_{\Delta}$ that intersect a matched node, $\mathrm{CC}_{\mathrm{mixed}}(G_{\Delta})$ is the collection of connected components that do not intersect any matched node and that intersect both $G^{(1)}[\mathbf{M},\mathbf{P}]$ and $G^{(2)}[\mathbf{M},\mathbf{P}]$, whereas $\mathrm{CC}_{pure}(G_{\Delta})$ contains the remaining components, ie the connected components that not intersect matched node and that lie only in $G^{(1)}[\mathbf{M},\mathbf{P}]$ or $G^{(1)}[\mathbf{M},\mathbf{P}]$. Obviously, we have $|\mathrm{CC}_{\mathrm{matched}}(G_{\Delta})|\leq |\mathbf{M}|$ the number of couples of matched nodes. 

Let us turn to a connected component $C\in \mathrm{CC}_{mixed}(G_{\Delta})$. Let us consider $C^{(1)}$ (resp. $C^{(2)}$) the induced subcomponents of $C$ in $G^{(1)}[\mathbf{M},\mathbf{P}]$ (resp. $G^{(2)}[\mathbf{M},\mathbf{P}]$). Since, in $G_{\Delta}$, $C^{(1)}$ and $C^{(2)}$ are connected and since $C$ does not intersect any matched node, in the construction of $G_{\Delta}$, we have added at least an edge from $C^{(1)}$ to $C^{(2)}$ by deletion of an open path. The number of edges between $C^{(1)}$ and $C^{(2)}$ has to be even. Indeed, each node in $G_{\Delta}$ has an even degree. Since the sum of the degrees of the graph induced by $C^{(1)}$ is even (as it is for any graph), this implies that the number of edges between $C^{(1)}$ and $C^{(2)}$ is even. We have proved that at least two open-paths pruned in $G_{\Delta}$ are associated to $C$. Besides, an open path that led to an edge between $C^{(1)}$ and $C^{(2)}$ has to be of odd size. We have proved that $2|\mathrm{CC}_{mixed}(G_{\Delta})|\leq |\mathrm{Op}_{\mathrm{odd}}|$.

Finally, we consider a connected component $C\in \mathrm{CC}_{\mathrm{pure}}(G_{\Delta})$. Without loss of generality, we can assume that $C$ is restricted to nodes in  $G^{(1)}[\mathbf{M},\mathbf{P}]$. Since $\mathbf{M}\in \mathcal{M}^{\star}$,  each connected component of $G^{(1)}[\mathbf{M},\mathbf{P}]$ contains at least a matched node.  This implies that at least one node $v$ in $C$ was connected by an edge in $G^{(1)}[\mathbf{M},\mathbf{P}]$ to a mached node $v'$.  Since this edge has been pruned in the construction in $G_{\Delta}$ (otherwise, we would have  $v\in C$), this means that there is a pruned open path that that lead to connected $C$ to itself. Since this pruned open paths allowed to connect two nodes from $G^{(1)}[\mathbf{M},\mathbf{P}]$, this implies that its length is even as pruning odd lengths open path connects $G^{(1)}[\mathbf{M},\mathbf{P}]$ to $G^{(2)}[\mathbf{M},\mathbf{P}]$. We have proved that $|\mathrm{CC}_{mixed}(G_{\Delta})|\leq |\mathrm{Op}_{\mathrm{even}}|$. This concludes the proof.

\subsection{Proof of Lemma~\ref{lem:pairing}}

We start with $B$. Observe that any cycle of length $l\geq 2$ requires $l$ pairing and that any open path of size $l\geq 1$ also requires $l$ pairings. As a consequence, $2|\mathrm{Cyc}| + 2|\mathrm{Op}_{\mathrm{even}}|+ |\mathrm{Op}_{\mathrm{odd}}|$ is smaller than the total number $|\mathbf{P}|$ of pairings and $B$ is therefore non-negative. In fact, we have the following lower bound which will useful for proving Lemma~\ref{lem:lower_bound_psi}. 
\begin{eqnarray}\label{eq:interpretation_B}
    B\geq \frac{1}{4}|\{\text{ Paired half-edges in Open path of length $>$ 2 or in cycle of length $>$ 2}\}| \enspace . 
\end{eqnarray}
Let us turn to $C$. We shall argue using the half-edges. Recall that  $\frac{1}{2}|E^{\mathrm{half},(1)}|=|E^{(1)}|$ so that 
\[
2C = |E^{\mathrm{half},(1)}| + |E^{\mathrm{half},(2)}| - 2|\mathbf{P}| - 2[2|V^{(1)}| + 2|V^{(2)}|-3|\mathbf{M}|-|\mathbf{M}_{\mathrm{full}}|]\ . 
\] 
Observe that $2|\mathbf{P}|$ is the number of half-edges in $G_{\cup}[\mathbf{M},\mathbf{P}]$ that are involved in a pairing. To control $2C$, we argue by counting the contribution of each node of  $G_{\cup}[\mathbf{M},\mathbf{P}]$. Consider any node $v$ of this graph. First, if $v$ has not been matched, then the corresponding term in $2C$ is 
the number half-edges incident to $v$ -4, which is non-negative, as non-matched node have a degree at least $4$. Second, if $v$ is a matched node and is also fully matched, then its corresponding term in $2C$ is zero, since $|E^{\mathrm{half},(1)}| + |E^{\mathrm{half},(2)}|$ accounts for the number of half-edges incident to $v$ and $2|\mathbf{P}|$ also accounts for the same quantity. Finally, if $v$ is matched but not fully matched, that is it arises in $\mathbf{M}\setminus \mathbf{M}_{\mathrm{full}}$, we know that, at least $2$ half-edges incident to $v$ are not paired --as the number of non-paired half-edges is even. Hence, the corresponding term in $|E^{\mathrm{half},(1)}| + |E^{\mathrm{half},(2)}| - 2|\mathbf{P}|$ is at least $2$. In conclusion, we have proved $C\geq 0$.

\subsection{Proof of Lemma~\ref{lem:counting_isomorphism}}
    Fix  $(\underline{\mathbf{M}}, \underline{\mathbf{P}}, \underline{PE}^{(1)},\underline{PE}^{(2)})\in \underline{\mathcal{M}\mathcal{P}}_{\mathrm{shadow}}$. Consider any two $(\mathbf{M},\mathbf{P})$ and $(\mathbf{M}',\mathbf{P}^{'})$ that have this $(\underline{\mathbf{M}}, \underline{\mathbf{P}}, \underline{PE}^{(1)},\underline{PE}^{(2)})$. 
    Hence, there exist two corresponding bijections $\phi:\underline{PE}^{(1)}\mapsto \underline{PE}^{(2)} $ and $\phi':\underline{PE}^{(1)}\mapsto \underline{PE}^{(2)}$. Then, we can define the bijection $\tilde{\phi}: E^{\mathrm{\mathrm{half}},(1)}\mapsto E^{\mathrm{\mathrm{half}},(1)}$  such that $\tilde{\phi}(\underline{e})= \underline{e}$ if $\underline{e}\notin \underline{PE}^{(1)}$ and $\tilde{\phi}(\underline{e})= \phi^{-1}(\phi'^{-1}(\underline{e}))$ otherwise. One can readily check that $\tilde{\phi}$ is an automorphism of $G^{(1)}$. Indeed, outside of the set of perfectly paired half-edges it is the identity, and inside this set, it corresponds to a matching and pairing that perfectly pairs the corresponding half-edges. 
    Besides any two distinct $\phi'$ lead to distinct automorphisms. Thus, we get
    \[
        \left|\Big\{(\mathbf{M}^{'},\mathbf{P}^{'})\in \mathcal{M}\mathcal{P} :  (\underline{\mathbf{M}}, \underline{\mathbf{P}}, \underline{PE}^{(1)},\underline{PE}^{(2)})\triangleleft  (\mathbf{M}^{'},\mathbf{P}^{'}) \Big\}\right|\leq |\mathrm{Aut}(G^{(1)})| \ . 
    \]
    We conclude by symmetry.

\subsection{Proof of Lemma~\ref{lem:lower_bound_psi}}

First we work with the definition~\eqref{eq:definition:C} of $C$ and the definition of $\psi$ to get 
\begin{eqnarray*}
    \phi(\mathbf{M},\mathbf{P})&=&   \frac{1}{2}\left[|E^{\mathrm{half},(1)}| + |E^{\mathrm{half},(2)}| - 2|\mathbf{P}|\right] + 2|\mathbf{M}|-|\mathbf{M}_{\mathrm{full}}|+ 2B \\
    &\geq & \frac{1}{4}\left[|E^{\mathrm{half},(1)}| + |E^{\mathrm{half},(2)}| - 2|\mathbf{P}|\right]+ \frac{1}{8}\left[ 2(|E^{\mathrm{half},(1)}| + |E^{\mathrm{half},(2)}| - 2|\mathbf{P}|)+16B\right]\\
    & &  :=\frac{1}{4} \phi_1 + \frac{1}{8}\phi_2\ .
\end{eqnarray*}
The quantity $\phi_1$ corresponds to the number of half-edges that are not paired. Hence, it suffices that show that $\phi_2$ is larger or equal to the number of half-edges that are paired but not perfectly paired. By~\eqref{eq:interpretation_B}, $16B$ is larger than the number of paired half-edges that are either involved in a cycle of length larger than $2$ or an open-path of size larger than $2$. Since perfectly paired half-edges correspond to cycle of length exactly $2$, it suffices to prove that $2(|E^{\mathrm{half},(1)}| + |E^{\mathrm{half},(2)}| - 2|\mathbf{P}|)$ is larger than the number of paired-half-edges involved in an open path of length $1$ or $2$. To see this, observe that an open path has a always two extremities that corresponds non-paired half-edges. As a consequence, $2(|E^{\mathrm{half},(1)}| + |E^{\mathrm{half},(2)}| - 2|\mathbf{P}|)$ is larger than $4$ times the number of open paths. Since each open path of length at most two involves at most $4$ paired half-edges, this concludes the proof.

\end{document}